\documentclass[onefignum,onetabnum]{siamart190516}

\usepackage{verbatim}
\usepackage{lipsum}
\usepackage{amsfonts}
\usepackage{graphicx}
\usepackage{subfigure}
\usepackage{epstopdf}
\usepackage{algorithmic,mathrsfs,hyperref}
\ifpdf
  \DeclareGraphicsExtensions{.eps,.pdf,.png,.jpg}
\else
  \DeclareGraphicsExtensions{.eps}
\fi
\renewcommand\theequation{\thesection.\arabic{equation}}
\theoremstyle{plain}

\newtheorem{col}{Corollary}[section]

\newtheorem{remark}{\textbf{Remark}}[section]

\newcommand{\eps}{\epsilon}

\newcommand{\bm}{\boldsymbol}

\newcommand{\Grad}[1]{\nabla #1}

\newcommand{\be}{\begin{equation}}
\newcommand{\ee}{\end{equation}}

\newcommand{\bse}{\begin{subequations}}
\newcommand{\ese}{\end{subequations}}
\def\benl{\begin{eqnarray*}}
\def\eenl{\end{eqnarray*}}

\def\bn{\bm{n}}

\def\be{\bm{e}}

\def\bx{\bm{x}}

\def\bz{\bm{z}}

\def\bmu1{\bm{\mu_1}}

\newcommand{\ben}{\begin{eqnarray}}
\newcommand{\een}{\end{eqnarray}}
\newcommand{\beq}{\begin{equation}}
\newcommand{\eeq}{\end{equation}}
\newcommand{\bea}{\begin{array}}
\newcommand{\eea}{\end{array}}
\newcommand{\bef}{\begin{figure}[H]}
\newcommand{\eef}{\end{figure}}

\newsiamremark{hypothesis}{Hypothesis}
\crefname{hypothesis}{Hypothesis}{Hypotheses}
\newsiamthm{claim}{Claim}

\title{A new  Lagrange multiplier approach for constructing structure preserving schemes, I. positivity preserving\thanks{This work is partially supported by NSF Grant DMS-2012585 and AFOSR Grant FA9550-20-1-0309.}}

\author{Qing Cheng\thanks{Department of Mathematics, Purdue University, West Lafayette, IN 47907, USA  (cheng573$@$purdue.edu, shen7$@$purdue.edu). }
\and  Jie Shen${}^\dag$}

\usepackage{amsopn}

\begin{document}
\bibliographystyle{plain}
\graphicspath{ {Figures_porous/} }
\maketitle

\begin{abstract}
We propose a new  Lagrange multiplier approach to construct positivity preserving schemes for  parabolic type equations.  The new approach  introduces a space-time Lagrange multiplier to enforce the positivity with the  Karush-Kuhn-Tucker (KKT) conditions. We then use a predictor-corrector approach to construct a class of positivity schemes:  with a generic semi-implicit or implicit scheme as the prediction step, and the correction step, which enforces  the positivity,
  can be implemented  with negligible cost. We also present a modification  which allows us to construct schemes which, in addition to positivity preserving, is also mass conserving. This new approach is not restricted to any particular spatial discretization and can be combined with various time discretization schemes.
We establish  stability results for our  first- and second-order schemes under a general setting, and  present ample numerical results to validate the new approach.

\end{abstract}



\section{Introduction}
Solutions for a large class of partial differential equations (PDEs) arising from sciences and engineering applications, e.g., solutions for physical variables such as density, concentration, height, population, etc.,  are required to be positive. It is of critical importance that their numerical approximations preserve the positivity of these variables at the discrete level, as violation of the positivity  may render the  discrete problems ill posed, although the original problems are well posed.

In recent years, a large effort has been devoted to construct  positivity preserving schemes for various problems. 
The existing approaches can be roughly classified into the following categories: 
\begin{itemize}
\item Cut-off approach: an ad-hoc approach which simply cuts off the values outside of the desired range. This approach is perhaps used in many simulations without being explicitly mentioned, and it is recently analyzed in \cite{MR3062022,MR4186541} for certain class of time discretization schemes. The main advantages of the cut-off approach is (i) simple to implement, and (ii) it is able to  preserve the accuracy of the underlying numerical schemes for problems with smooth solutions (cf. \cite{MR4186541}). A disadvantage is that it does not preserve mass. 

\item Discrete maximum principle preserving schemes (see, for instance, \cite{du2020maximum} and the references therein):  these schemes are usually based on second-order finite differences or piecewise linear finite elements so they are limited in accuracy, see however some recent work on fourth-order finite differences \cite{MR3880256,MR4107225} applied to  second-order elliptic or parabolic equations.  

\item Post-processing approach:   sophisticated procedures are designed in  \cite{zhang2011maximum,zhang2010positivity} for  hyperbolic systems: these are explicit schemes which are not quite suitable for parabolic systems.
\item Convex splitting approach: for examples, see \cite{chen2019positivity} for Cahn-Hilliard equations with logarithmic potential, \cite{liu2018positivity,hu2020fully} for PNP and Keller-Segel equations. The drawback of this approach is that  a nonlinear system has to be solved at each time step.
\item Reformulation approach: reformulate the problem so that solution of the corresponding discrete problem  is always positive, see, for instance, \cite{liu2018positivity,HuaS21}).
\item  Lagrange multiplier approach: 
 regarding the positivity as a constraint and introducing a Lagrange multiplier to enforce the positivity constraint. This can be regarded as a special case of the constrainted minimization \cite{hinze2008optimization}. The main difficulty of this approach is that the resulting problem is only semi-smooth  which excludes the use of standard Newton method.  Hence, special, and often expensive,  iterative methods such as  semi-smooth Newton method,  have to be used \cite{qi1997semismooth,van2019positivity,hintermuller2002primal}. 
 \end{itemize}


We consider  in this paper a class of linear or nonlinear parabolic  equations with positive solutions in the following form: 
\begin{equation}\label{strong}
 u_t + \mathcal{L}  u=0,
\end{equation}
with  suitable initial and boundary conditions, where $\mathcal{L}$ could be $\mathcal{L}  u=Au+f(u)$ with $A$ being a linear or nonlinear positive operator and $f(u)$  a semi-linear or quasi-linear operator. 
Consider a  generic spatial discretization of \eqref{strong}:
\begin{equation}\label{appro_strong}
\begin{split}
\partial_t u_h + \mathcal{L}_h  u_h=0,
\end{split}
\end{equation}
where $u_h$ is in certain finite dimensional approximation space $X_h$  and $\mathcal{L}_h$ is a certain approximation of $\mathcal{L}$. In general, $u_h$,  if it exists, may not preserve positivity. Oftentimes, \eqref{appro_strong} may not be well posed if $u_h$ can not preserve positivity, e.g., a direct finite elements or spectral approximation to the  porous media equation \cite{vazquez2007porous} $u_t-m\Grad\cdot( u^{m-1}\Grad  u)=0$ with $(m>1)$ can not preserve positivity so it is not well posed. Hence, special efforts have to be devoted to construct spatial  discretization such that \eqref{appro_strong} is positivity preserving. 

Alternatively, we can introduce  a Lagrange multiplier function $\lambda_h(x,t)$ and solve the following expanded system:
\begin{equation}\label{kkt}
\begin{split}
&\partial_t u_h +  \mathcal{L}_h  u_h=\lambda_h,\\
&\lambda_h \ge 0,\; u_h\ge 0,\;\lambda_h u_h=0.
\end{split}
\end{equation}
Note that the second equation in the above represents the well-known Karush-Kuhn-Tucker (KKT) conditions \cite{ito2008lagrange,facchinei2007finite,harker1990finite,bergounioux1999primal} for constrained minimizations. In the absence of time variable, the problem \eqref{kkt} has been well studied mathematically and numerically. However, how to efficiently solve the time dependent  \eqref{kkt} is a completely different issue which has not received much attention. One can of course use an implicit discretization scheme such as backward-Euler or BDF schemes so that at each time step, the nonlinear system can still be interpreted as a  constrained minimization and apply a suitable iterative procedure. But since these constrained minimization problems are only semi-smooth, a delicate and  costly iterative method has to be used. We refer to \cite{van2019positivity} for such an attempt with a diagonally implicit Runge-Kutta  discretization.

We propose  in this paper   an operator splitting approach  for \eqref{kkt} 
to construct efficient and accurate schemes for \eqref{kkt}.  This approach enjoys the following advantages:
\begin{itemize}
\item It allows us to construct  positivity preserving  schemes for a large class of linear or nonlinear parabolic equations with positive solutions, and the schemes can also be made mass conservative if the PDE is mass conserving; 
 \item It can be applied directly to any finite difference discretization or other spatial discretization with a Lagrangian basis;
 \item It has essentially the same computational cost as the corresponding semi-implicit or implicit scheme with the same spatial discretization;
 \item It has good stability property: the first- and second-order versions of our scheme are  proven  to be unconditionally stable for a large class of problems. 
\end{itemize}
Moreover, we show that schemes based on the ad-hoc  cut-off approach can be interpreted as special cases of our approach. Thus, this approach allows us to construct mass conserving schemes based on the cut-off approach, and our analysis  leads to new stability results for the  cut-off approach.  
We shall apply our new schemes to a variety of  problems with positive solutions, including the challenging porous media equation \cite{vazquez2007porous,liu2011high} and the very challenging Lubrication equation \cite{zhornitskaya1999positivity}.

The rest of the paper is organized as follows.  In Section 2, we introduce the positivity and mass preserving schemes with Lagrange multiplier. In Section 3, we carry out stability analysis for the proposed positivity preserving schemes. In Section 4, we present  numerical results for a variety of  problems to validate  our schemes.  Some concluding remarks are given in Section 5.

\section{Positivity and mass preserving schemes with Lagrange multiplier}
We start with a general description of the spatial discretization, followed by the construction of positivity preserving time discretization schemes without and with mass conservation.
\subsection{Spatial discretization}
We now give a more precise description on the generic spatial discretization in \eqref{kkt}. Let $\Sigma_h$ be a set of mesh points or collocation ponts in $\bar\Omega$. Note that $\Sigma_h$ should not include the points at the part of the boundary where a Dirichlet (or essential) boundary condition is prescribed, while it should include the points at the part of the boundary where a Neumann or mixed  (or non-essential) boundary condition is prescribed.

We consider a Galerkin type discretization with finite-elements or spectral methods or  finite-differences with summation-by-parts in a subspace $X_h\subset X$, and 
 define a discrete inner product, i.e. numerical integration, on  $\Sigma_h=\{\bm{z}\}$ in $\bar \Omega$: 
\begin{equation}\label{numint}
  [u,v]=\sum_{\bm{z}\in \Sigma_h} \omega_{\bm{z}}u(\bm{z})v(\bm{z}), 
\end{equation}
where we require that the weights $\omega_{\bm{z}}>0$. We also denote the induced norm by $\|u\|=[u,u]^{\frac 12}$. For finite element methods, the sum should be understood as  $\sum_{K\subset \mathcal{T}}\sum_{\bm{z}\in Z(K)}$  where $\mathcal{T}$ is a given triangulation. 
And we assume that there is a unique function $\psi_{\bm{z}}(\bm {x})$ in $X_h$ satisfying $\psi_{\bm{z}}(\bm{z'})=\delta_{\bm{z}\bm{z'}}$ for $\bm{z},\bm{z'}\in \Sigma_h$.  
Then,  \eqref{kkt} is interpreted   as follows: Find $u_h\in X_h$ such that
\begin{equation}\label{kkt:v}
\begin{split}
&\partial_t u_h(\bz,t) +  \mathcal{L}_h  u_h(\bz,t)=\lambda_h(\bz,t),\quad\forall {\bm z} \in \Sigma_h,\\
&\lambda_h(\bz,t) \ge 0,\; u_h(\bz,t)\ge 0,\;\lambda_h(\bz,t) u_h(\bz,t)=0,\quad\forall {\bm z} \in \Sigma_h,
\end{split}
\end{equation}
with the Dirichlet boundary condition to be satisfied pointwisely if the original problem includes Dirichlet boundary condition at part or all of boundary.

\subsection{Time discretization}


Let  $\mathcal{L}^n_h$ be an approximate operator of $\mathcal{L}$ at $t_n$. For examples, 
 if $\mathcal{L}u=-\nabla \cdot (f(u)\nabla u)$, $\mathcal{L}^n_h$ could be a lagged linear approximation
\begin{equation}\label{lagged}
 \mathcal{L}^n_h \tilde u^{n+1}_h:=-\nabla \cdot (f(u_h^n)\nabla u_h^{n+1}),
\end{equation}
and for  $\mathcal{L}u=Au+f(u)$,  $\mathcal{L}^n_h$ could be a fully implicit discretization
\begin{equation}\label{impli}
 \mathcal{L}^n_h \tilde u^{n+1}_h:=A\tilde u^{n+1}_h+ f(\tilde u^{n+1}_h),
\end{equation}
 or  an implicit-explicit (IMEX) discretization,
\begin{equation}\label{semil}
 \mathcal{L}^n_h \tilde u^{n+1}_h:=A\tilde u^{n+1}_h+ f(\tilde u^{n}_h),
\end{equation}
 or some other type of discretization such as the convex splitting \cite{EllS93} or the SAV  approach \cite{SXY19,cheng2021generalized}.

\subsubsection{First-order operator splitting scheme}
 Let $\lambda_h^0\equiv 0$, for $n\ge 0$, we proceed as follows.

{\bf Step 1}:  solve $\tilde{ u}_h^{n+1}$ from
\begin{equation}\label{first:por:lag:1}
\frac{\tilde{ u}_h^{n+1}(\bz)- u_h^n(\bz)}{\delta t} + \mathcal{L}^n_h\tilde{ u}_h^{n+1}(\bz)= 0
,\quad\forall {\bm z} \in \Sigma_h;
\end{equation}
{\bf Step 2}: solve $ (u_h^{n+1},\lambda_h^{n+1})$ from
\begin{equation}\label{first:por:lag:2}
\begin{split}
&\frac{ u_h^{n+1}(\bz)-\tilde{ u}_h^{n+1}(\bz)}{\delta t} =\lambda_h^{n+1}(\bz),\quad\forall {\bm z} \in \Sigma_h,\\
&\lambda_h^{n+1}(\bz)\ge 0,\; u_h^{n+1}(\bz)\ge 0,\; \lambda_h^{n+1}(\bz) u_h^{n+1}(\bz)=0,\quad\forall {\bm z} \in \Sigma_h.
\end{split}
\end{equation}
The above scheme can be viewed as an operator splitting method. The first step is just a usual time stepping scheme and can be implemented as usual. However, $\tilde u^{n+1}_h$ may not be positive. In the second step, we use the KKT conditions to enforce the positivity of $u^{n+1}_h$.

 A remarkable property of  \eqref{first:por:lag:2}  is that it can be solved pointwisely as follows: 
\begin{equation}
 ( u_h^{n+1}(\bz),\lambda_h^{n+1}(\bz))=\left\{
\begin{array}{rcl}
(\tilde{ u}_h^{n+1}(\bz),0)       &      & {\mbox{if} \quad 0      <  \tilde{ u}_h^{n+1}(\bz)}\\
(0,-\frac{\tilde{ u}_h^{n+1}(\bz)}{\delta t})    &      & \text{otherwise}\\
\end{array}, \right. \quad\forall {\bm z} \in \Sigma_h.
\end{equation}
\begin{remark}\label{rem:cut}
 The second step in the above scheme is equivalent to the simple cut-off approach \cite{MR3062022,MR4186541}:
 \begin{equation}\label{cutoff}
  u_h^{n+1}(\bz)=\begin{cases} \tilde  u_h^{n+1}(\bz) & \text{ if }\; \tilde  u_h^{n+1}(\bz) >0\\
  0  & \text{ if }\; \tilde  u_h^{n+1}(\bz)\le 0\end{cases},  \quad\forall {\bm z} \in \Sigma_h.
\end{equation}
Hence,  the cut-off approach can also be understood from an operator splitting point of view which opens new avenue for analysis and further algorithm improvement.
\end{remark}

\subsubsection{Higher-order schemes}
We can construct higher-order schemes by using a predictor-corrector approach. More precisely, a $k$th-order IMEX scheme based on BDF and Adam-Bashforth  can be constructed as follows:

\noindent{\bf Step 1} (prediction): solve $\tilde u_h^{n+1}$ from
\begin{align}
\frac{\alpha_k \tilde{u}_h^{n+1}-A_k(u_h^n)}{ \delta t}+\mathcal{L}^n_h \tilde u_h^{n+1}=B_{k-1}(\lambda_h^n);\label{por:lag:1}
\end{align}
{\bf Step 2} (correction): solve $(u_h^{n+1},\lambda_h^{n+1})$ from
\begin{equation}\label{por:lag:2}
\begin{split}
&\frac{\alpha_k( u_h^{n+1}(\bz)-\tilde{ u}_h^{n+1}(\bz))}{\delta t} =\lambda_h^{n+1}(\bz)-B_{k-1}(\lambda_h^n),\quad\forall {\bm z} \in \Sigma_h,\\
&\lambda_h^{n+1}(\bz)\ge 0,\; u_h^{n+1}(\bz)\ge 0,\; \lambda_h^{n+1}(\bz)u_h^{n+1}(\bz)=0,\quad\forall {\bm z} \in \Sigma_h,
\end{split}
\end{equation}
where  $\alpha_k,$ the operators $A_k$ and $B_{k-1}$ $(k=2,3,4)$ are given by:

\noindent {\bf First-order:}
\begin{equation}\label{eq:bdf1}
\alpha_1=1, \quad A_1(u_h^n)=u_h^n,\quad B_0(\lambda_h^n)=0;
\end{equation}
\noindent {\bf Second-order:}
\begin{equation}\label{eq:bdf2}
\alpha_2=\frac{3}{2}, \quad A_2(u_h^n)=2u_h^n-\frac{1}{2}u_h^{n-1},\quad B_1(\lambda_h^n)=\lambda_h^n;
\end{equation}
\noindent {\bf Third-order:}
\begin{equation}\label{eq:bdf3}
\alpha_3=\frac{11}{6}, \quad A_3(u_h^n)=3u_h^n-\frac{3}{2}u_h^{n-1}+\frac{1}{3}u_h^{n-2},\quad B_2(\lambda_h^n)=2\lambda_h^n-\lambda_h^{n-1};
\end{equation}
\noindent {\bf Fourth-order:}
\begin{equation}\label{eq:bdf4}
\alpha_4=\frac{25}{12}, \; A_4(u_h^n)=4u_h^n-3u_h^{n-1}+\frac{4}{3}u_h^{n-2}-\frac{1}{4}u_h^{n-3},\; B_3(\lambda_h^n)=3\lambda_h^n-3\lambda_h^{n-1}+\lambda_h^{n-2}.\end{equation}
The formulae for $k=5,6$ can be derived similarly with Taylor expansions. For the sake of simplicity and  with a slight abuse of notations, we used $A_k(u_h^n)$  and $B_k(u_h^n)$ to denote $A_k(u_h^n,\cdots,u_h^{n-k+1})$ and
$B_k(u_h^n,\cdots,u_h^{n-k+1})$, respectively.
Note that for $k=1$, the above scheme is exactly \eqref{first:por:lag:1}-\eqref{first:por:lag:2}.

 The first-step is a usual $k$th-order IMEX scheme. The second step \eqref{por:lag:2} can be viewed as a correction step in which $\lambda_h^{n+1}$ is introduced to enforce  the pointwise positivity of $u_h^{n+1}$, and can  be efficiently solved as follows:
\begin{equation}
 ( u_h^{n+1}(\bz),\lambda_h^{n+1}(\bz))=\left\{
\begin{array}{rcl}
(\tilde{ u}_h^{n+1}(\bz)-\frac{\delta t}{\alpha_k}B_{k-1}(\lambda_h^n) ,0)          & {\mbox{if} \quad 0      <  \tilde{ u}_h^{n+1}(\bz)}-\frac{\delta t}{\alpha_k}B_{k-1}(\lambda_h^n)\\
(0,B_{k-1}(\lambda_h^n)-\frac{\alpha_k}{\delta t}{\tilde{ u}_h^{n+1}(\bz)})       & \text{otherwise}\\
\end{array}, \right. \;\forall {\bm z} \in \Sigma_h.
\end{equation}

\begin{remark}
 Since $\lambda_h^n$ is an approximation to $\lambda_h$ which tends to zero as $h\rightarrow 0$, an alternative is to replace  $B_{k-1}(\lambda_h^n)$ by zero, i.e., leading to the scheme:
  
 \noindent{\bf Step 1}: solve $ \tilde u_h^{n+1}$ from
\begin{align}\label{por:lag:1b}
\frac{\alpha_k \tilde{u}_h^{n+1}-A_k(u_h^n)}{ \delta t}+\mathcal{L}^n_h \tilde u_h^{n+1}=0;
\end{align}
{\bf Step 2}: solve $ (u_h^{n+1},\lambda_h^{n+1})$ from
\begin{equation}\label{por:lag:2b}
\begin{split}
&\frac{\alpha_k( u_h^{n+1}(\bz)-\tilde{ u}_h^{n+1)}(\bz))}{\delta t} =\lambda_h^{n+1}(\bz),\quad\forall {\bm z} \in \Sigma_h,\\
&\lambda_h^{n+1}(\bz)\ge 0,\; u_h^{n+1}(\bz)\ge 0,\; \lambda_h^{n+1}(\bz) u_h^{n+1}(\bz)=0,\quad\forall {\bm z} \in \Sigma_h.
\end{split}
\end{equation}
Since the second step is once again equivalent to the cut-off approach \eqref{cutoff}, the above scheme can be interpreted as a $k$th-order cut-off  scheme. 
\end{remark}

\subsubsection{Positivity preserving schemes with mass conservation}
A drawback of the scheme \eqref{por:lag:1}-\eqref{por:lag:2} is that it does not preserve mass if the original equation does. For clarity, we consider first the first-order scheme \eqref{first:por:lag:1}-\eqref{first:por:lag:2}.

 Let $<\mathcal{L} \cdot,\cdot>$  (resp. $<\mathcal{L}_h^n \cdot,\cdot>$) denotes the continuous (resp. discrete) bilinear form  after proper integration by parts, e.g., if
$\mathcal{L}^n_h u_h:=-\nabla \cdot (f(u_h^n)\nabla u_h)$, then
$<\mathcal{L}^n_h u_h,v_h>:=[f(u_h^n)\nabla u_h,\nabla v_h]$.
Assuming $<\mathcal{L} u,1>=0$, we find from \eqref{strong} that $\partial_t (u,1)=0$, i.e., the mass is conserved. But assuming  $ <\mathcal{L}^n_h v_h, 1>=0$ for any $v_h\in X_h$, we derive from \eqref{first:por:lag:1}-\eqref{first:por:lag:2} that 
$$[u_h^{n+1},1]-[u_h^{n},1]=\delta t[\lambda_h^{n+1},1].$$
 Since $\lambda_h^{n+1}\ge 0$, we find that the mass is not conserved, in fact it is increasing with $n$.

We present below a simple modification which enables mass conservation. More precisely, we introduce another Lagrange multiplier $\xi^{n+1}_h$, which is independent of spatial variables, to enforce the mass conservation in the correction step. 

{\bf Step 1}:  solve $\tilde{ u}_h^{n+1}$ from
\begin{equation}\label{mass:first:por:lag:1}
\frac{\tilde{ u}_h^{n+1}(\bz)- u_h^n(\bz)}{\delta t} + \mathcal{L}^n_h\tilde{ u}_h^{n+1}(\bz)= 0
,\quad\forall {\bm z} \in \Sigma_h;
\end{equation}

{\bf Step 2}: solve $ (u_h^{n+1},\lambda_h^{n+1})$ from
\begin{subequations}\label{mass:first:por:lag:2}
\begin{align}
&\frac{ u_h^{n+1}(\bz)-\tilde{ u}_h^{n+1}(\bz)}{\delta t} =\lambda_h^{n+1}(\bz)+\xi^{n+1}_h,\quad\forall {\bm z} \in \Sigma_h,\label{mass:first:por:lag:2a}\\
&\lambda_h^{n+1}(\bz)\ge 0,\; u_h^{n+1}(\bz)\ge 0,\; \lambda_h^{n+1}(\bz) u_h^{n+1}(\bz)=0,\quad\forall {\bm z} \in \Sigma_h,\label{mass:first:por:lag:2b}\\
&[ u_h^{n+1},1]=[u_h^{n},1].\label{mass:first:por:lag:2c}
\end{align}
\end{subequations}
In order to solve \eqref{mass:first:por:lag:2}, we rewrite  \eqref{mass:first:por:lag:2a} in the following equivalent form 
\begin{equation}\label{mass:first:por:lag:2d}
\frac{ u_h^{n+1}(\bz)-(\tilde{ u}_h^{n+1}(\bz)+\delta t\xi^{n+1}_h)}{\delta t} =\lambda_h^{n+1}(\bz).
\end{equation}
Hence, assuming $\xi^{n+1}_h$ is known, \eqref{mass:first:por:lag:2d} and \eqref{mass:first:por:lag:2b}  can be solved pointwisely as follows: 
\begin{equation}\label{mass:solution}
 ( u_h^{n+1}(\bz),\lambda_h^{n+1}(\bz))=\left\{
\begin{array}{rcl}
(\tilde{ u}_h^{n+1}(\bz)+\delta t\xi^{n+1}_h,0)       &      & {\mbox{if} \quad 0      <  \tilde{ u}_h^{n+1}(\bz)+\delta t\xi^{n+1}_h}\\
(0,-\frac{\tilde{ u}_h^{n+1}(\bz)+\delta t\xi^{n+1}_h}{\delta t})    &      & \text{otherwise}\\
\end{array}, \right. \quad\forall {\bm z} \in \Sigma_h.
\end{equation}
It remains  to determine $\xi^{n+1}_h$. We find from \eqref{mass:first:por:lag:2c} and \eqref{mass:first:por:lag:2d} that
\begin{equation*}
 [ \tilde{ u}_h^{n+1}+\delta t\xi^{n+1}_h,1] =[u_h^n,1] -\delta t[\lambda_h^{n+1},1],
\end{equation*}
which, thanks to \eqref{mass:solution}, can be rewritten as
\begin{equation*}
\sum_{z\in \Sigma_h\,s.t.\,0<\tilde{ u}_h^{n+1}(\bz)+\delta t\xi^{n+1}_h} (\tilde{ u}_h^{n+1}(\bz)+\delta t\xi^{n+1}_h)\omega_z =[u_h^n,1].
\end{equation*}
Hence, $\xi^{n+1}_h$ is a solution to the nonlinear algebraic equation 
\begin{equation}
F(\xi)=\sum_{z\in \Sigma_h\,s.t.\,0<\tilde{ u}_h^{n+1}(\bz)+\delta t\xi} (\tilde{ u}_h^{n+1}(\bz)+\delta t\xi)\omega_z -[u_h^n,1]=0.
\end{equation}
Since $F'(\xi)$ may not exist and difficult to compute if it exists, instead of the Newton iteration, we can use the following secant method:
\begin{equation}\label{newton}
\xi_{k+1}=\xi_k -\frac{F(\xi_k)(\xi_k-\xi_{k-1})}{F(\xi_k)-F(\xi_{k-1})}.
\end{equation}
Since $\xi^{n+1}_h$ is an approximation to zero and it will be shown below that $\xi^{n+1}_h\le 0$, we can choose  $\xi_0=0$ and $\xi_1=-O(\delta t)$. In all our experiments,  \eqref{newton} converges in a few iterations so that the cost is negligible.

 Once $\xi^{n+1}_h$ is known, we can update $(u_h^{n+1},\lambda^{n+1}_h)$ with \eqref{mass:solution}.

Similarly, the  higher-order scheme \eqref{por:lag:1}-\eqref{por:lag:2} can be modified to preserve mass as follows: 

\noindent{\bf Step 1} (prediction): solve $\tilde u_h^{n+1}$ from
\begin{align}
\frac{\alpha_k \tilde{u}_h^{n+1}-A_k(u_h^n)}{ \delta t}+\mathcal{L}^n_h \tilde u_h^{n+1}=B_{k-1}(\lambda_h^n)+B_{k-1}(\xi_h^n);\label{mass:por:lag:1}
\end{align}
{\bf Step 2} (correction): solve $(u_h^{n+1},\lambda_h^{n+1})$ from
\begin{subequations}\label{mass:por:lag:2}
\begin{align}
&\frac{\alpha_k( u_h^{n+1}(\bz)-\tilde{ u}_h^{n+1}(\bz))}{\delta t} =\lambda_h^{n+1}(\bz)-B_{k-1}(\lambda_h^n)+\xi^{n+1}_h-B_{k-1}(\xi_h^n),\quad\forall {\bm z} \in \Sigma_h,\label{mass:por:lag:2a}\\
&\lambda_h^{n+1}(\bz)\ge 0,\; u_h^{n+1}(\bz)\ge 0,\; \lambda_h^{n+1}(\bz)u_h^{n+1}(\bz)=0,\;\quad\forall {\bm z} \in \Sigma_h,\label{mass:por:lag:2b}\\
&[ u_h^{n+1},1]=[u_h^{n},1].\label{mass:por:lag:2c}
\end{align}
\end{subequations}
In order to solve the above system, we  denote
$\eta_h^{n+1}:=\frac{\delta t}{\alpha_k}(\xi^{n+1}_h-B_{k-1}(\xi_h^n)-B_{k-1}(\lambda_h^n))$ and
 rewrite  \eqref{mass:por:lag:2a} as 
\begin{equation}\label{mass:por:lag:2d}
\frac{\alpha_k( u_h^{n+1}(\bz)-(\tilde{ u}_h^{n+1}(\bz)+\eta^{n+1}_h(\bz)))}{\delta t} =\lambda_h^{n+1}(\bz). 
\end{equation}
Assuming $\xi^{n+1}_h$ is known, we find from \eqref{mass:por:lag:2a} and \eqref{mass:por:lag:2d} that
\begin{equation}
 ( u_h^{n+1}(\bz),\lambda_h^{n+1}(\bz))=\left\{
\begin{array}{rcl}
(\tilde{ u}_h^{n+1}(\bz)+\eta_h^{n+1} ,0)       &      & \text{if} \quad 0      <  \tilde{ u}_h^{n+1}(\bz)+\eta_h^{n+1},\\
(0,-\frac{\alpha_k}{\delta t}(\tilde{u}_h^{n+1}(\bz)+\eta_h^{n+1}(\bz)))    &      & \text{otherwise}.\\
\end{array} \right.
\end{equation}

Finally, we can determine $\xi^{n+1}_h$ by solving the nonlinear algebraic equation
\begin{equation}\label{xi}
F(\xi^{n+1}_h)=\sum_{z\in \Sigma_h\,s.t.\, 0      <\tilde{ u}_h^{n+1}(\bz)+\eta_h^{n+1(\bz)}} (\tilde{ u}_h^{n+1}(\bz)+\eta_h^{n+1}(\bz))\omega_z -[u_h^n,1]=0.
\end{equation}

\begin{remark}
 Replacing $B_{k-1}(\lambda_h^n)$ in \eqref{mass:por:lag:1}-\eqref{mass:por:lag:2} by zero, we obtain a mass conserved $k$th-order cut-off scheme.
\end{remark}

\section{Stability results}
We prove in this section that the first- and second-order positivity preserving schemes with or without mass conservation are dissipative and unconditionally stable if $< \mathcal{L}^n_h { v}_h,{ v}_h>\;\ge 0\;\forall  v_h\in X_h$ for all $n$. 

\subsection{First-order schemes}
We consider first the  scheme \eqref{first:por:lag:1}-\eqref{first:por:lag:2}.
\begin{theorem}\label{firststab}
For the  scheme \eqref{first:por:lag:1}-\eqref{first:por:lag:2}, we have 
\begin{equation}\label{stab:rsults}
\begin{split}
&\| u_h^{m}\|^2
+\sum\limits_{n=0}^{m-1}(\|\tilde{ u}_h^{n+1}- u_h^{n}\|^2+\delta t^2\|\lambda_h^{n+1}\|^2)+2\delta t  \sum\limits_{n=0}^{m-1} < \mathcal{L}^n_h\tilde{ u}_h^{n+1},\tilde{ u}_h^{n+1}>
=\|u_h^0\|^2,\quad\forall m\ge 1.
\end{split}
\end{equation}
In particular, if for all $n$, $< \mathcal{L}^n_h { v}_h,{ v}_h>\;\ge 0\;\forall  v_h\in X_h$, then the scheme  \eqref{first:por:lag:1}-\eqref{first:por:lag:2} with $k=1$ is dissipative and unconditionally stable.
\end{theorem}
\begin{proof}
Taking the discrete inner product of  \eqref{first:por:lag:1} with $2\delta t\tilde{ u}_h^{n+1}$, we obtain
\begin{equation}\label{stab:1a}
\begin{split}
\| \tilde u_h^{n+1}\|^2-\| u_h^n\|^2+\|\tilde u_h^{n+1}- u_h^n\|^2
+2\delta t< \mathcal{L}^n_h\tilde{ u}_h^{n+1},\tilde{ u}_h^{n+1}>
=0.
\end{split}
\end{equation}
We rewrite  \eqref{first:por:lag:2}  as
\begin{equation}\label{inter1}
 u_h^{n+1}(\bz)-\delta t\lambda_h^{n+1}(\bz)=\tilde{ u}_h^{n+1}(\bz).
\end{equation}
Taking the  the discrete inner product of each side of the above equation with itself,  thanks to the last KKT condition in \eqref{first:por:lag:2}, we derive
\begin{equation*}
\| u_h^{n+1}\|^2+\delta t^2\|\lambda_h^{n+1}\|^2 = \|\tilde{ u}_h^{n+1}\|^2.
\end{equation*}
Summing up the above with \eqref{stab:1a}, we obtain
\begin{equation*}
\| u_h^{n+1}\|^2-\| u_h^n\|^2+\delta t^2\|\lambda_h^{n+1}\|^2+\|\tilde{ u}_h^{n+1}- u_h^{n}\|^2
+2\delta t< \mathcal{L}^n_h\tilde{ u}_h^{n+1},\tilde{ u}_h^{n+1}>=0.
\end{equation*}
Summing up the above for $n$ from 0 to $m-1$, we arrive at  the desired result.
\end{proof}

Next, we consider the mass conserved scheme \eqref{mass:first:por:lag:1}-\eqref{mass:first:por:lag:2}.
\begin{theorem}
For the  scheme \eqref{mass:first:por:lag:1}-\eqref{mass:first:por:lag:2}, if $ <\mathcal{L}^n_h v_h, 1>=0$ for any $v_h\in X_h$, we have 
\begin{equation}\label{mass1:stab:results}
\| u_h^{m}\|^2
+\sum\limits_{n=0}^{m-1}(\|\tilde{ u}_h^{n+1}- u_h^{n}\|^2+\delta t^2\|\lambda_h^{n+1}+\xi_h^{n+1}\|^2)+2\delta t  \sum\limits_{n=0}^{m-1} < \mathcal{L}^n_h\tilde{ u}_h^{n+1},\tilde{ u}_h^{n+1}>
\le \|u_h^0\|^2,\quad\forall m\ge 1.
\end{equation}
In particular, if for all $n$, $< \mathcal{L}^n_h{ v}_h,{ v}_h>\;\ge 0\;\forall  v_h\in X_h$, then the scheme  \eqref{mass:first:por:lag:1}-\eqref{mass:first:por:lag:2}  is dissipative and unconditionally stable.
\end{theorem}
\begin{proof}
The proof follows the same procedure  as that of Theorem \ref{firststab}. Indeed, we can replace \eqref{inter1} by 
\begin{equation}\label{mass1:stab1}
 u_h^{n+1}(\bz)-\delta t(\lambda_h^{n+1}(\bz)+\xi_h^{n+1})=\tilde{ u}_h^{n+1}(\bz),
\end{equation}
Taking the inner product of \eqref{mass1:stab1} with itself on both sides, we obtain
\begin{equation}\label{mass1:stab2}
 \|u_h^{n+1}(\bz)\|^2+\delta t^2\|\lambda_h^{n+1}(\bz)+\xi_h^{n+1}\|^2-2\delta t[u_h^{n+1}(\bz),\xi_h^{n+1} ]=\|\tilde{ u}_h^{n+1}(\bz)\|^2.
\end{equation}
Summing up \eqref{mass:first:por:lag:1} and \eqref{mass:first:por:lag:2a}, we obtain
\begin{equation}\label{mass:1:sum}
\frac{ u_h^{n+1}(\bz)- u_h^n(\bz)}{\delta t} + \mathcal{L}^n_h\tilde{ u}_h^{n+1}(\bz)= \lambda_h^{n+1}+\xi_h^{n+1}
,\quad\forall {\bm z} \in \Sigma_h,
\end{equation}
Taking the discrete inner product  of \eqref{mass:1:sum} with $1$ on both sides, using the fact that $ <\mathcal{L}^n_h \tilde u_h^{n+1}, 1>=0$, we obtain
\begin{equation}
[\lambda_h^{n+1},1]+[\xi_h^{n+1},1]=0,
\end{equation}
which implies that $\xi_h^{n+1}=-\frac{[\lambda_h^{n+1},1]}{|\Omega|}\le 0$ since $\lambda_h^{n+1}\ge 0$.
Therefore, 
\begin{equation*}
-2\delta t[u_h^{n+1}(\bz),\xi_h^{n+1} ] \ge 0.
\end{equation*}
Finally, summing up  \eqref{mass1:stab2} with \eqref{stab:1a} and dropping some unnecessary  terms, we   arrive at the desired result.
\end{proof}

\subsection{Second-order schemes}
We first consider the scheme \eqref{por:lag:1}-\eqref{por:lag:2} with $k=2$.
\begin{theorem}\label{secondstab}
For the scheme \eqref{por:lag:1}-\eqref{por:lag:2} with $k=2$, we assume that the first step is computed with the first-order scheme  \eqref{first:por:lag:1}-\eqref{first:por:lag:2}. Then,  we have
\begin{equation*}\label{stab:gron:01}
4\| u_h^{m}\|^2+\frac 43\delta t^2\|\lambda_h^{m}\|^2+
4\delta t \sum\limits_{n=0}^{m-1} < \mathcal{L}_h^n\tilde{ u}_h^{n+1},\tilde{ u}_h^{n+1}>
\le \|2 u_h^1- u_h^{0}\|^2+4\| u_h^0\|^2,\quad\forall m\ge 1.
\end{equation*}
In particular, if for all $n$, $< \mathcal{L}^n_h{ v}_h,{ v}_h>\;\ge 0\;\forall  v_h\in X_h$, then the scheme \eqref{por:lag:1}-\eqref{por:lag:2} with $k=2$ is dissipative and unconditionally stable.
\end{theorem}
\begin{proof}
Taking inner product of equation \eqref{por:lag:1} (with $k=2$)  with $4\delta t\tilde{ u}_h^{n+1}$, we obtain
\begin{equation}\label{stab:1}
\begin{split}
[3\tilde{ u}_h^{n+1}-4 u_h^n+ u_h^{n-1},2\tilde{ u}_h^{n+1}]
+4\delta t< \mathcal{L}_h^n\tilde{ u}_h^{n+1},\tilde{ u}_h^{n+1}>
=4\delta t[\lambda_h^n,\tilde{ u}_h^{n+1}].
\end{split}
\end{equation}
The term  on the left can be written as 
\begin{equation}\label{stab:2}
\begin{split}
&[3\tilde{ u}_h^{n+1}-4 u_h^n+ u_h^{n-1},2\tilde{ u}_h^{n+1}]=
2[3 u_h^{n+1}-4 u_h^n+ u_h^{n-1}, u_h^{n+1}] 
\\&+ 2[3 u^{n+1}_h-4 u_h^n+ u_h^{n-1},\tilde{ u}_h^{n+1}- u_h^{n+1}] + 6[\tilde{ u}_h^{n+1}- u_h^{n+1},\tilde{ u}_h^{n+1}].
\end{split}
\end{equation}
For the first term on the righthand side of  \eqref{stab:2}, we have 
\begin{equation}
\begin{split}
2[3 u_h^{n+1}-4 u_h^n+ u_h^{n-1}, u_h^{n+1}]=&
\| u_h^{n+1}\|^2-\| u_h^n\|^2+\|2 u_h^{n+1}- u_h^n\|^2\\
&-\|2 u_h^n- u_h^{n-1}\|^2+\| u_h^{n+1}-2 u_h^n+ u_h^{n-1}\|^2.
\end{split}
\end{equation}
For the last term in \eqref{stab:2}, we have
\begin{equation}\label{stab:3}
 6[\tilde{ u}_h^{n+1}- u_h^{n+1},\tilde{ u}_h^{n+1}]=3(\|\tilde{ u}_h^{n+1}\|^2-\| u_h^{n+1}\|^2+\|\tilde{ u}_h^{n+1}- u_h^{n+1}\|^2).
\end{equation}
And for the second term on the righthand side of  \eqref{stab:2}. Similarly, we have
\begin{equation}\label{inter3} 
\begin{split}
2[3 u_h^{n+1}-4 u_h^n+ u_h^{n-1},\tilde{ u}_h^{n+1}- u_h^{n+1}]=&
2[ u_h^{n+1}-2 u_h^n+ u_h^{n-1},\tilde{ u}_h^{n+1}- u_h^{n+1}]\\&+
4[ u_h^{n+1}- u_h^n,\tilde{ u}_h^{n+1}- u_h^{n+1}].
\end{split}
\end{equation}
By Cauchy-Schwartz inequality,   the first term on the righthand side of \eqref{inter3} can be bounded by
\begin{equation}\label{stab:cau}
2[ u_h^{n+1}-2 u_h^n+ u_h^{n-1},\tilde{ u}_h^{n+1}- u_h^{n+1}]\leq
\| u_h^{n+1}-2 u_h^n+ u_h^{n-1}\|^2 + \|\tilde{ u}_h^{n+1}- u_h^{n+1}\|^2.
\end{equation}
Thanks to the KKT conditions in \eqref{por:lag:2}, we have   $[u_h^{n+1},\lambda_h^{n+1}]=[ u_h^{n},\lambda_h^{n}]=0$  for all $n$, so for the second term on the righthand side  of \eqref{inter3}, we have
\begin{equation}\label{larger}
\begin{split}
4[ u_h^{n+1}- u_h^n,\tilde{ u}_h^{n+1}- u_h^{n+1}]&=-\frac{8\delta t}{3}[ u_h^{n+1}- u_h^n,\lambda_h^{n+1}-\lambda_h^n]\\&=\frac{8\delta t}{3}\{[ u_h^{n+1},\lambda_h^n]+[u_h^n,\lambda_h^{n+1}]\}\ge 0.
\end{split}
\end{equation}

Next, we rewrite \eqref{por:lag:2} with $k=2$ as
\begin{equation}\label{eq:proj}
3 u_h^{n+1}-2\delta t\lambda_h^{n+1}=3\tilde{ u}_h^{n+1}-2\delta t\lambda_h^n.
\end{equation}
Taking the discrete inner product of each side of the above equation with itself,  since $[\lambda_h^{n+1}, u_h^{n+1}]=0$, we derive
\begin{equation}\label{inter2}
3\| u_h^{n+1}\|^2+\frac 43\delta t^2\|\lambda_h^{n+1}\|^2 = 3\|\tilde{ u}_h^{n+1}\|^2-4\delta t[\tilde{ u}_h^{n+1},\lambda_h^n] + \frac 43\delta t^2\|\lambda_h^n\|^2.
\end{equation}
Now, summing up \eqref{stab:1} with \eqref{inter2}, and using \eqref{stab:2} to \eqref{larger},  after dropping some unnecessary positive terms, we obtain that for $n\ge 1$, 
\begin{equation}\label{stab:final}
\begin{split}
&4(\| u_h^{n+1}\|^2-\| u_h^n\|^2)+\|2 u_h^{n+1}- u_h^n\|^2-\|2 u_h^n- u_h^{n-1}\|^2\\&
+2\|\tilde{ u}_h^{n+1}- u_h^{n+1}\|^2
+\frac 43\delta t^2(\|\lambda_h^{n+1}\|^2-\|\lambda_h^n\|^2)
+4\delta t< \mathcal{L}_h^n\tilde{ u}_h^{n+1},\tilde{ u}_h^{n+1}>\le 0.
\end{split}
\end{equation}
On the other hand, since the first step is computed by using the first-order scheme, we take $n=1$ in \eqref{stab:rsults} to obtain
\begin{equation}\label{firststep}
\begin{split}
&\| u_h^{1}\|^2
+\|\tilde{ u}_h^{1}- u_h^{0}\|^2+2\delta t  < \mathcal{L}_h^0\tilde{ u}_h^{1},\tilde{ u}_h^{1}>
+\delta t^2\|\lambda_h^{1}\|^2=
\|u_h^0\|^2.
\end{split}
\end{equation}
Finally, summing up \eqref{stab:final} from $n=1$ to $n=m-1$ with \eqref{firststep} multiplied by 4, we obtain, after dropping some unnecessary terms, 
\begin{equation}\label{stab:gron:1}
\begin{split}
&4\| u_h^{m}\|^2+\frac 43\delta t^2\|\lambda_h^{m}\|^2+
4\delta t \sum\limits_{n=0}^{m-1} <\mathcal{L}_h^n\tilde{ u}_h^{n+1},\tilde{ u}_h^{n+1}>
\le \|2 u_h^1- u_h^{0}\|^2+4\| u_h^0\|^2,
\end{split}
\end{equation}
which implies the desired result.
\end{proof}

Next, we consider the mass conserved scheme \eqref{mass:por:lag:1}-\eqref{mass:por:lag:2} with $k=2$.
\begin{theorem}\label{mass2:second:stab} 
For the scheme \eqref{mass:por:lag:1}-\eqref{mass:por:lag:2} with $k=2$, we assume that the first step is computed with the first-order scheme \eqref{mass:first:por:lag:1}-\eqref{mass:first:por:lag:2}. Then,  if $ <\mathcal{L}^n_h v_h, 1>=0$ for any $v_h\in X_h$, we have
\begin{equation*}\label{mass2:stab:gron:01}
4\| u_h^{m}\|^2+\frac 43\delta t^2\|\lambda_h^{m}+\xi_h^m\|^2+
4\delta t \sum\limits_{n=0}^{m-1} < \mathcal{L}_h^n\tilde{ u}_h^{n+1},\tilde{ u}_h^{n+1}>
\le \|2 u_h^1- u_h^{0}\|^2+4\| u_h^0\|^2,\quad\forall m\ge 1.
\end{equation*}
In particular, if for all $n$, $< \mathcal{L}^n_h{ v}_h,{ v}_h>\;\ge 0\;\forall  v_h\in X_h$, then the scheme \eqref{mass:por:lag:1}-\eqref{mass:por:lag:2} with $k=2$ is dissipative and unconditionally stable.
\end{theorem}
\begin{proof}
 The proof is again similar to that of Theorem \ref{secondstab} so we just point out the differences below.
 
First, \eqref{stab:1} should be replace by
\begin{equation}\label{stab:1c}
\begin{split}
[3\tilde{ u}_h^{n+1}-4 u_h^n+ u_h^{n-1},2\tilde{ u}_h^{n+1}]
+4\delta t< \mathcal{L}_h^n\tilde{ u}_h^{n+1},\tilde{ u}_h^{n+1}>
=4\delta t[\lambda_h^n+\xi_h^n,\tilde{ u}_h^{n+1}].
\end{split}
\end{equation}
 Then  \eqref{larger} should be replaced by
  \begin{equation}\label{mass2:larger2}
\begin{split}
4[ u_h^{n+1}- u_h^n,\tilde{ u}_h^{n+1}- u_h^{n+1}]&=-\frac{8\delta t}{3}[ u_h^{n+1}- u_h^n,\lambda_h^{n+1}+\xi_h^{n+1}-(\lambda_h^n+\xi_h^n)]\\&=
-\frac{8\delta t}{3}[ u_h^{n+1}- u_h^n,\lambda_h^{n+1}-\lambda_h^n]-\frac{8\delta t}{3}[ u_h^{n+1}- u_h^n,\xi_h^{n+1}-\xi_h^n]
\\&=-\frac{8\delta t}{3}[ u_h^{n+1}- u_h^n,\lambda_h^{n+1}-\lambda_h^n]\\& =\frac{8\delta t}{3}\{[ u_h^{n+1},\lambda_h^n]+[u_h^n,\lambda_h^{n+1}]\}\ge 0,
\end{split}
\end{equation}
where we used the fact that
\begin{equation*}
-\frac{8\delta t}{3}[ u_h^{n+1}- u_h^n,\xi_h^{n+1}-\xi_h^n]=-\frac{8\delta t}{3}(\xi_h^{n+1}-\xi_h^n)\left([ u_h^{n+1},1]- [u_h^n,1]\right)=0.
\end{equation*}
Next,  \eqref{eq:proj} should be replaced by 
\begin{equation}\label{mass2:inter3}
3 u_h^{n+1}(\bz)-2\delta t(\lambda_h^{n+1}(\bz)+\xi_h^{n+1})=3\tilde{ u}_h^{n+1}(\bz)-2\delta t(\lambda_h^{n}(\bz)+\xi_h^{n}).
\end{equation}
Taking the discrete inner product of \eqref{mass2:inter3} with itself on both sides, we obtain
\begin{equation}\label{mass2:inter4}
\begin{split}
&3 \|u_h^{n+1}(\bz)\|^2+\frac 43\delta t^2\|\lambda_h^{n+1}(\bz)+\xi_h^{n+1}\|^2-2\delta t[u_h^{n+1},\lambda_h^{n+1}+\xi_h^{n+1}]\\&=3\|\tilde{ u}_h^{n+1}(\bz)\|^2+\frac 43\delta t^2\|\lambda_h^{n}(\bz)+\xi_h^{n}\|^2-2\delta t[\tilde u_h^{n+1},\lambda_h^n+\xi_h^n].
\end{split}
\end{equation}
Summing up \eqref{mass:por:lag:1} and \eqref{mass:por:lag:2a}, we obtain
\begin{equation}\label{mass:sum}
\frac{\alpha_k u_h^{n+1}-A_k(u_h^n)}{ \delta t}+\mathcal{L}^n_h \tilde u_h^{n+1}=\lambda_h^{n+1}+\xi_h^{n+1}.
\end{equation}
Taking the discrete inner product of \eqref{mass:sum}  with 1 on both sides, using \eqref{mass:por:lag:2c}, we obtain
\begin{equation}
[\lambda_h^{n+1}+\xi_h^{n+1},1]=[\lambda_h^{n+1},1]+[\xi_h^{n+1},1]=0,
\end{equation}
which implies $\xi_h^{n+1} \le 0$ since $\lambda_h^{n+1}\ge 0$.
Therefore,
\begin{equation}\label{mass2:ineq}
-2\delta t[u_h^{n+1},\lambda_h^{n+1}+\xi_h^{n+1}]
=-2\delta t[u_h^{n+1},\xi_h^{n+1}] \ge 0.
\end{equation}
Then, summing up \eqref{stab:1c} with  \eqref{mass2:inter4}, and using \eqref{mass2:larger2} and  \eqref{mass2:ineq}, after dropping some unnecessary terms, we arrive at 
\begin{equation}\label{mass2:stab:finalb}
\begin{split}
&4(\| u_h^{n+1}\|^2-\| u_h^n\|^2)+\|2 u_h^{n+1}- u_h^n\|^2-\|2 u_h^n- u_h^{n-1}\|^2\\&
+2\|\tilde{ u}_h^{n+1}- u_h^{n+1}\|^2
+\frac 43\delta t^2(\|\lambda_h^{n+1}+\xi_h^{n+1}\|^2-\|\lambda_h^n+\xi_h^n\|^2)
+4\delta t< \mathcal{L}_h^n\tilde{ u}_h^{n+1},\tilde{ u}_h^{n+1}>\le 0.
\end{split}
\end{equation}
For the first step, we take $m=1$ in \eqref{mass1:stab:results} to obtain
\begin{equation}\label{firststepb}
\begin{split}
&\| u_h^{1}\|^2
+\|\tilde{ u}_h^{1}- u_h^{0}\|^2+2\delta t  < \mathcal{L}_h^0\tilde{ u}_h^{1},\tilde{ u}_h^{1}>
+\delta t^2\|\lambda_h^{1}+\xi_h^1\|^2\le 
\|u_h^0\|^2.
\end{split}
\end{equation}
Finally, summing up \eqref{mass2:stab:finalb} from $n=1$ to $n=m-1$ with \eqref{firststepb} multiplied by 4, we obtain, after dropping some unnecessary terms, 
\begin{equation*}
4\| u_h^{m}\|^2+\frac 43\delta t^2\|\lambda_h^{m}+\xi_h^m\|^2+
4\delta t \sum\limits_{n=0}^{m-1} <\mathcal{L}_h^n\tilde{ u}_h^{n+1},\tilde{ u}_h^{n+1}>
\le \|2 u_h^1- u_h^{0}\|^2+4\| u_h^0\|^2,
\end{equation*}
which implies the desired result.
\end{proof}

\begin{remark}
The results in the previous theorems are derived for a general approximate operator $\mathcal{L}_h^n$. They imply in particular:
\begin{itemize}
\item If $\mathcal{L}_h^n$ is non-negative, e.g., as in \eqref{lagged} with application to Porous Media equation, then the  first- and second-order  positivity preserving schemes with Lagrange multiplier are  unconditionally energy stable.

\item  If one can  show,  perhaps under certain  condition $\Delta t \le c_0 h^\alpha$  with a semi-implicit discretization (where $\alpha>0$ depending on the problem and discretization), that for the usual schemes,  i.e., by setting $\lambda^n_h\equiv 0$ for all $n$,
we have  
 $$ \delta t  \sum\limits_{n=1}^{m-1} < \mathcal{L}^n_h\tilde{ u}_h^{n+1},\tilde{u}_h^{n+1}>\; \ge \;\beta \delta t  \sum\limits_{n=1}^{m-1} a_h(\tilde{u}_h^{n+1},\tilde{u}_h^{n+1}) -C_1, \;\forall m\le T/{\Delta t} -1,$$
 where $\beta$ is some positive constant in $(0,1]$ and $T$ is the final time, 
 then we derive from the above and \eqref{stab:rsults} that the solutions of the corresponding schemes  \eqref{por:lag:1}-\eqref{por:lag:2} and \eqref{mass:por:lag:1}-\eqref{mass:por:lag:2} with $k=1,2$ are bounded in the sense that  
\begin{equation*}
\| u_h^{m}\|^2+   2\beta \delta t  \sum\limits_{n=1}^{m-1} a_h(\tilde{ u}_h^{n+1},\tilde{ u}_h^{n+1})\le \|u_h^0\|^2+C_1,\;\forall m\le T/{\Delta t} -1.
\end{equation*}
\end{itemize}
\end{remark}
 
\begin{remark}

 We are unable to prove similar results for  the schemes \eqref{por:lag:1}-\eqref{por:lag:2} and  \eqref{mass:por:lag:1}-\eqref{mass:por:lag:2} with $k\ge 3$. The situation is similar to the pressure-correction schemes for the  Navier-Stokes equations \cite{GMS06}.
\end{remark}

\section{Numerical experiments}
In this section, we carry out various numerical experiments to demonstrate the performance of proposed positivity  preserving schemes.  We use spectral Galerkin methods with numerical integration   \cite{shen2011spectral} for all cases, namely, Fourier-spectral method is used for problems with periodic boundary conditions, while a Legendre-spectral method  is used for problems with Dirichlet or Neumann boundary conditions. Note that in general it is much more difficult to preserve positivity with a spectral method than with a lower-order finite element or finite difference method.  Below,  $h=1/N$ where  $N$ is the number of collocation points in each direction.

\subsection{Convergence rate}
We first  test the convergence rates in time for the positivity preserving schemes using the  Allen-Cahn equation \cite{allen1979microscopic}  
\begin{equation}\label{allen:cahn}
u_t-\Delta u +\frac{1}{\eps^2}u(u-1)(u-\frac 12)=0,
\end{equation}
with periodic boundary condition in $\Omega=[0,2\pi)^2$. It is well-known that the solution will remain in $[0,1]$ if the values of the initial condition $u(x,y,0)$ are in $[0,1]$. In particular, it is positivity preserving.

 We choose the following initial condition
\begin{equation}\label{ini:allen}
u(x,y,0)=\frac 12\big(1+\tanh(\frac{1-\sqrt{(x-\pi)^2+(y-\pi)^2}}{\sqrt{2}\eps})\big),
\end{equation}
with $\eps^2=0.001$ and use $32^2$ uniform collocation points in $[0,2\pi)^2$, i.e., $\Sigma_h=\{x_{jk}=(\frac j{2\pi},\frac k{2\pi}): j,k=0,1,\cdots, 31\}$. We note that with this coarse mesh, the usual semi-implicit Fourier-collocation method will produce numerical solutions with negative values, i.e., the spatial discretized problem  \eqref{kkt:v} will lead to non zero $\lambda_h$. The spatial discretized problem is smooth in time so it can be used to  test the convergence rates of the positivity preserving time discretization schemes. On the other hand, the Fourier-spectral method with $32\times 32$ uniform collocation points is enough to provide a reasonable approximation to this problem as shown in Fig.\;\ref{u_allen}.
As a reference solution, we use the numerical solution computed by  the scheme  \eqref{por:lag:1b}-\eqref{por:lag:2b} with $k=2$ and $\delta t=10^{-6}$.

 \begin{figure}[htbp]
\centering
\subfigure[$u_h$ at $ t=0.01$]{
\includegraphics[width=0.30\textwidth,clip==]{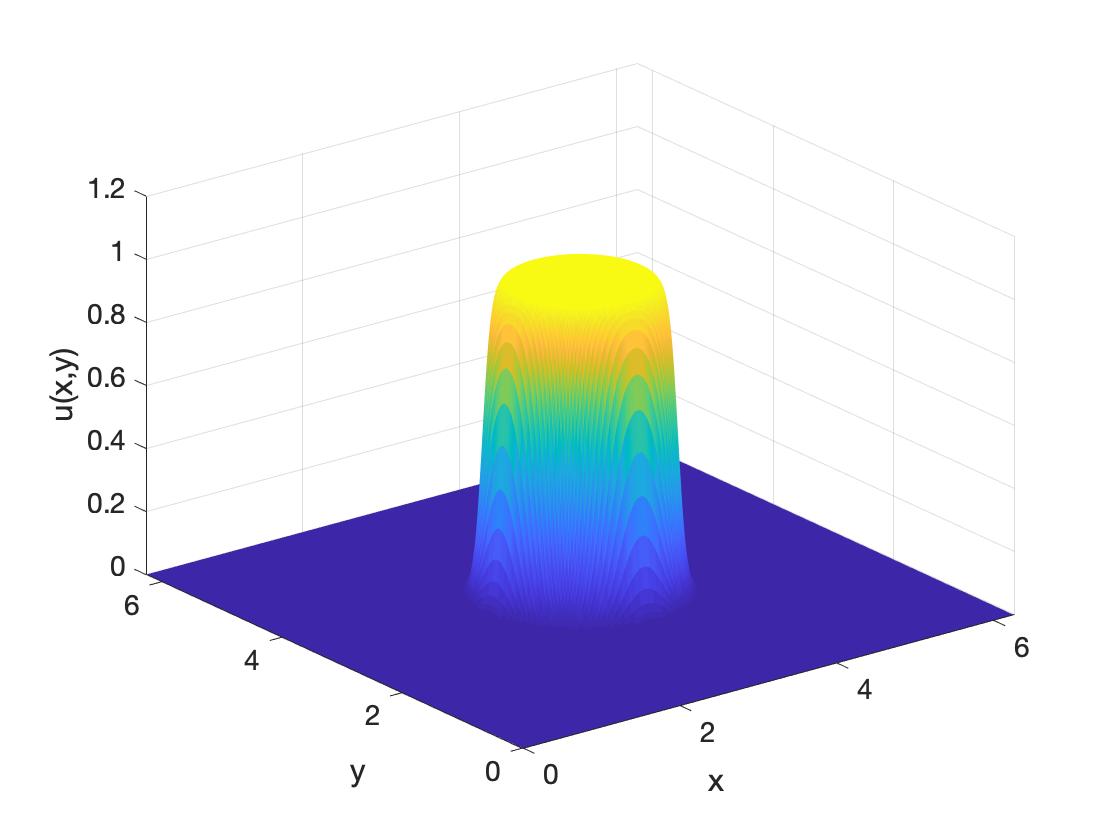}}
\subfigure[$u_h$ at $t=0.4$]{
\includegraphics[width=0.30\textwidth,clip==]{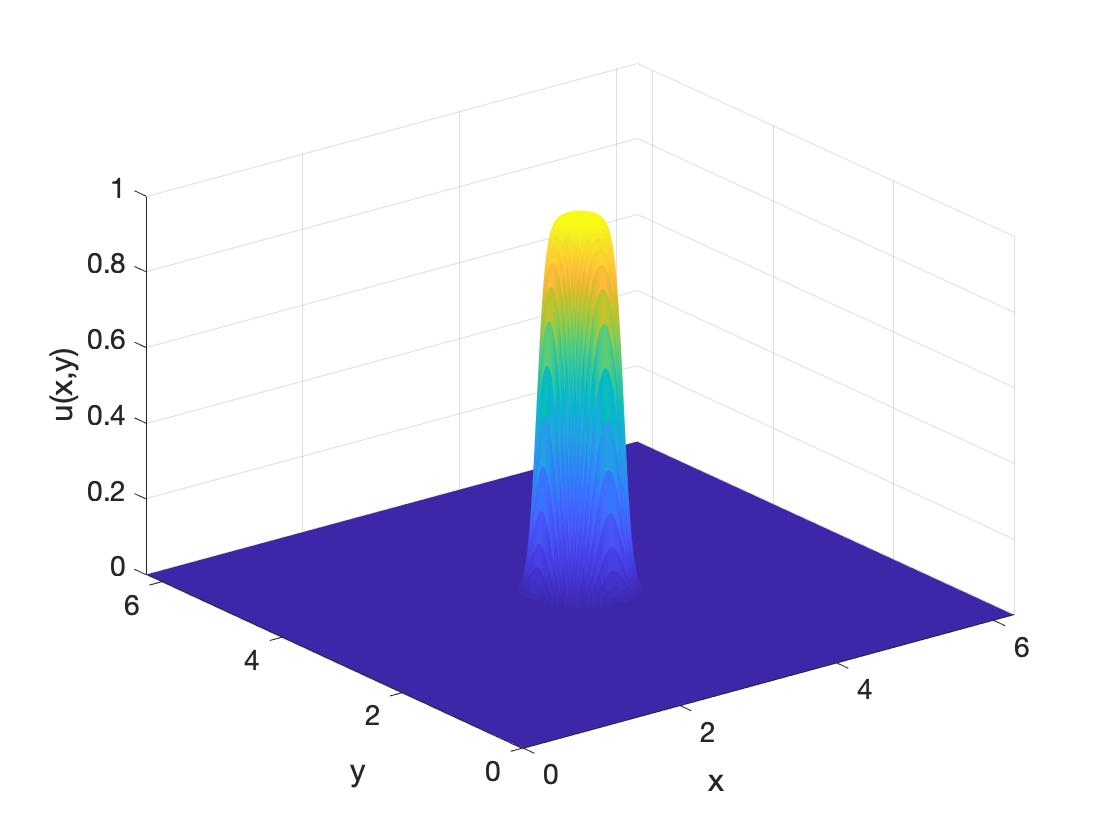}}
\caption{Numerical solution of Allen-Cahn equation \eqref{allen:cahn} with $\eps^2=0.001$  at $t=0.01$ and $t=0.4$ computed with  $32\times 32$ Fourier modes but plotted on the $256\times 256$ grid. }\label{u_allen}
\end{figure}

 In Table \ref{table1}, we list the $L^{\infty}$ errors of numerical solution between the reference solution  $u^{n+1}_h$ obtained using the  schemes \eqref{por:lag:1}-\eqref{por:lag:2}  with $k=1,2$. 
 We observe from Table \ref{table1} that the  schemes \eqref{por:lag:1}-\eqref{por:lag:2} are indeed $k$-th order accurate.

  \begin{table}[ht!]
\centering
\begin{tabular}{r||c|c|c|c}
\hline
$\delta t$      & {\eqref{por:lag:1}-\eqref{por:lag:2}  $k=1$}  & Order          & {\eqref{por:lag:1}-\eqref{por:lag:2} $k=2$}  & Order \\ \hline
$4\times 10^{-5}$  & $2.71E(-4)$ & $-$   &$1.20E(-5)$  & $-$      \\\hline
$2\times 10^{-5}$  & $1.37E(-4)$ & $0.98$     &$2.97E(-6)$ & $2.01$  \\\hline
$1\times 10^{-5}$   & $6.85E(-5)$ & $1.00$   &$7.31E(-7)$ &$2.02$   \\\hline
$5\times 10^{-6}$  & $3.42E(-5)$ & $1.00$  &$1.74E(-7)$ &$2.07$    \\ \hline
$2.5\times 10^{-6}$  & $1.71E(-5)$ & $1.00$  &$3.54E(-8)$&$2.30$  \\\hline
\hline
\end{tabular}
\vskip 0.5cm
\caption{Accuracy test: The $L^{\infty}$ errors between $u^n_h$ and the reference solution at $t=0.01$ for the Allen-Cahn equation \eqref{allen:cahn}  with $\eps^2=0.001$ using \eqref{por:lag:1}-\eqref{por:lag:2}.}\label{table1}
\end{table}

\subsection{Porous medium equation}
\label{sec:main}
In this subsection, we consider the   porous medium equation (PME) (\cite{vazquez2007porous}): 
\begin{eqnarray}
&&  u_t=\Delta  u^m=m\Grad\cdot( u^{m-1}\Grad   u), \label{PME:1}\end{eqnarray}
with homogeneous Dirichlet boundary condition in $\Omega=(-5,5)^d$ $(d=1,2,3)$ where $m \geq 1$ is a physical parameter. The porous medium equation has wide applications in various areas, including fluid dynamics, heater transfer and image processing. We  observe from \eqref{PME:1} that the PME is  degenerate  and its solution has to be positive. 

We shall use the Legendre-Galerkin method with numerical integration in space. Let $P_N$ be the set of polynomials with degree less than or equal to $N$ in each direction, and let $\Sigma_h$ be the set of  the interior Legendre-Gauss-Lobatto points, i.e., in the one dimensional case,  $ \Sigma_h=\{x_k:\;k=1,2,\cdots, N-1\}$ where $\{x_k\}$ are the roots of $L'_N(x)$ with $L_N$ being the Legendre polynomial of $N$-th degree, and in the multi-dimensional case, $\Sigma_h$ is obtained by the tensor product of one-dimensional set. We set $X_h=\{v_h\in P_N: v_h|_{\partial\Omega=0}\}$, and  use the scheme \eqref{por:lag:1}- \eqref{por:lag:2} with $k=2$ and  $ \mathcal{L}^n_h( v_h)=-\Grad\cdot(m( u^{n+1,*})^{m-1}\Grad v_h)$. For the reader's convenience, 
it is explicitly described below:

Find $u_h^{n+1}\in X_h$ such that
\begin{eqnarray}
&&[\frac{3\tilde{ u}_h^{n+1}-4 u_h^n+ u_h^{n-1}}{2\delta t},v_h] +m[( u_h^{n+1,*})^{m-1}\Grad   u_h^{n+1},\Grad  v_h]= [\lambda_h^n,v_h],\quad\forall v_h\in X_h,\label{schem:PME:1}\\
&&\frac{3 u_h^{n+1}(\bz)-3\tilde{ u}_h^{n+1}(\bz)}{2\delta t} =\lambda_h^{n+1}(\bz)-\lambda_h^n(\bz),\quad\forall z\in \Sigma_h,\label{scheme:PME:2}\\
&&\lambda_h^{n+1}(\bz)\ge 0,\; u_h^{n+1}(\bz)\ge 0,\; \lambda_h^{n+1}(\bz) u_h^{n+1}(\bz)=0,\label{scheme:PME:3}
\end{eqnarray}
where 
\begin{equation}\label{ustar}
 u_h^{n+1,*}=\begin{cases} 2u_h^n-u_h^{n-1} \quad if \; u_h^n\ge u_h^{n-1},\\
 \frac{1}{2/u_h^n -1/u_h^{n-1}} \quad if \; u_h^n< u_h^{n-1}.\end{cases}
\end{equation}
At each time step, we need to solve an  elliptic equation with variable coefficients  in  \eqref{schem:PME:1}, which can be  efficiently solved by a preconditioned conjugate gradient iteration with a constant coefficient problem as the preconditioner.

\subsubsection{Comparison with a usual semi-implicit scheme}
We now compare the positivity preserving scheme \eqref{schem:PME:1}-\eqref{scheme:PME:3} with the corresponding usual semi-implicit scheme
\begin{equation}
[\frac{3{ u}_h^{n+1}-4 u_h^n+ u_h^{n-1}}{2\delta t},v_h] +m[( u_h^{n+1,*})^{m-1}\Grad   u_h^{n+1},\Grad  v_h]=0,\quad\forall v_h\in X_h,\label{schem:PME:1b}
\end{equation}
using the  exact solution of the porous medium equation \eqref{PME:1} in  the Barenblatt form
\begin{equation}\label{barenblatt}
 u(x,t)=\frac{1}{t_0^\alpha}\Big(C-\alpha\frac{m-1}{2m}\frac{x^2}{t_0^{2\alpha}}\Big)_{+}^{\frac{1}{m-1}},
\end{equation}
where   $f_{+}=\max\{f,0\}$, $\alpha =\frac{1}{m+1}$, $C=1$ and $t_0=t+1$.  The solution is  compactly  supported in $(0,1)$ with  the interface moving outward in a finite speed. 
 The initial condition for the numerical simulations is chosen as $u(x,0)$.


In Fig,\;\ref{pme:m2}, we  plot the $L^2$ errors by the usual semi-implicit scheme \eqref{schem:PME:1b} and by the positivity preserving \eqref{schem:PME:1}-\eqref{scheme:PME:3} with $m=2$ and set $\delta t=10^{-3}$. We  observe that the errors  grow rapidly  after a short time with $N=128, 256, 512$ using \eqref{schem:PME:1b} since the numerical solution becomes negative at some places; while the error appears to be under control  for at least up to $T=1$ with $N=1024$. On the other hand, by using  \eqref{schem:PME:1}-\eqref{scheme:PME:3},  the errors remain under control and accurate solutions are obtained for all $N$.  We observe from Fig.\;\ref{pme:m2}(c) that, even at $N=1024$, the Lagrange multiplier $\lambda_h$ becomes non-zero in order to maintain positivity.

\begin{figure}[htbp]
\centering
\subfigure[$L^2$ error by  \eqref{schem:PME:1b}.]{
\includegraphics[width=0.45\textwidth,clip==]{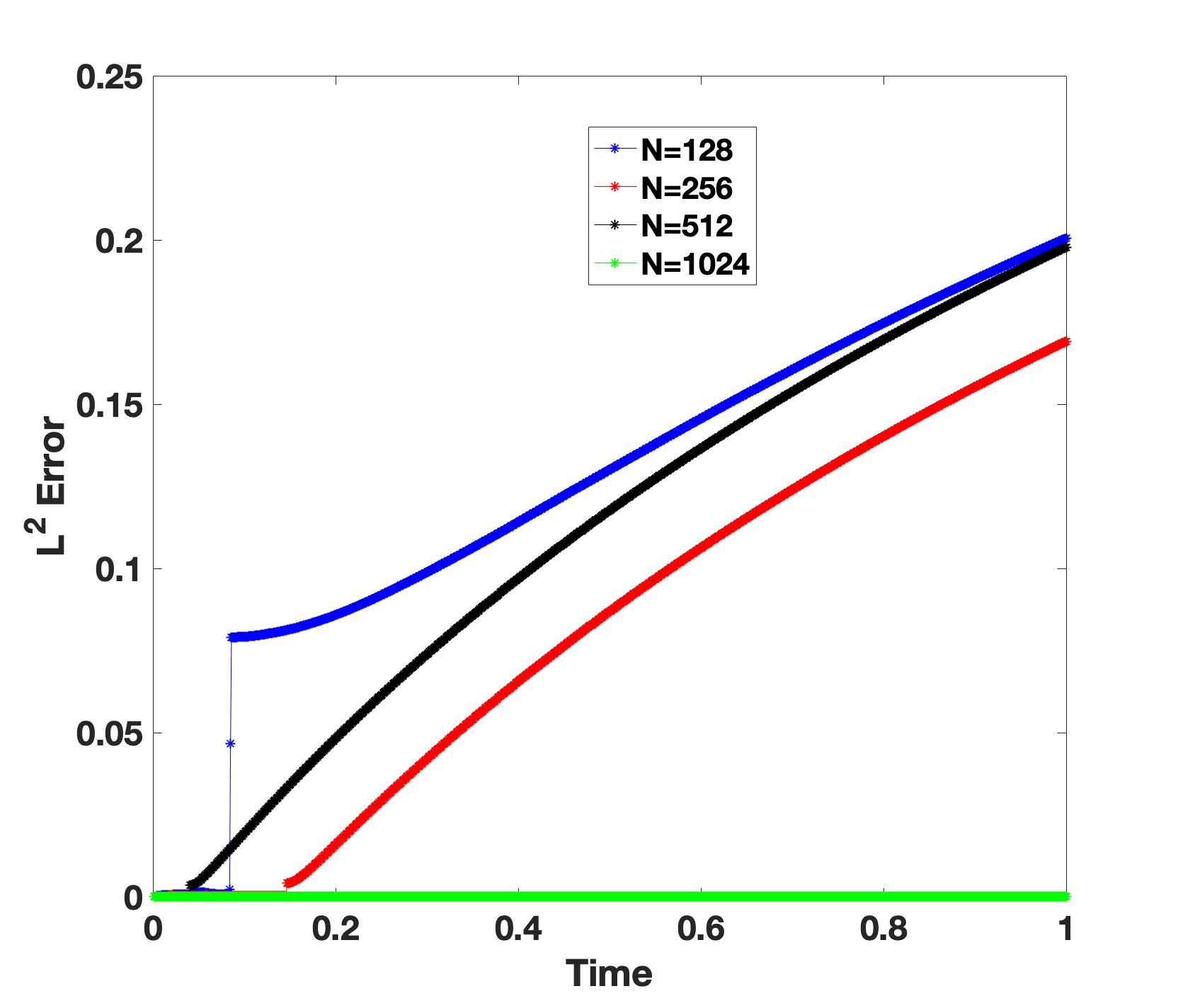}}
\subfigure[$L^2$ error by \eqref{schem:PME:1}-\eqref{scheme:PME:3}.]{
\includegraphics[width=0.45\textwidth,clip==]{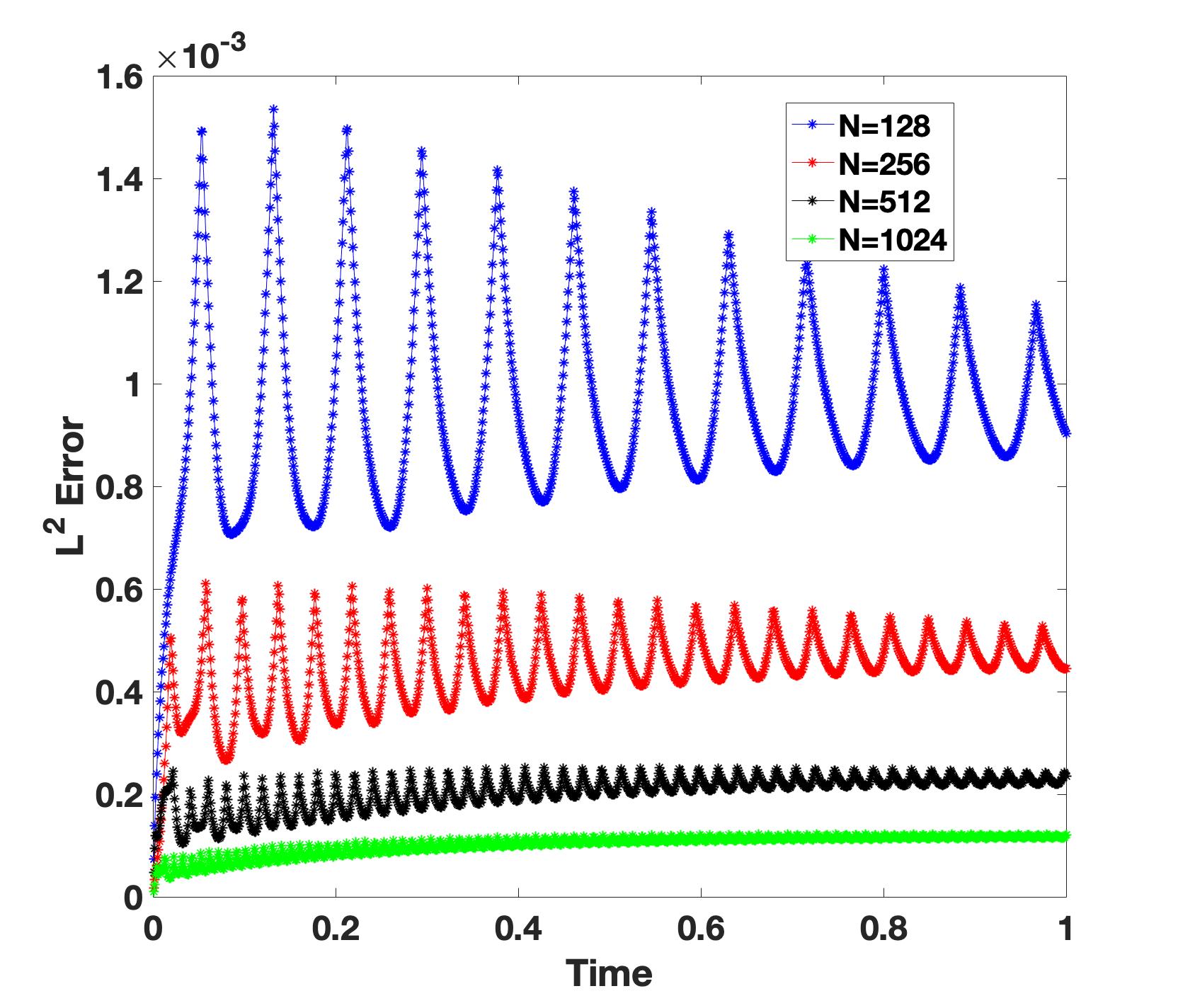}}
\subfigure[$\lambda_h$ by \eqref{schem:PME:1}-\eqref{scheme:PME:3}  with $N=1024$.]{
\includegraphics[width=0.45\textwidth,clip==]{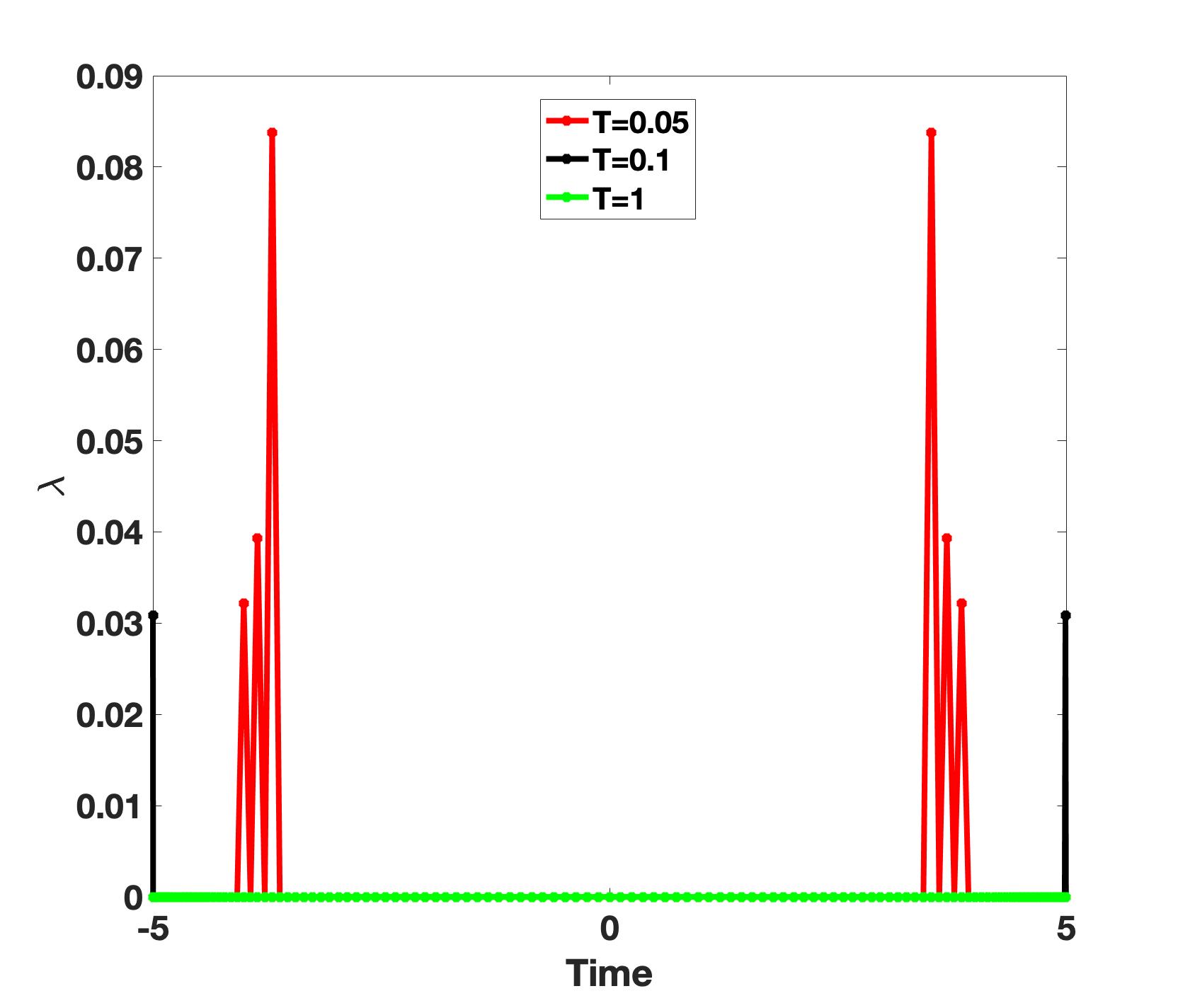}}
\subfigure[$ u_h$ by \eqref{schem:PME:1}-\eqref{scheme:PME:3}  with $N=1024$.]{
\includegraphics[width=0.45\textwidth,clip==]{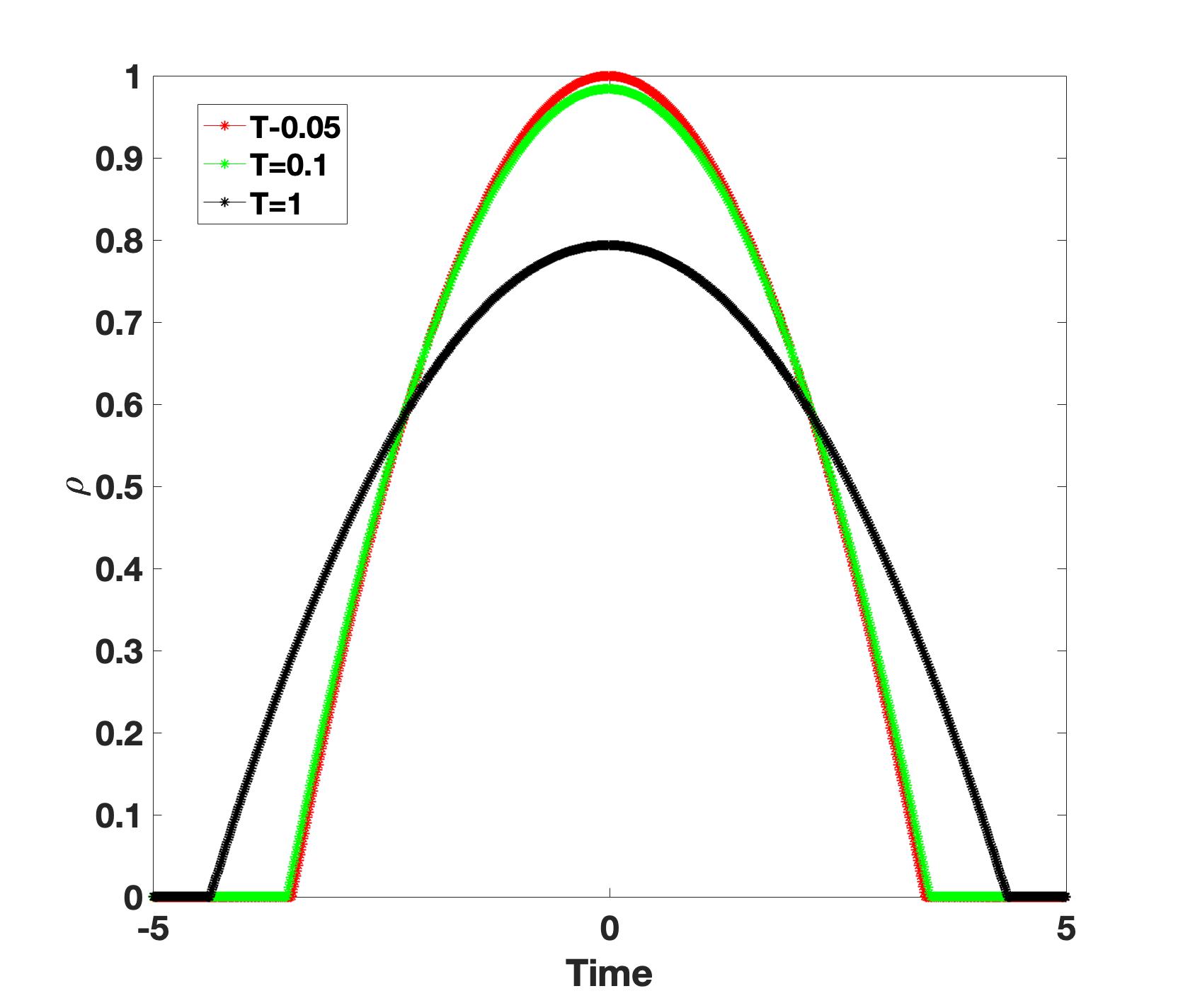}}
\caption{The $L^2$ error of numerical solution  by  \eqref{schem:PME:1b} and by  \eqref{schem:PME:1}-\eqref{scheme:PME:3} with $\delta t=10^{-3}$ and $m=2$.} \label{pme:m2}
\end{figure}

In Fig.\;\ref{pme:f1},  we consider a more challenging case with $m=5$ using  $\delta t=10^{-3}$ and $N=1024$, and plot the numerical solution  at $T=0.1$ using the usual semi-implicit scheme \eqref{schem:PME:1b} and  the positivity preserving \eqref{schem:PME:1}-\eqref{scheme:PME:3}  in Fig.\;\ref{pme:f1}(a) and (b).  We observe that the scheme  \eqref{schem:PME:1b} produces negative values near the interface while the scheme \eqref{schem:PME:1}-\eqref{scheme:PME:3} leads to accurate positive solutions.  We also plot the Lagrange multiplier $\lambda_h$ in Fig.\;\ref{pme:f1}(c) which indicates that $\lambda_h$  becomes larger near the  interface to maintain the positivity of $u_h$.

\begin{figure}[htbp]
\centering
\subfigure[$u_h$ by \eqref{schem:PME:1b}.]{
\includegraphics[width=0.45\textwidth,clip==]{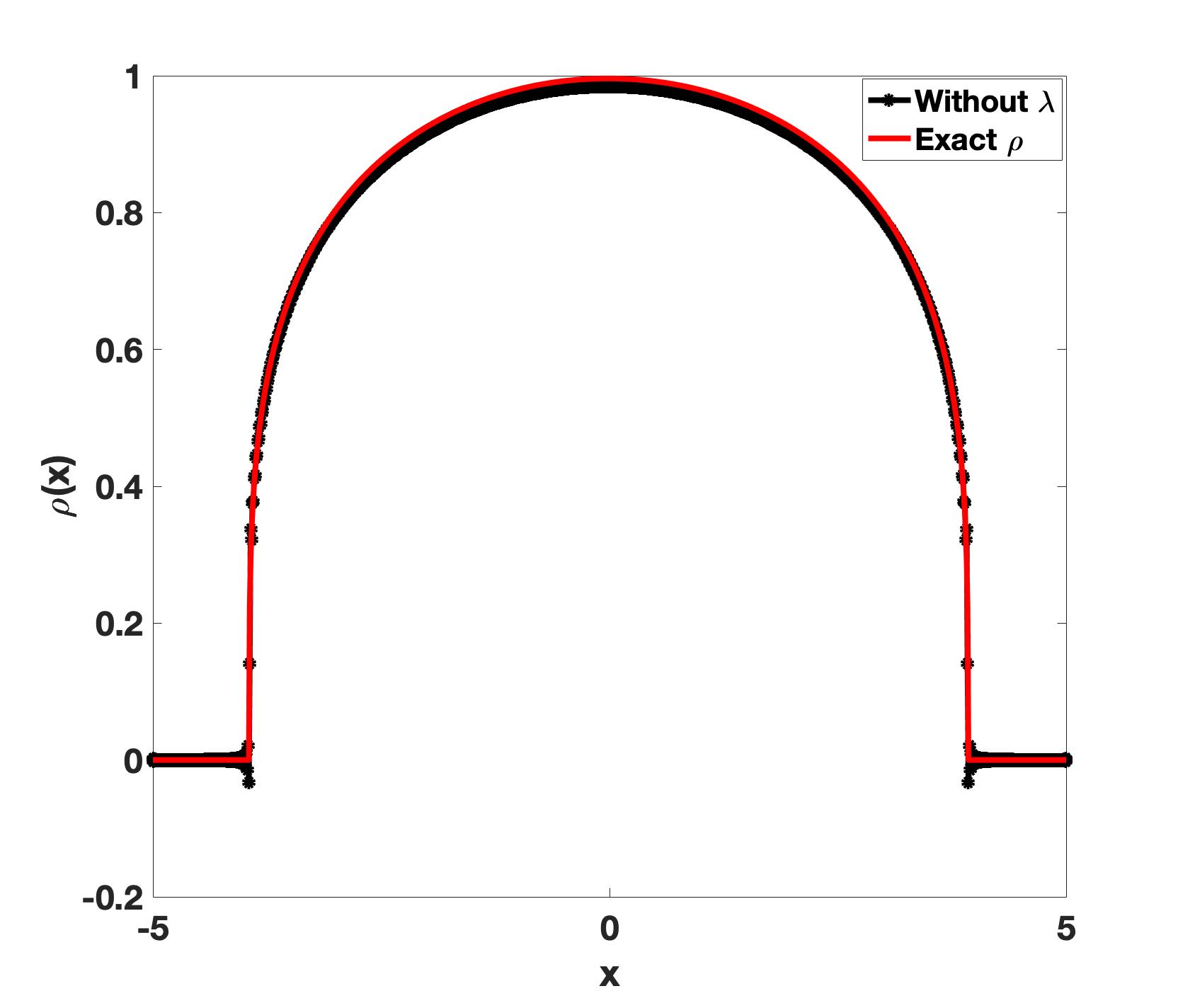}}
\subfigure[$u_h$ by \eqref{schem:PME:1}-\eqref{scheme:PME:3}.]{
\includegraphics[width=0.45\textwidth,clip==]{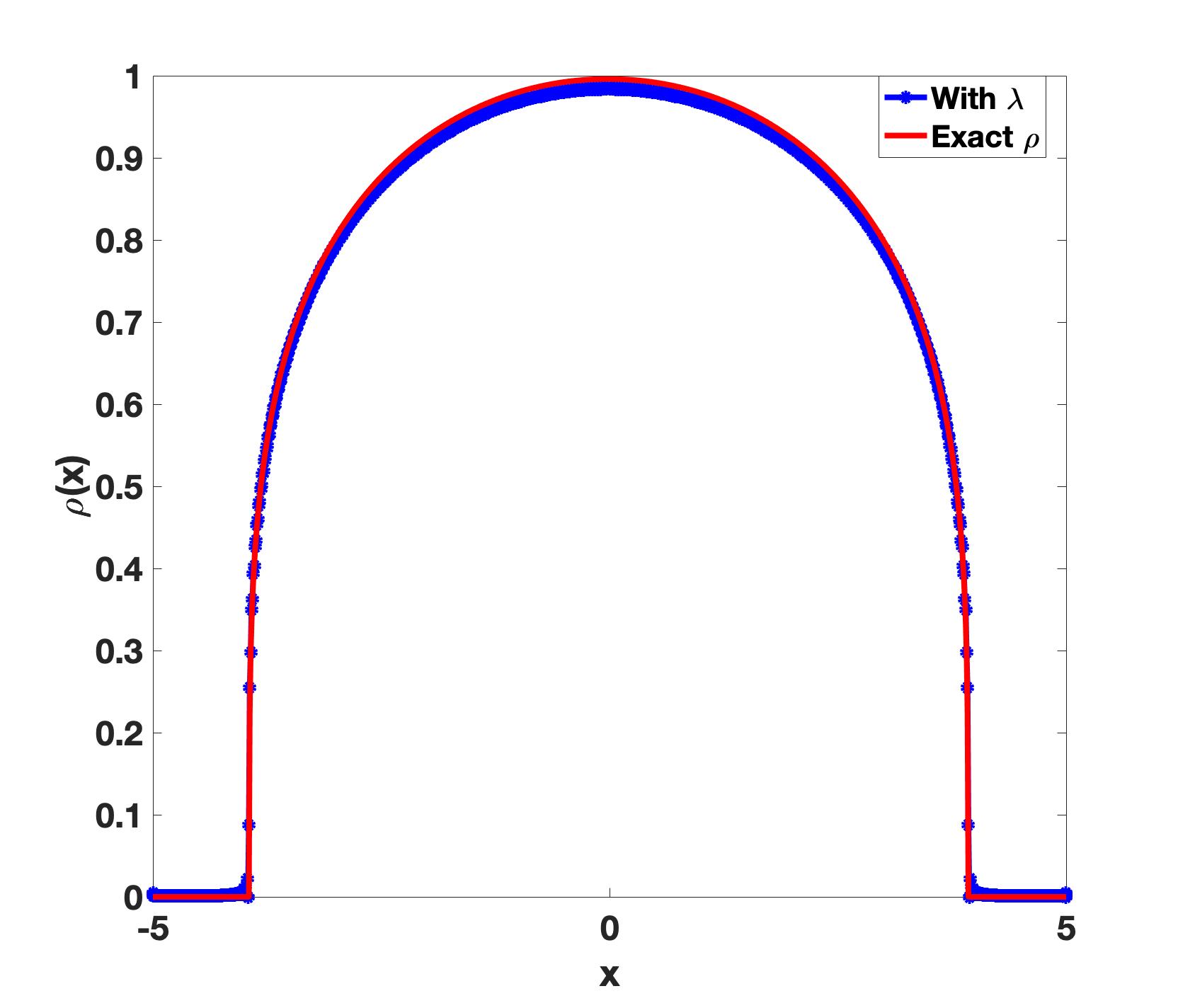}}
\subfigure[Lagrange multiplier $\lambda_h$.]{
\includegraphics[width=0.45\textwidth,clip==]{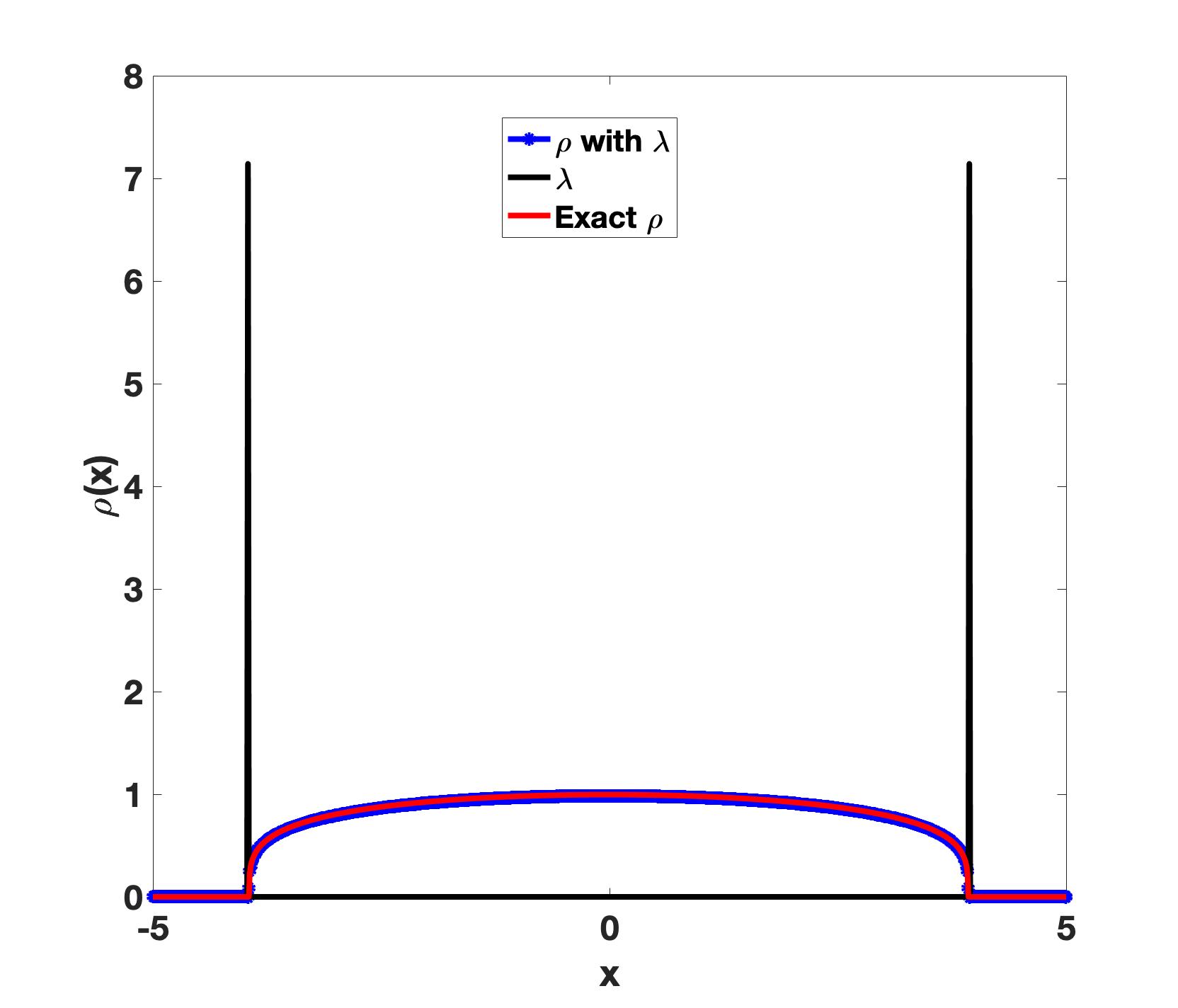}}
\caption{(a) and (b) Numerical solutions $u_h$ at $T=0.1$ with $m=5,\, \delta t=10^{-3},\, N=1024$ by \eqref{schem:PME:1b} and by \eqref{schem:PME:1}-\eqref{scheme:PME:3}. (c) Lagrange multiplier $\lambda_h$ and $ u_h$  by \eqref{schem:PME:1}-\eqref{scheme:PME:3}.}\label{pme:f1}
\end{figure}

Next we consider the 2D case with the exact solution  in the Barenblatt form 
\begin{equation}
 u(x,y,t)|_{t=0}=\frac{1}{t_0^\alpha}\Big(C-\alpha\frac{m-1}{2m}\frac{x^2+y^2}{t_0^{2\alpha}}\Big)_{+}^{\frac{1}{m-1}},
\end{equation}
 where $C=1$, $t_0=t+1$ and $\alpha=\frac{1}{m+1}$. 
 We set $N=200$,  $\delta t=2\times 10^{-4}$ and consider $m=2,5$. We observe from  Fig.\;\ref{PME:2d} that  correct solutions are obtained by the positivity preserving scheme and that the values of Lagrange multiplier $u_h$ are quite large near the interface in order to maintain the positivity of $u_h$. 

\begin{figure}[htbp]
\centering
\subfigure[m=2]{
\includegraphics[width=0.45\textwidth,clip==]{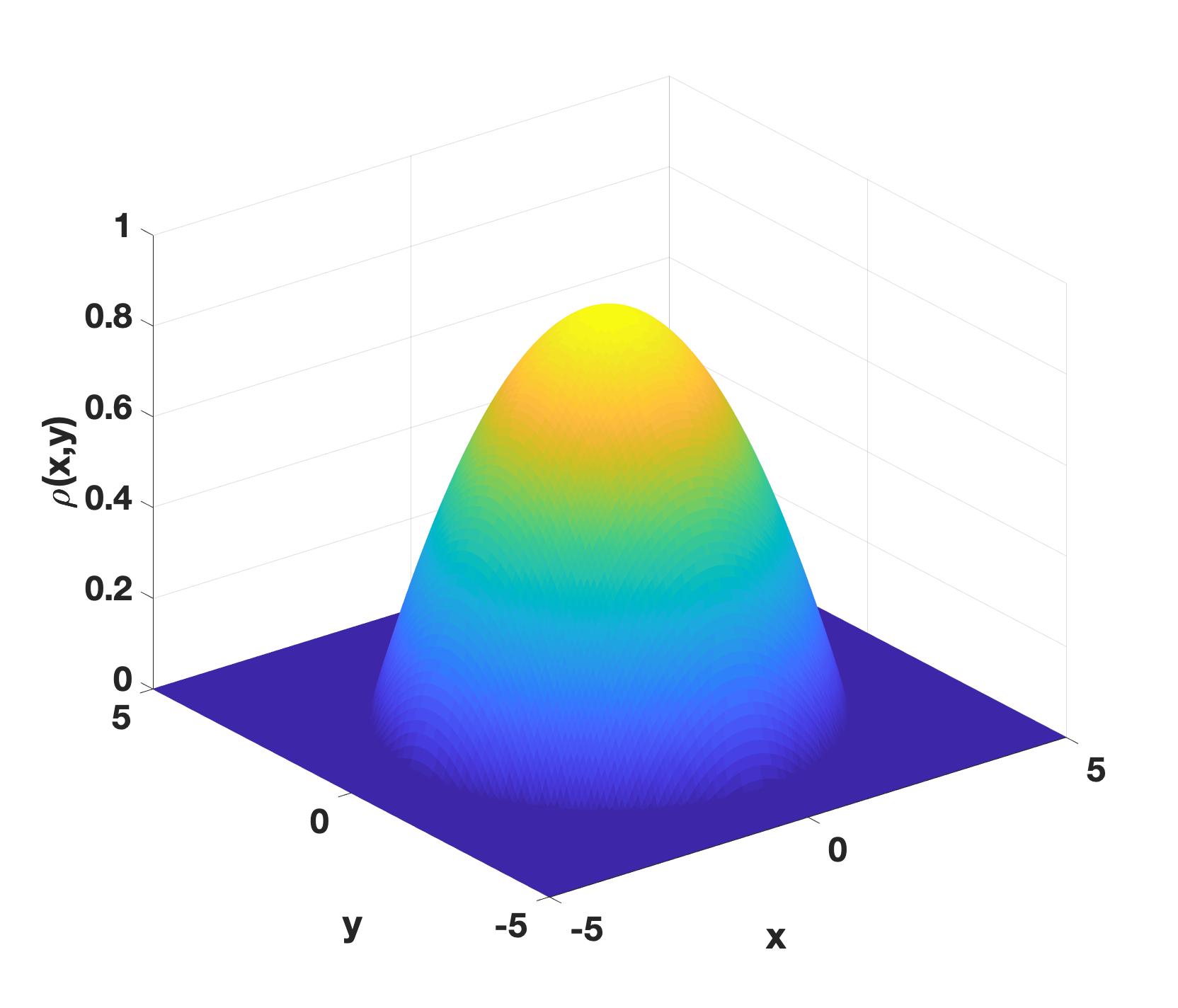}}
\subfigure[m=2]{
\includegraphics[width=0.45\textwidth,clip==]{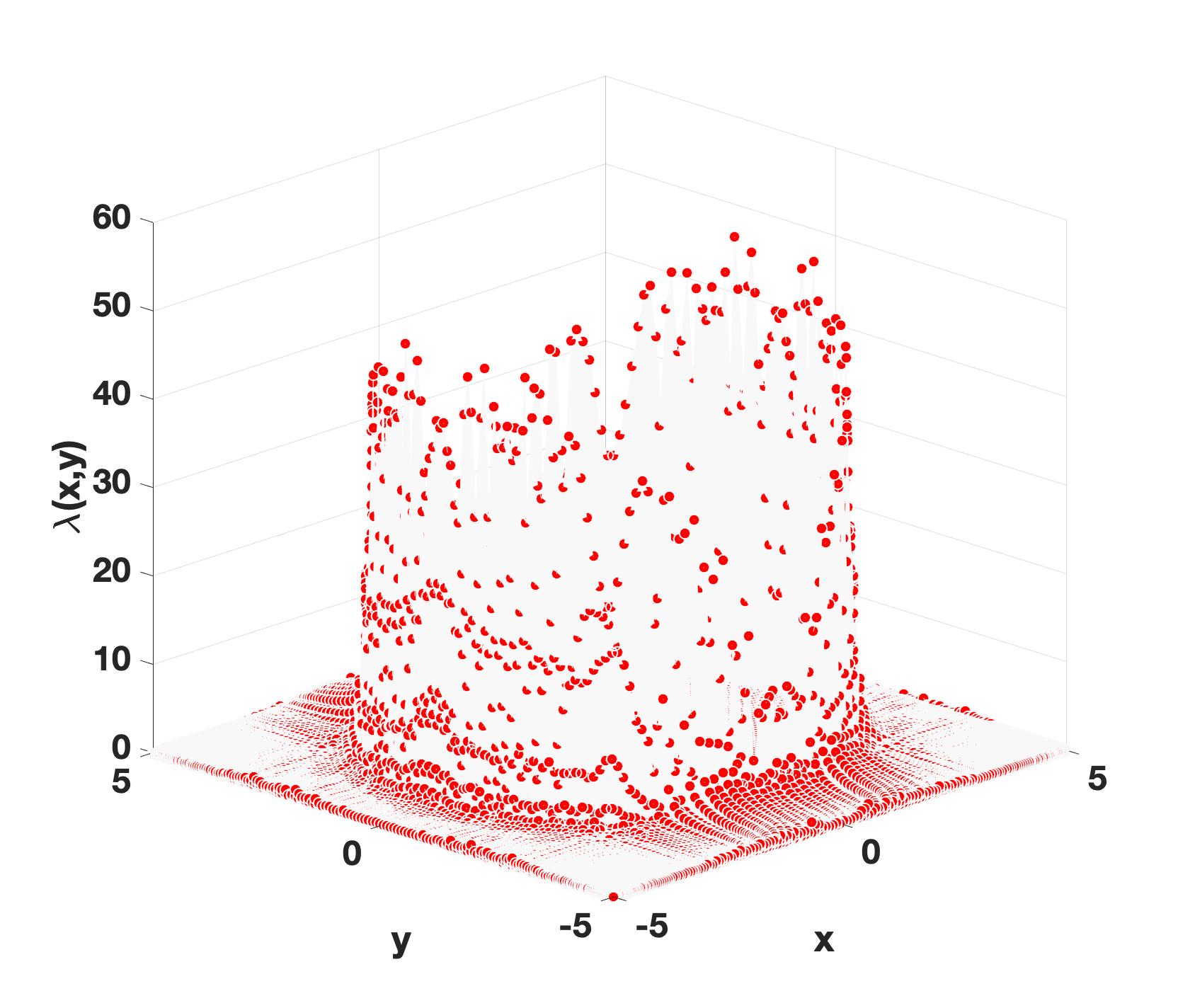}}
\subfigure[m=5]{
\includegraphics[width=0.45\textwidth,clip==]{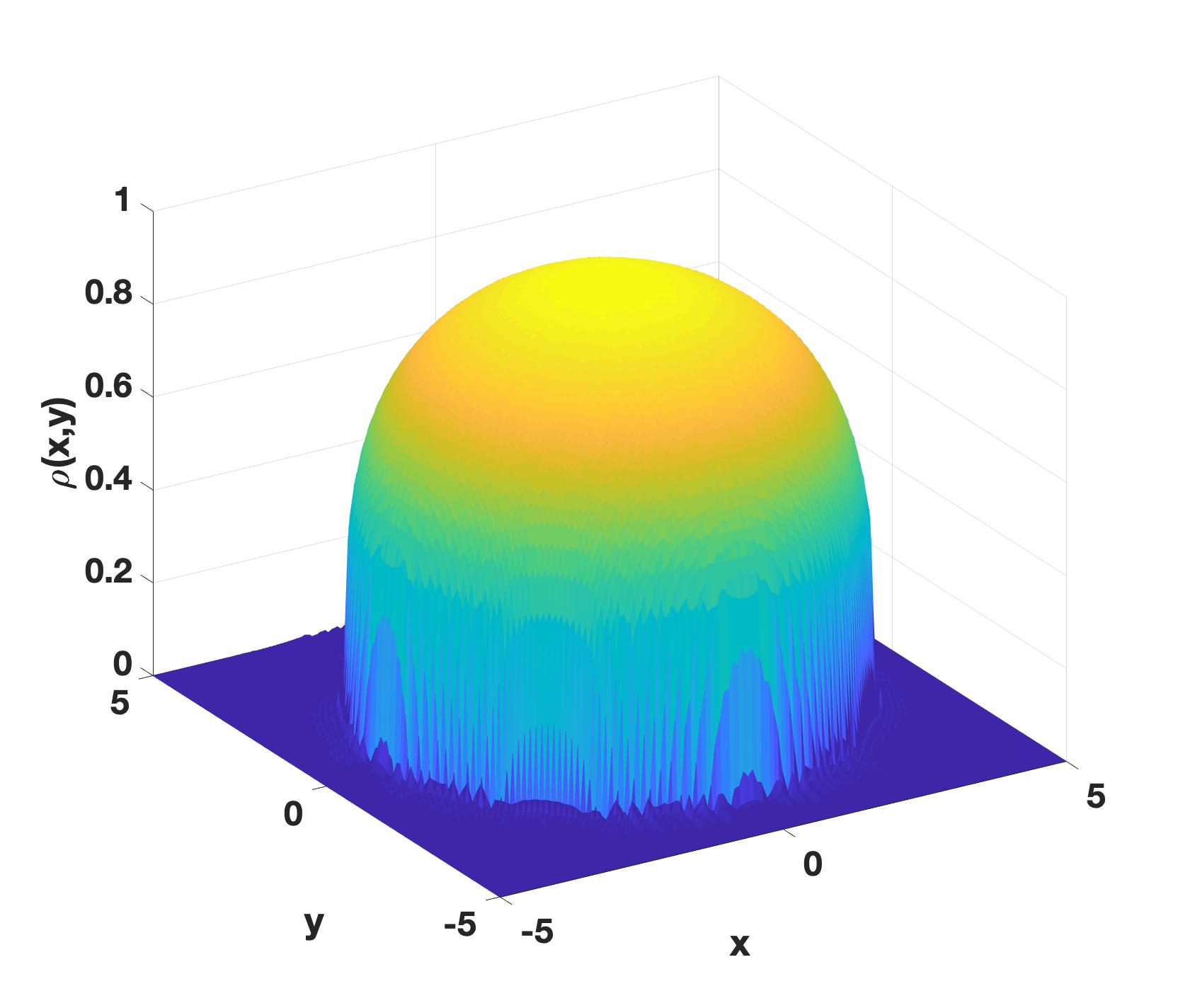}}
\subfigure[m=5]{
\includegraphics[width=0.45\textwidth,clip==]{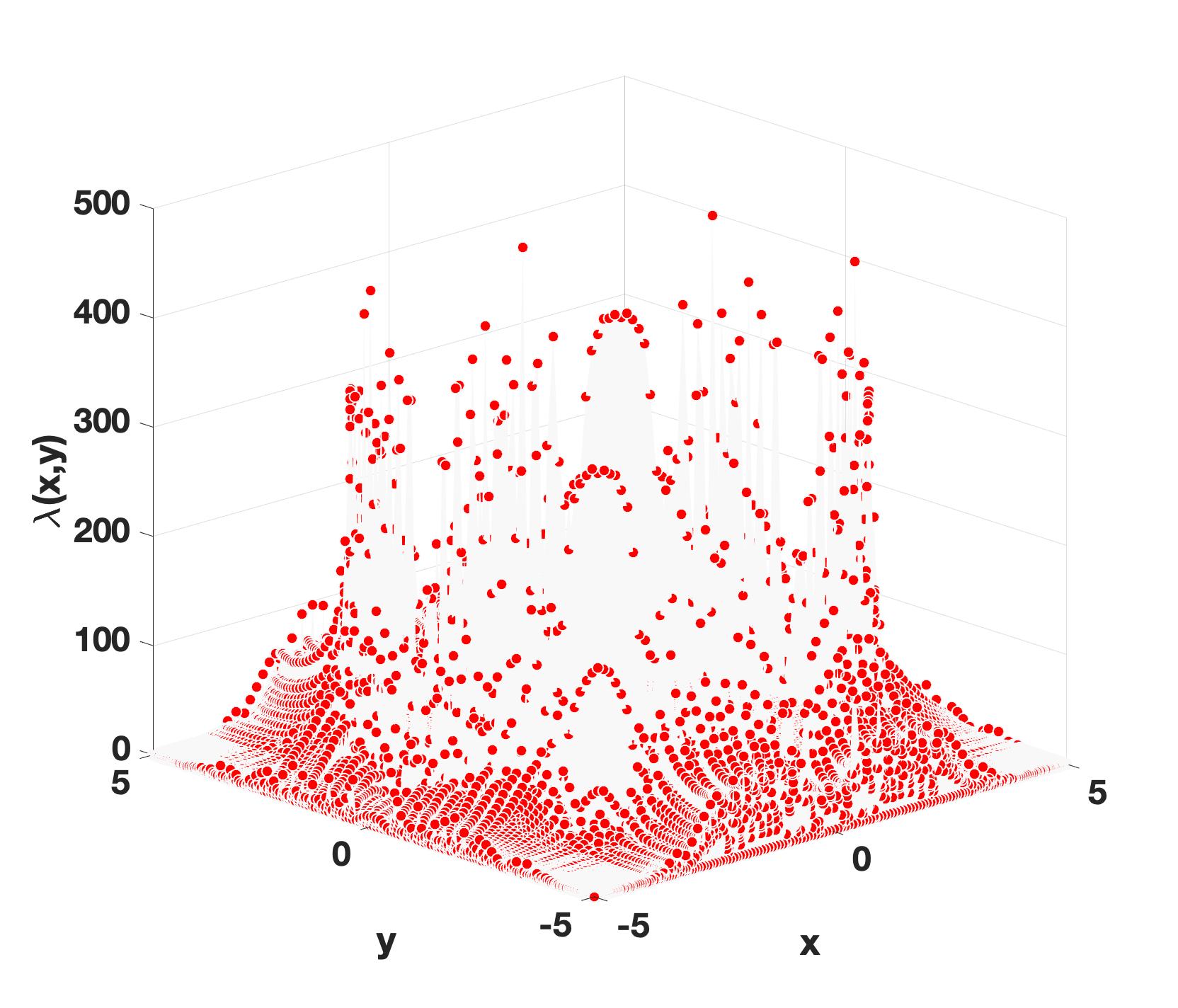}}
\caption{Numerical solution of 2D porous medium equation at $T=0.2$  with  $\delta t=2\times 10^{-4}$ and $N=200$: (a) $u_h$ with $m=2$. (b)  Lagrange multiplier $\lambda_h$ with $m=2$. (c) $ u_h$ with $m=5$. (d) Lagrange multiplier $\lambda_h$ with $m=5$. }\label{PME:2d}
\end{figure}

\subsubsection{Effect of mass conservation}
The porous media equation \eqref{PME:1} with homogeneous Dirichlet boundary conditions is mass conserving. So we  compare the second-order positivity conserving schemes without mass conservation and with mass conservation for the porous medium equation. The results with $\delta t=10^{-4}$ and $N=128$ are plotted  in Fig.\;\ref{mass_error_compare}. We observe that the scheme with mass conservation preserves the mass and is slightly more accurate  in terms of $L^2$ error than the scheme without mass conservation whose mass is monotonically increasing. Only two iterations are needed at each time step to solve $\xi$ using secant method  from Fig.\;\ref{mass_error_compare}. We can also observe Lagrange multiplier $\xi \le 0$ in  time interval $[0, 2]$. 

\begin{figure}[htbp]
\centering
\subfigure[$L^2$ error]{
\includegraphics[width=0.45\textwidth,clip==]{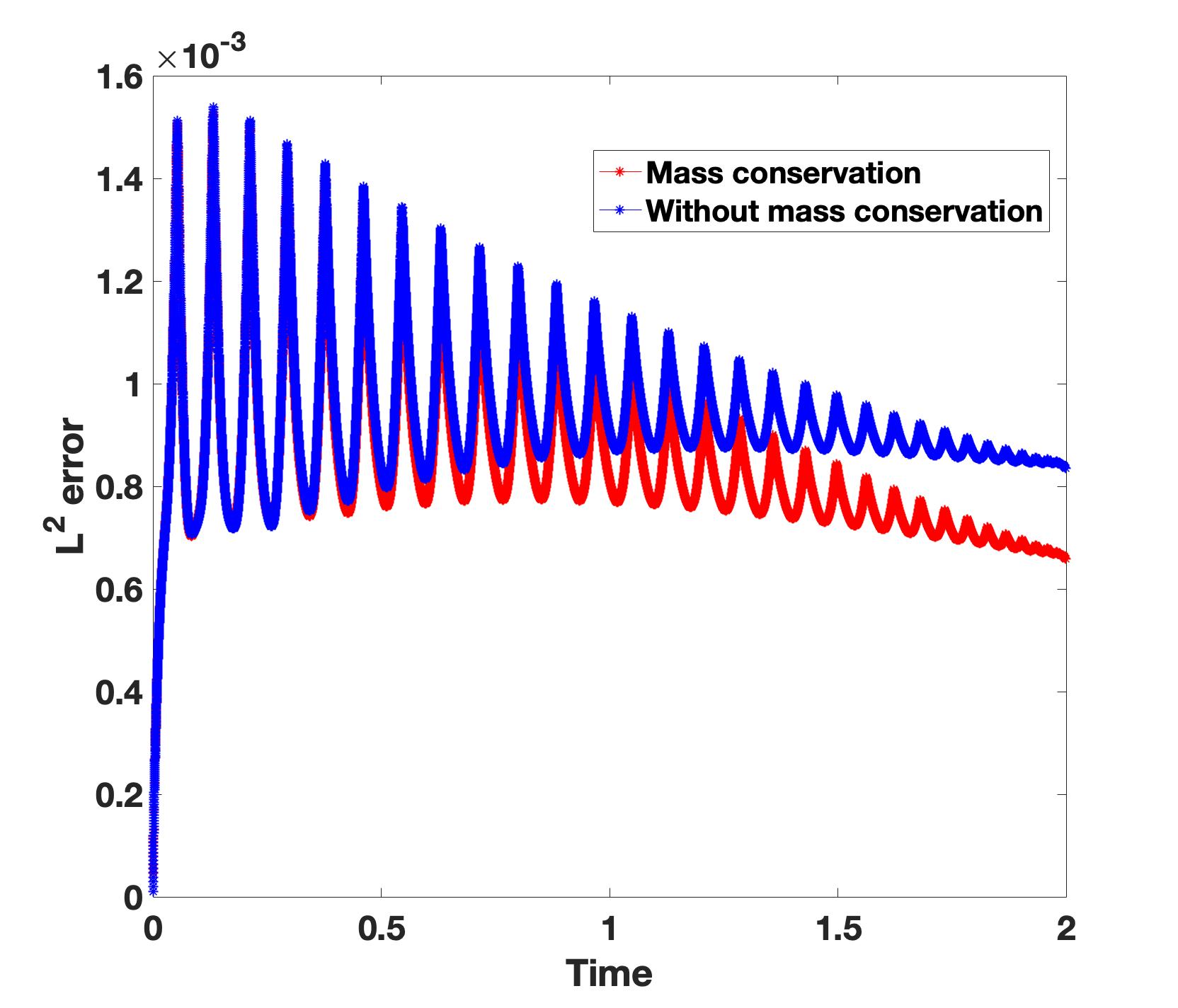}}
\subfigure[Evolution of mass]{
\includegraphics[width=0.45\textwidth,clip==]{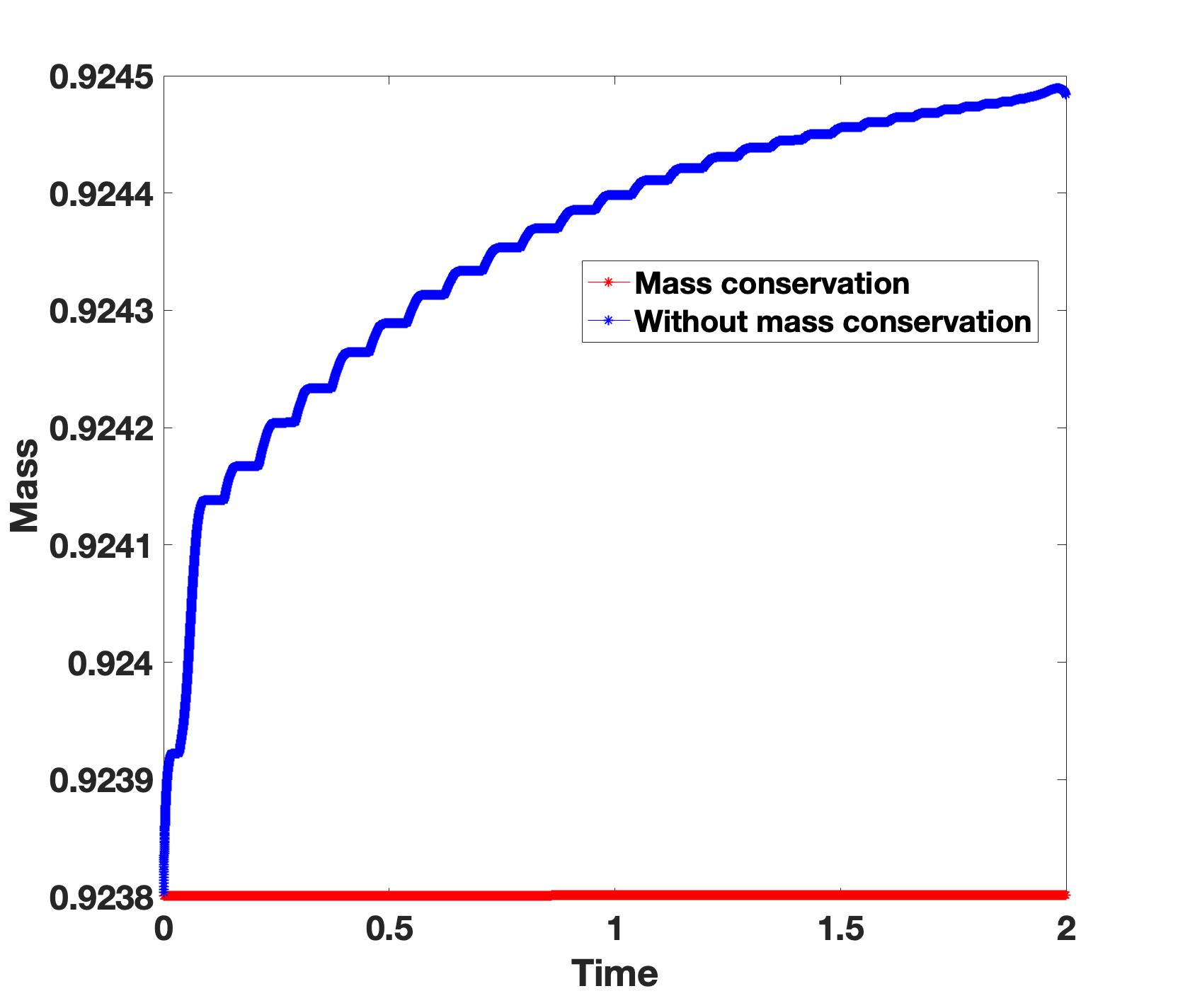}}
\subfigure[$\xi$ ]{
\includegraphics[width=0.45\textwidth,clip==]{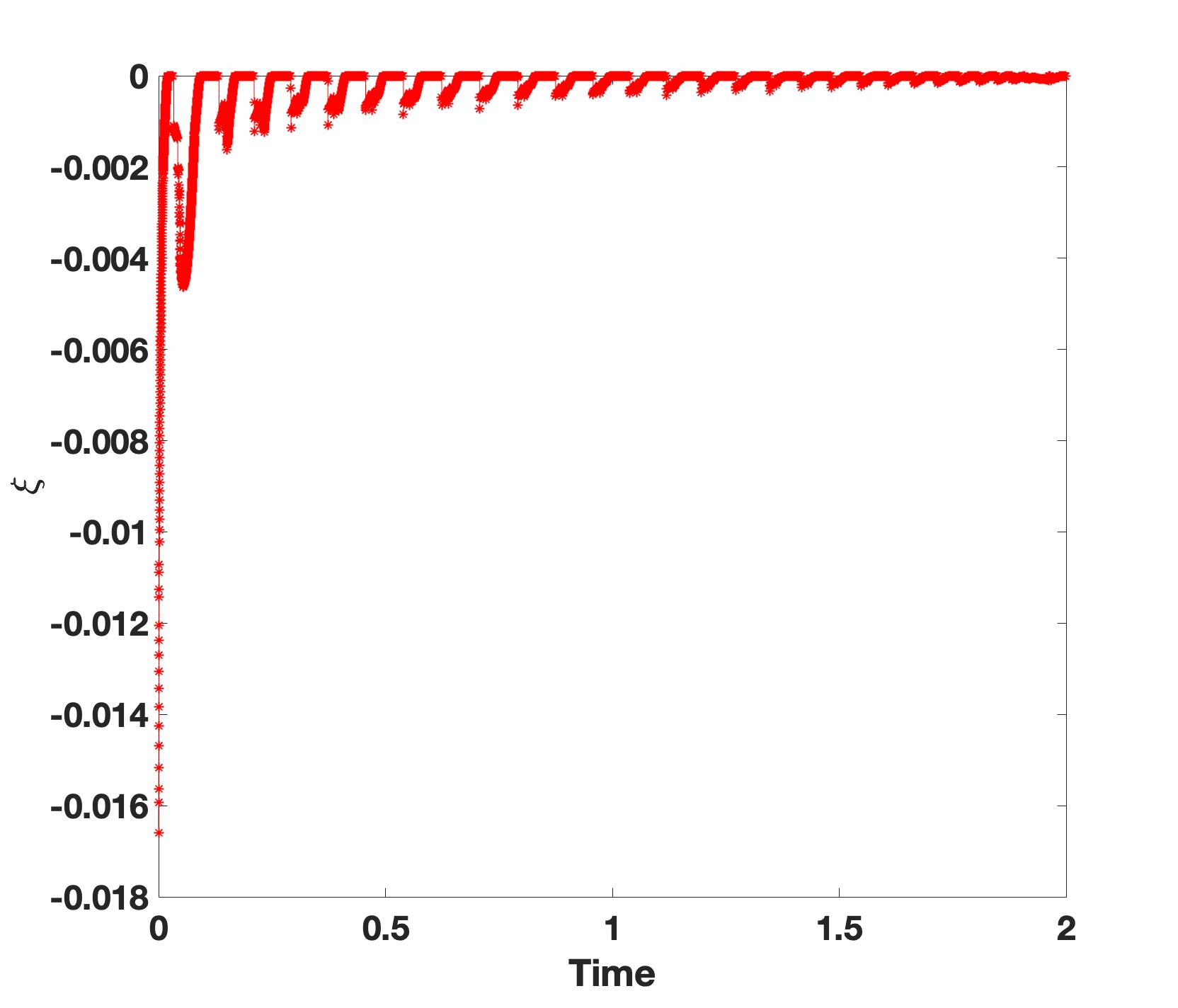}}
\subfigure[Iteration numbers]{
\includegraphics[width=0.45\textwidth,clip==]{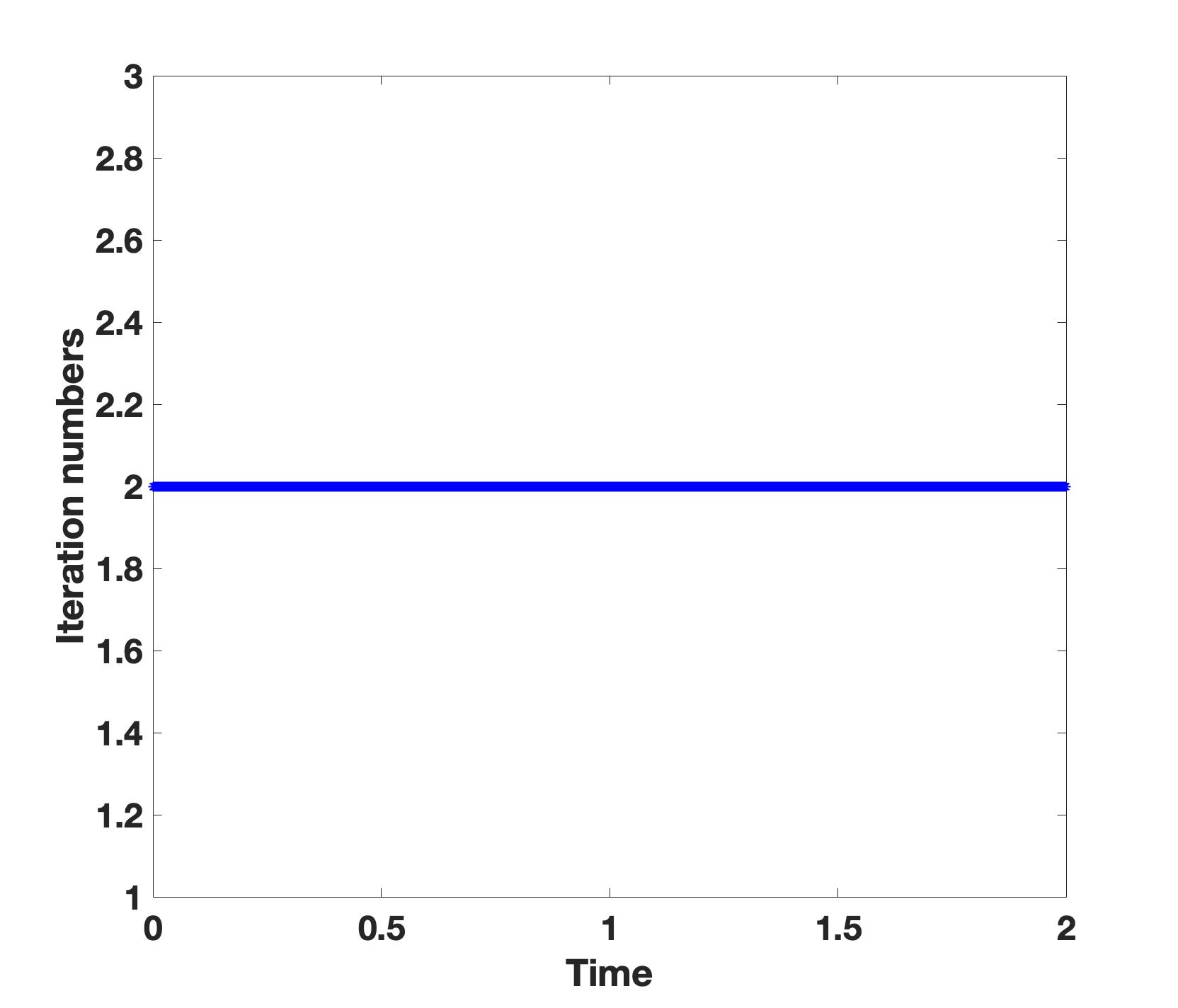}}
\caption{(a): $L^2$ error by second-order positivity schemes with mass conservation and without mass conservation. (b): Evolution of mass with respective to time. (c): Lagrange multiplier $\xi$ for mass conservation. (d): Iteration numbers of solving $F(\xi)=0$.}\label{mass_error_compare}
\end{figure}

\subsection{Poisson-Nernst-Planck equations}
We consider  the followng Poisson-Nernst-Planck  (PNP) system \cite{hu2020fully,he2019positivity} which describes the dynamics of ion transport in ion channels: 
\begin{eqnarray}
&&-\eps^2\Delta \phi = p-n, \mbox{in} ~ \Omega_T:=(0,T]\times \Omega. \label{PNP:1}\\
&&p_t = \Grad\cdot(\Grad p+p\Grad \phi),\label{PNP:2}\\
&&n_t=\Grad\cdot(\Grad n-n\Grad\phi),\label{PNP:3}
\end{eqnarray}
with initial conditions
\begin{equation}
p(0,x)=p_0(x)\geq 0,\quad n(0,x)=n_0(x) \geq 0, \quad \mbox{in} ~ \Omega,
\end{equation}
and homogeneous Neumann boundary conditions
\begin{equation}\label{Nueman}
\frac{\partial p}{\partial \bn}=\frac{\partial n}{\partial \bn}=\frac{\partial \phi}{\partial \bn}=0, \quad \mbox{on} \quad \partial \Omega_T:= (0,T]\times \partial \Omega.
\end{equation}
In the above,  $p$ and $n$ are concentration of positive and negative ions with valence $+1$ and $-1$, respectively, $\phi$ is the electrical potential, $\eps$ is a small positive dimensionless number representing the ratio of the Debye length to the physical characteristic length. The unknown functions $p$ and $n$ have to be positive for the problem to be well posed. Below we use the general approach presented in the last section to construct a positivity preserving scheme for  the PNP equations. Since we need to keep both $p$ and $n$ positive, two Lagrange multipliers $\lambda$ and $\eta$ are needed . Lagrange multipliers $\xi$ and $\gamma$ are used to preserve mass.


We set $X_h=P_N\times P_N$, and  $ \Sigma_h=\{(x_i,x_j), \,1\le i,j\le N-1\}$, where $\{x_k\}_{k=0}^N$ are the roots of $(1-x^2)L'_N(x)$ with $L_N$ being the Legendre polynomial of $N$-th degree. 
And we use the Legendre-Galerkin method with numerical integration in space \cite{shen2011spectral}. Then a second-order positivity preserving scheme based on the scheme \eqref{mass:por:lag:1}-\eqref{mass:por:lag:2c} with $k=2$ is as follows: for $\forall q_h, m_h \in X_h$
\begin{eqnarray}
&&[\frac{3\tilde{p}_h^{n+1}-4p_h^n+p_h^{n-1}}{2\delta t},q_h] =[\Grad p_h^{n+1}+p_h^{n+1,\star}\Grad \phi_h^{n+1,\star},\Grad q_h]+[\lambda_h^n+\xi_h^n,q_h],\label{scheme:PNP:2}\\
&&\frac{3p_h^{n+1}-\tilde{p}_h^{n+1}}{2\delta t}=\lambda_h^{n+1}-\lambda_h^n+\xi_h^{n+1}-\xi_h^n,\label{scheme:PNP:3}\\
&& \lambda_h^{n+1} \ge 0,\; p_h^{n+1}\ge,\; \lambda_h^{n+1} p_h^{n+1}=0,\;[p_h^{n+1},1]=[p_h^n,1];\label{scheme:PNP:4}\\
&&[\frac{3\tilde{n}_h^{n+1}-4n_h^n+n_h^{n-1}}{2\delta t},m_h]= [\Grad n_h^{n+1}-n_h^{n+1,\star}\Grad\phi_h^{n+1,\star},\Grad m_h]+[\eta_h^n+\gamma_h^n,m_h],\label{scheme:PNP:5}\\
&&\frac{3n_h^{n+1}-3\tilde{n}_h^{n+1}}{2\delta t}=\eta_h^{n+1}-\eta_h^n+\gamma_h^{n+1}-\gamma_h^n,\\
&& \eta_h^{n+1} \ge 0,\; n_h^{n+1}\ge 0,\; \eta_h^{n+1} n_h^{n+1}=0,\;[n_h^{n+1},1]=[n_h^n,1];\label{scheme:PNP:6}\\
&&\eps^2[\Grad \phi_h^{n+1},\Grad\psi_h ]=[ p_h^{n+1}-n_h^{n+1},\psi_h], \quad \forall \psi_h \in X_h;\label{scheme:PNP:1}
\end{eqnarray}
where $p_h^{n+1,\star}=2p_h^n-p_h^{n-1}$ and $n_h^{n+1,\star}=2n_h^n-n_h^{n-1}$. In the above,  $p^{n+1}_h$ and $n_h^{n+1}$ are decoupled and can be determined from \eqref{scheme:PNP:2}-\eqref{scheme:PNP:4} and  \eqref{scheme:PNP:5}-\eqref{scheme:PNP:6} respectively.  Once $p^{n+1}_h$ and $n_h^{n+1}$ are known,  $\phi^{n+1}_h$ can be obtained from \eqref{scheme:PNP:1}. Hence, the scheme is very efficient.

We set $\Omega=(-1,1)^2$, $\eps=0.1$, and  use  $\delta t=10^{-3}$,   $N=256$ in the above scheme with the initial conditions:
\begin{equation}\label{initial:pnp}
 \begin{split}
  &p(x,y,0),n(x,y,0) =\begin{cases}1,& x^2+y^2 \leq 0.25,\\
  0,& \mbox{otherwise}\end{cases}.\\
&  \phi(x,y,0)=\begin{cases} (x-0.5)^2(y-0.5)^2,& x^2+y^2 \leq 0.25,\\
  0,& \mbox{otherwise}\end{cases}.
  \end{split}
  \end{equation}
 The numerical solution at different times are  plotted in  Fig,\;\ref{pnp:positivity}. We
 observe that  $p$ and $n$ are always non-negative. We also plot the Lagrange multipliers $\lambda$ and $\eta$ in Fig.\;\ref{pnp:lambda} at time $t=3\times 10^{-3}$. Since the solutions of the  PNP system  are smooth, the   Lagrange multipliers are zero at most places, and are  non-zero only  at some localized boundary with quite small values.

\begin{figure}[htbp]
\centering
\subfigure[$p_h:t=0.01$]{
\includegraphics[width=0.30\textwidth,clip==]{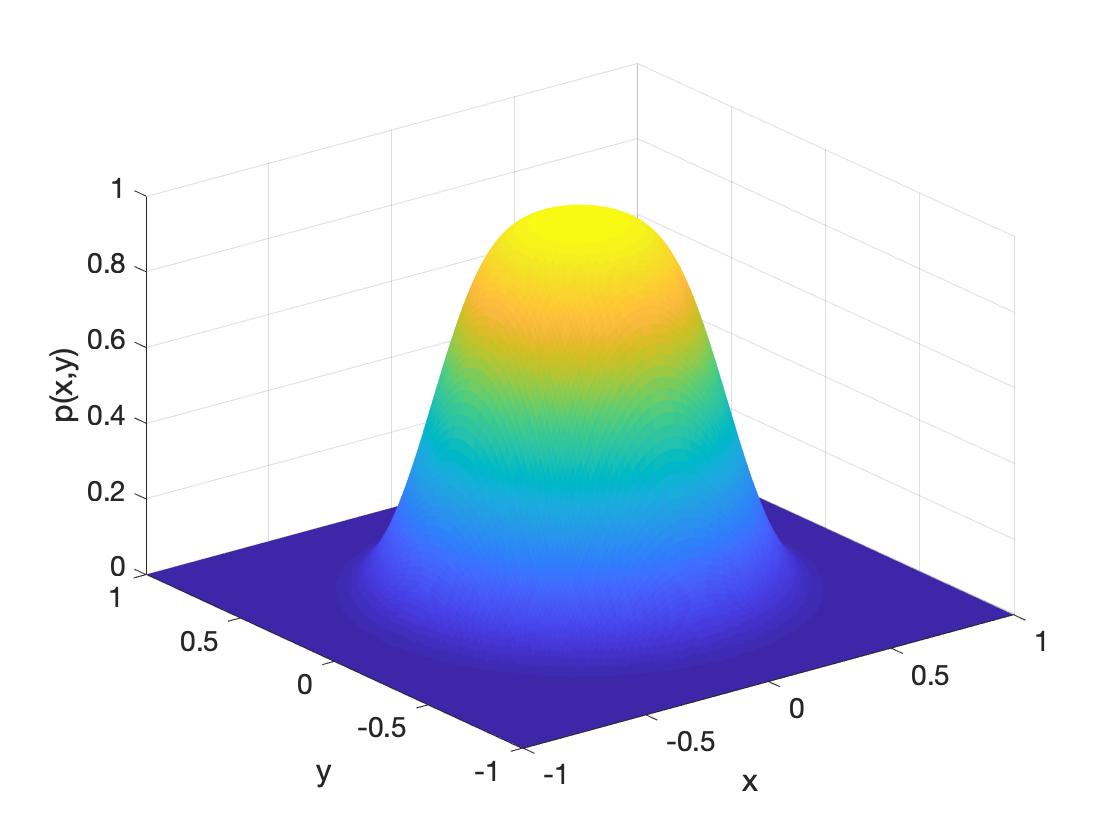}}
\subfigure[$p_h:t:=0.1$]{
\includegraphics[width=0.30\textwidth,clip==]{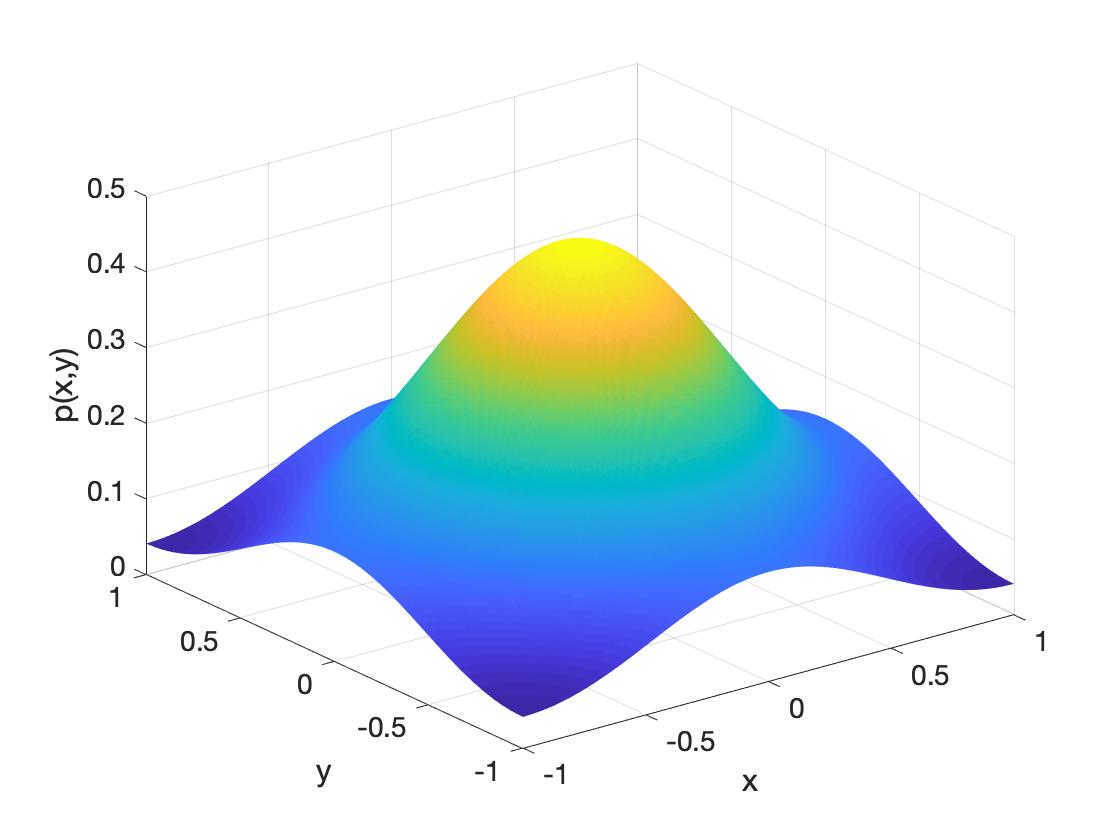}}
\subfigure[$p_h:t=0.5$]{
\includegraphics[width=0.30\textwidth,clip==]{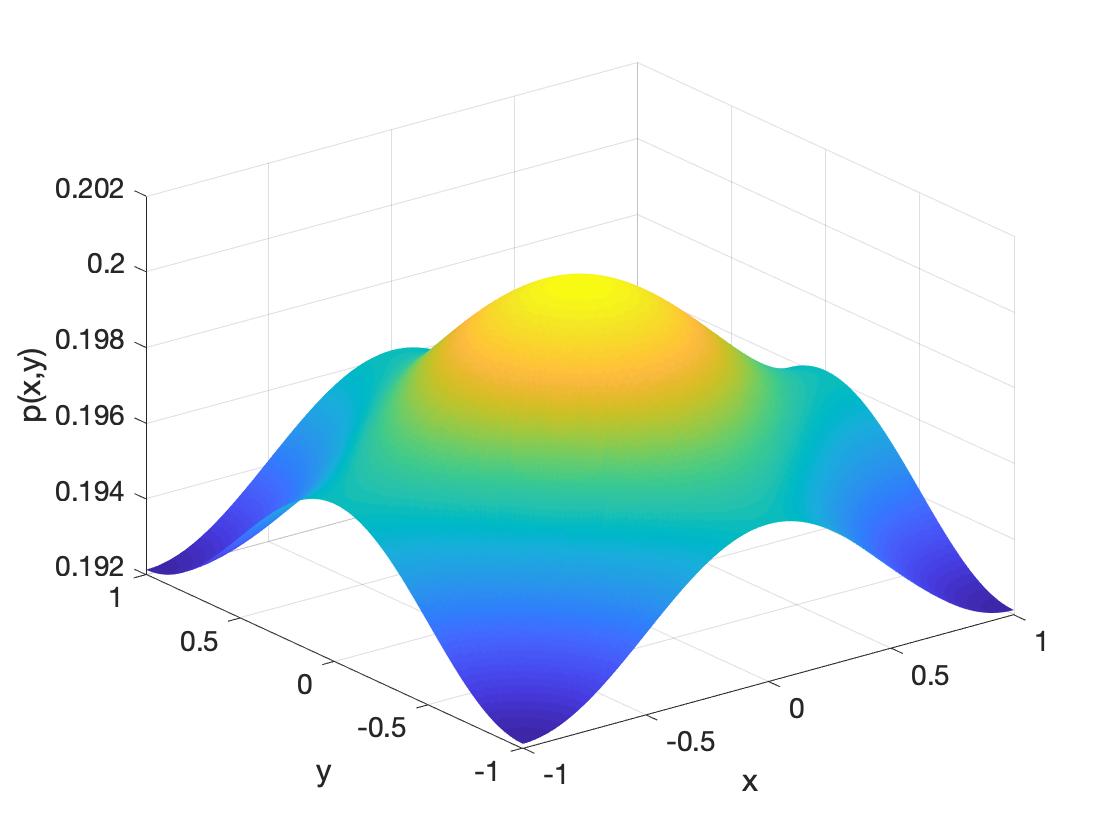}}
\subfigure[$p_h:t=1$]{
\includegraphics[width=0.30\textwidth,clip==]{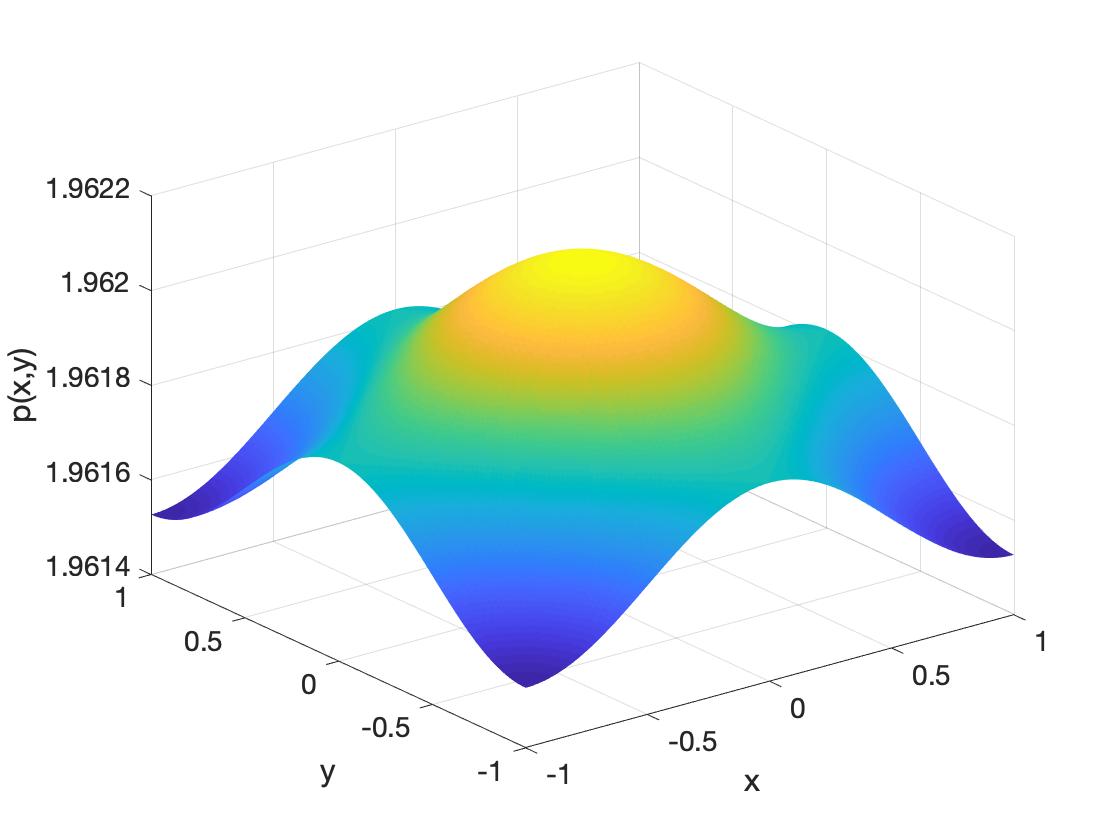}}
\subfigure[$n_h:t=0.01$]{
\includegraphics[width=0.30\textwidth,clip==]{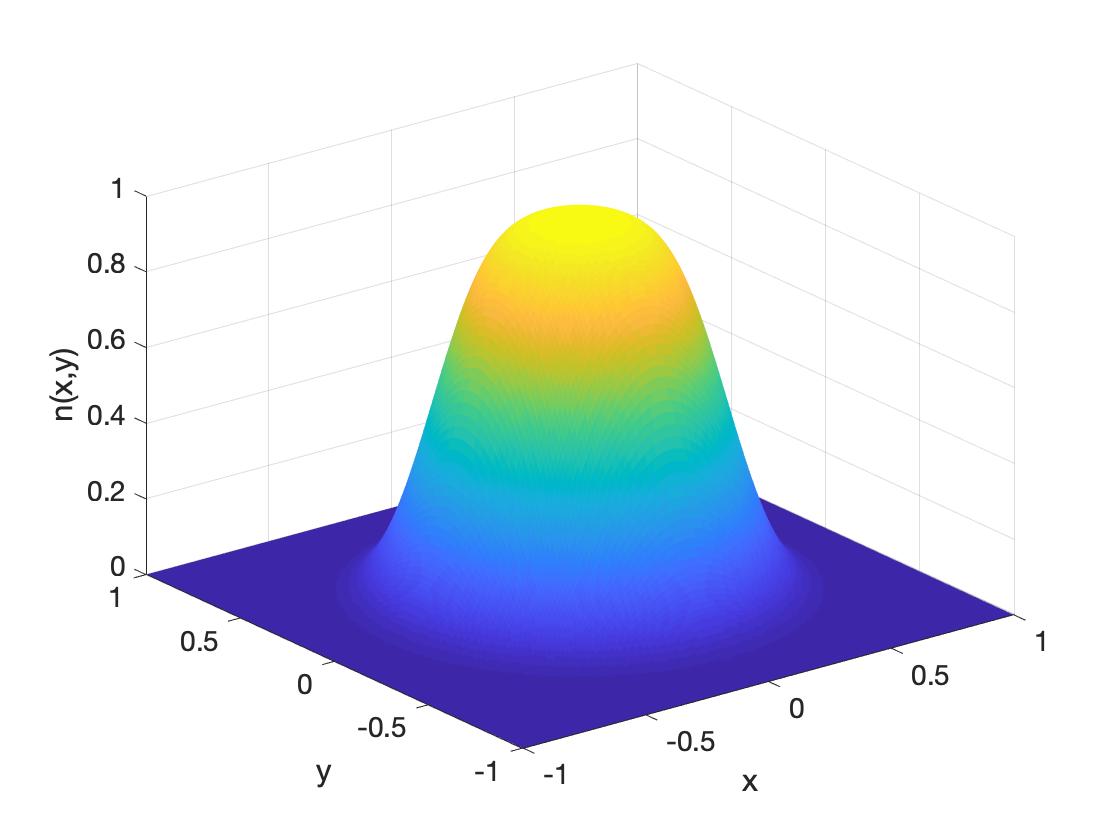}}
\subfigure[$n_h:t:=0.1$]{
\includegraphics[width=0.30\textwidth,clip==]{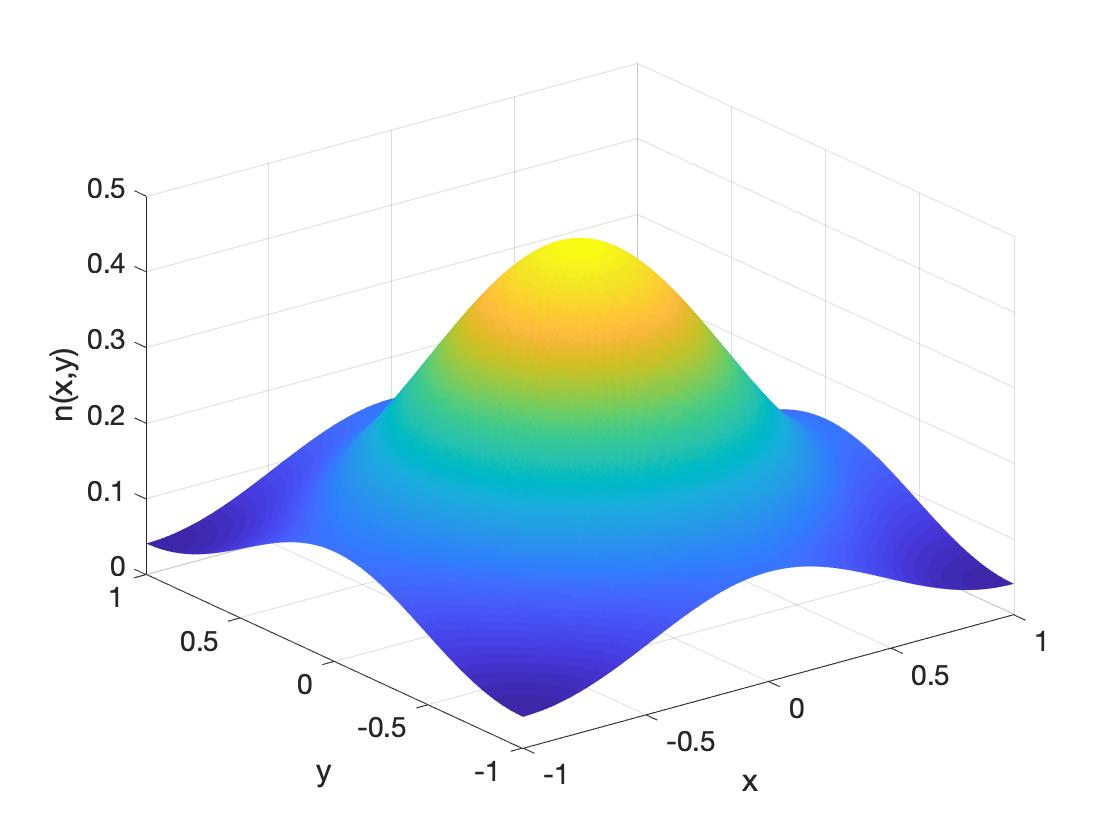}}
\subfigure[$n_h:t=0.5$]{
\includegraphics[width=0.30\textwidth,clip==]{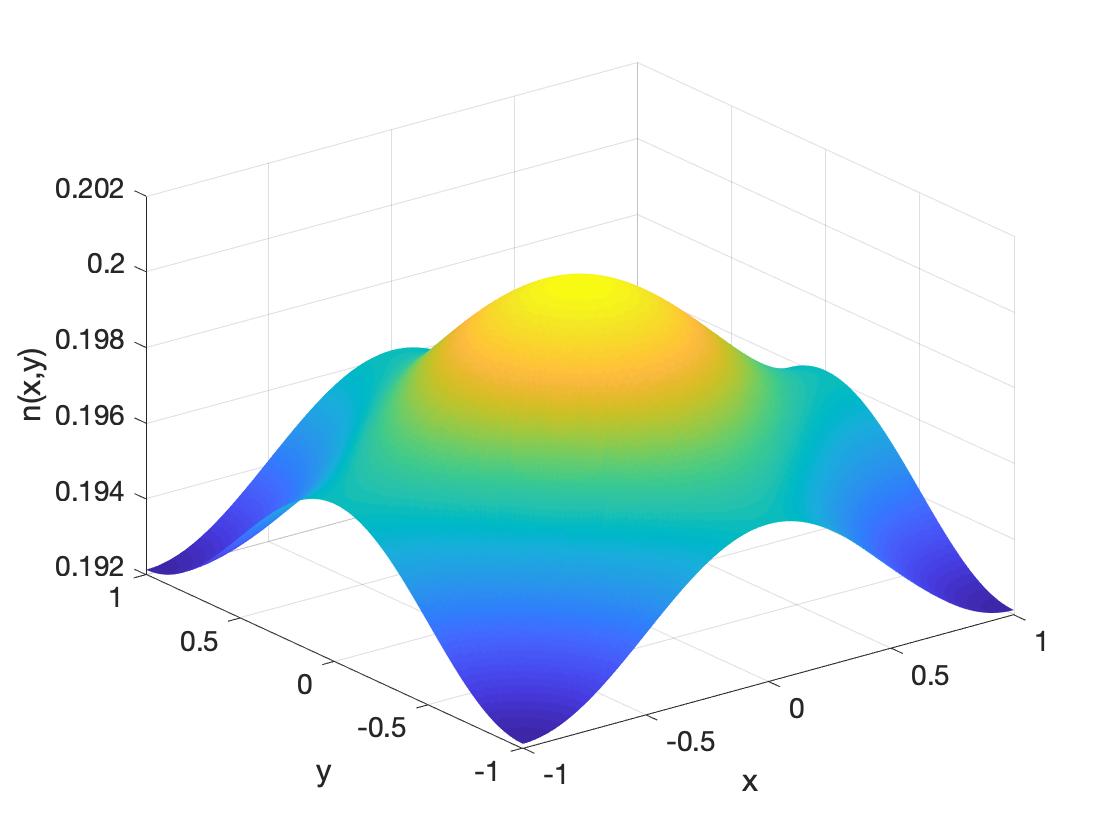}}
\subfigure[$n_h:t=1$]{
\includegraphics[width=0.30\textwidth,clip==]{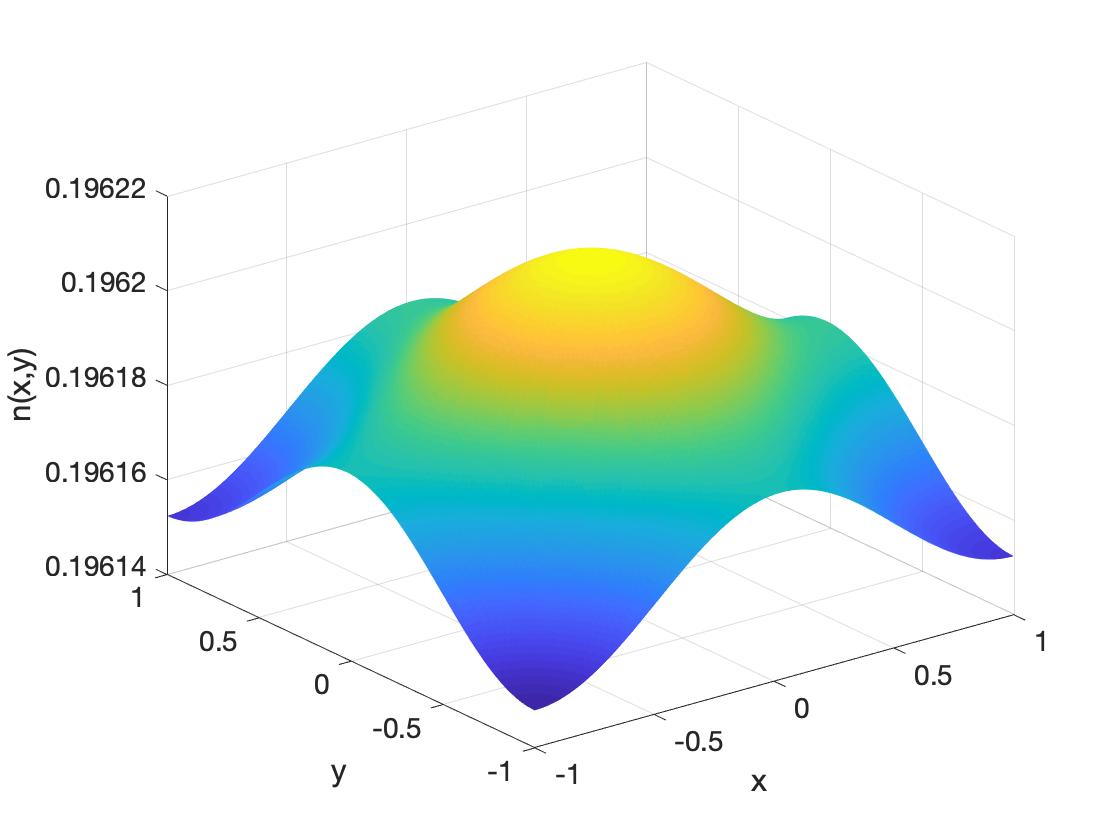}}
\subfigure[$\phi_h:t=0$]{
\includegraphics[width=0.30\textwidth,clip==]{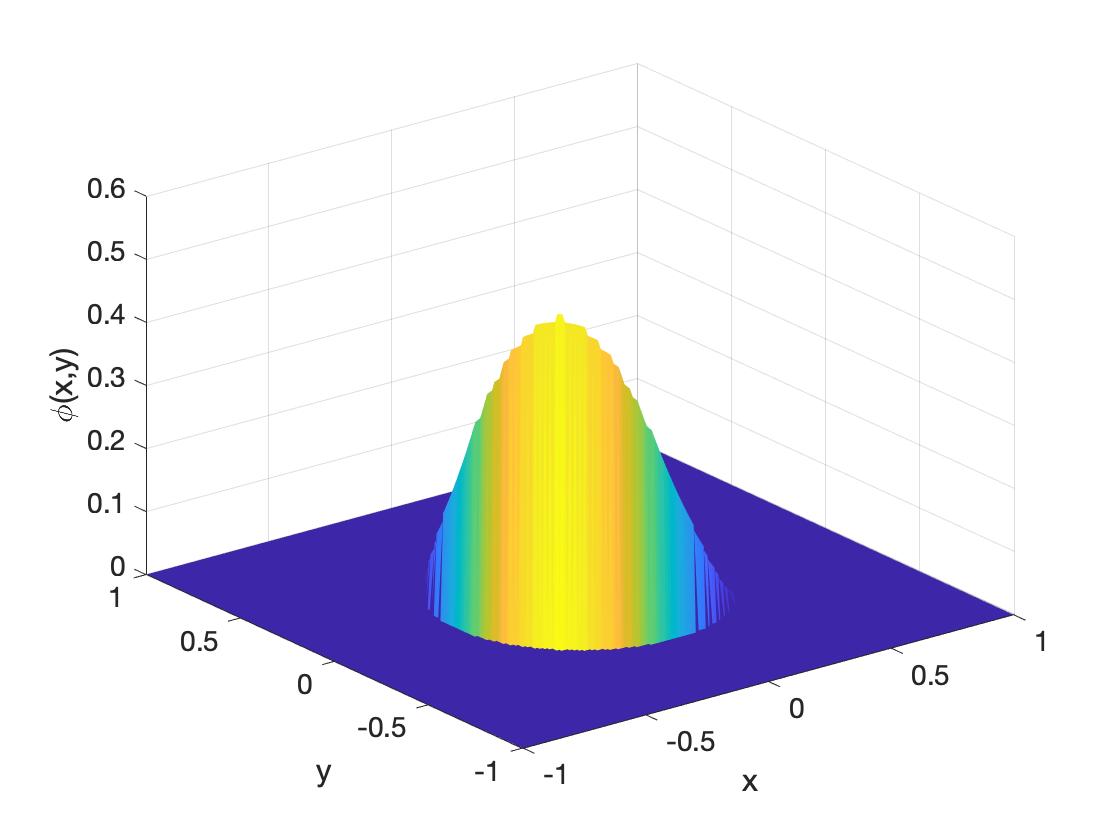}}
\subfigure[$\phi_h:t:=0.01$]{
\includegraphics[width=0.30\textwidth,clip==]{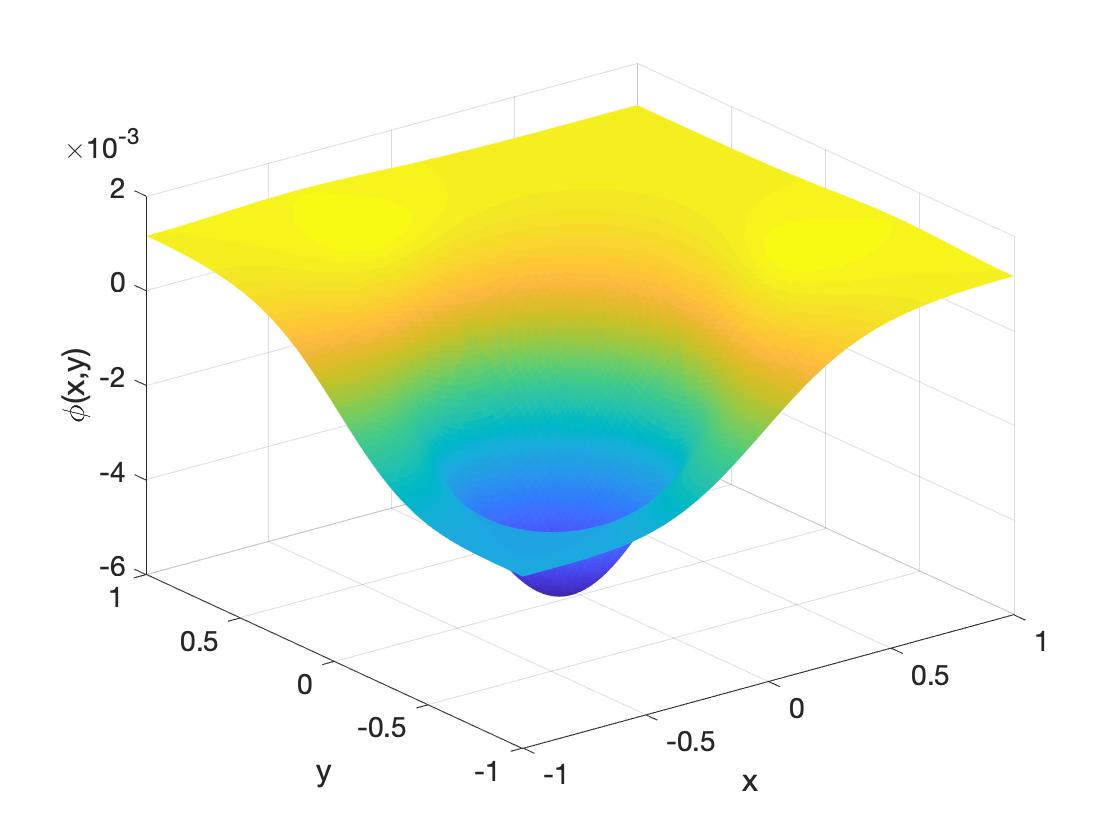}}
\subfigure[$\phi_h:t=0.1$]{
\includegraphics[width=0.30\textwidth,clip==]{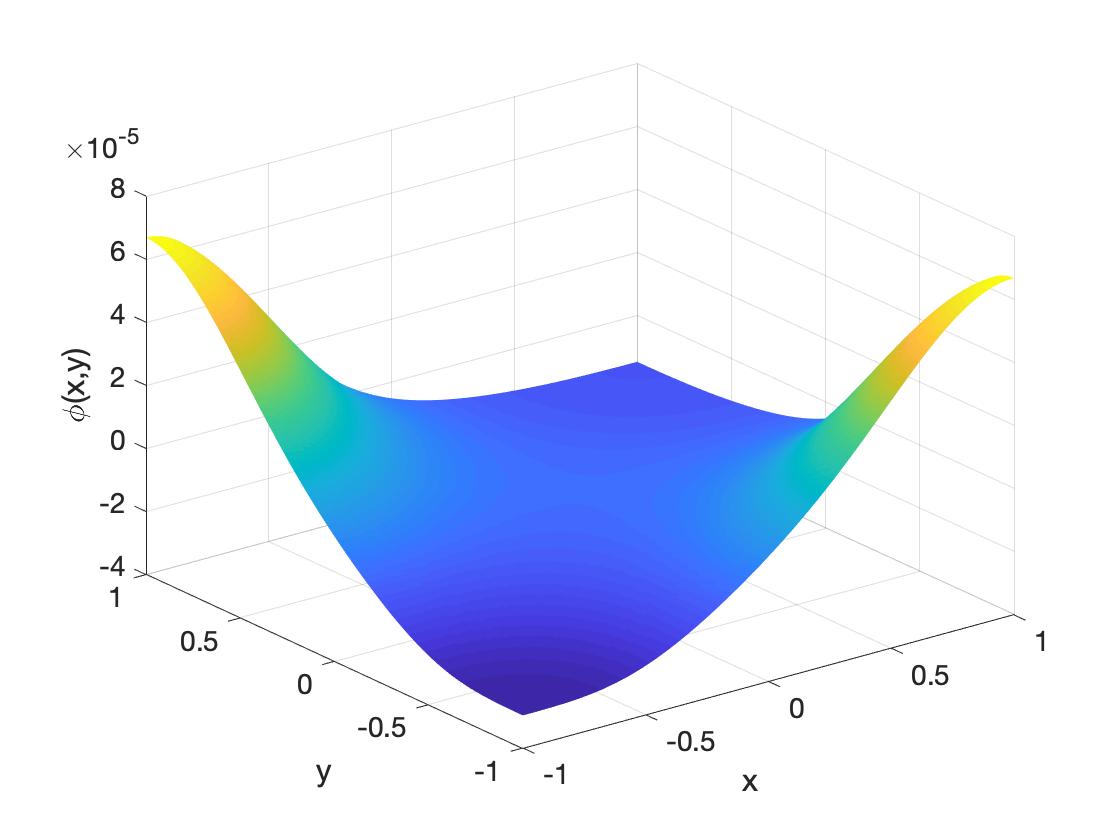}}
\subfigure[$\phi_h:t=1$]{
\includegraphics[width=0.30\textwidth,clip==]{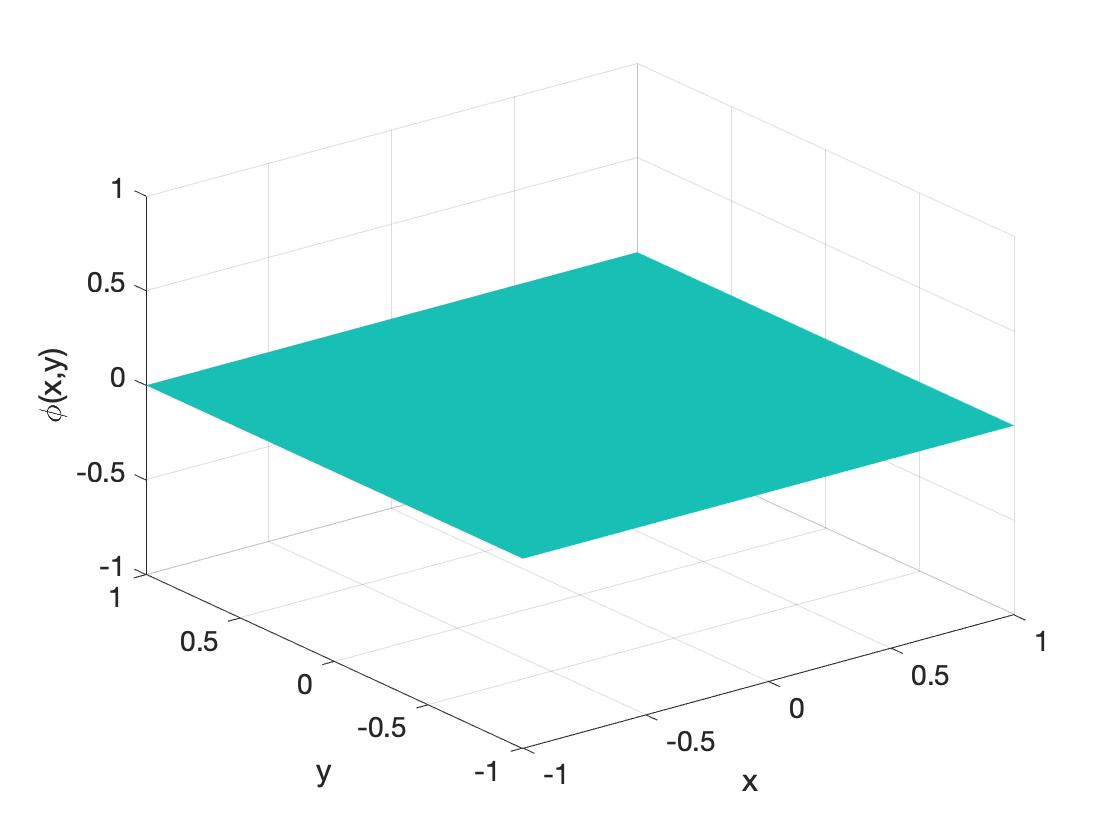}}
\caption{The numerical solution of  PNP equation \eqref{PNP:1}-\eqref{PNP:3} with time step $1\times 10^{-3}$ and $N=256$  in 2D by using scheme \eqref{scheme:PNP:2}-\eqref{scheme:PNP:1}.  Parameter $\eps=0.1$. }\label{pnp:positivity}
\end{figure}

\begin{figure}[htbp]
\centering
\subfigure[$\lambda_h$]{
\includegraphics[width=0.30\textwidth,clip==]{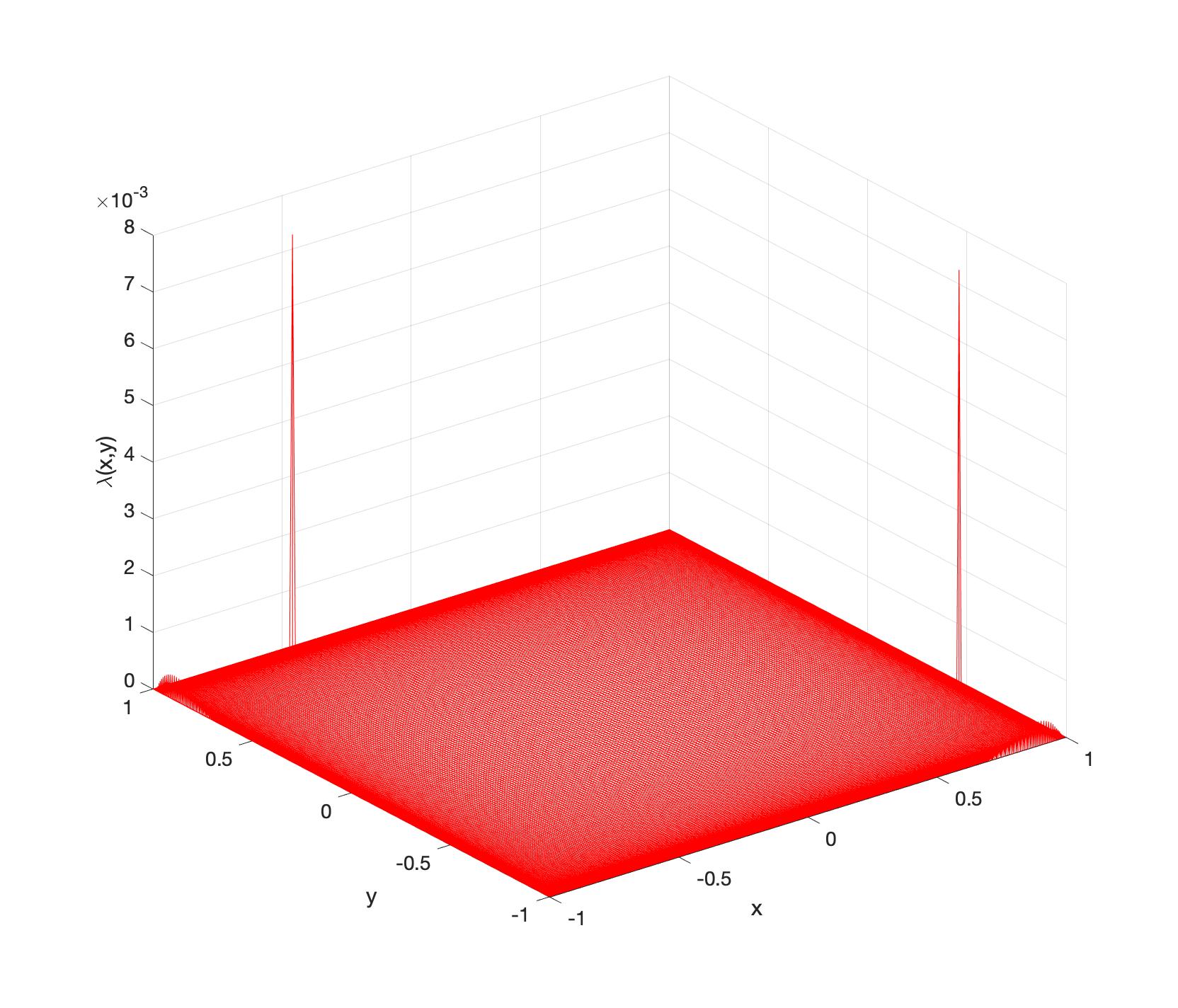}}
\subfigure[$\eta_h$]{
\includegraphics[width=0.30\textwidth,clip==]{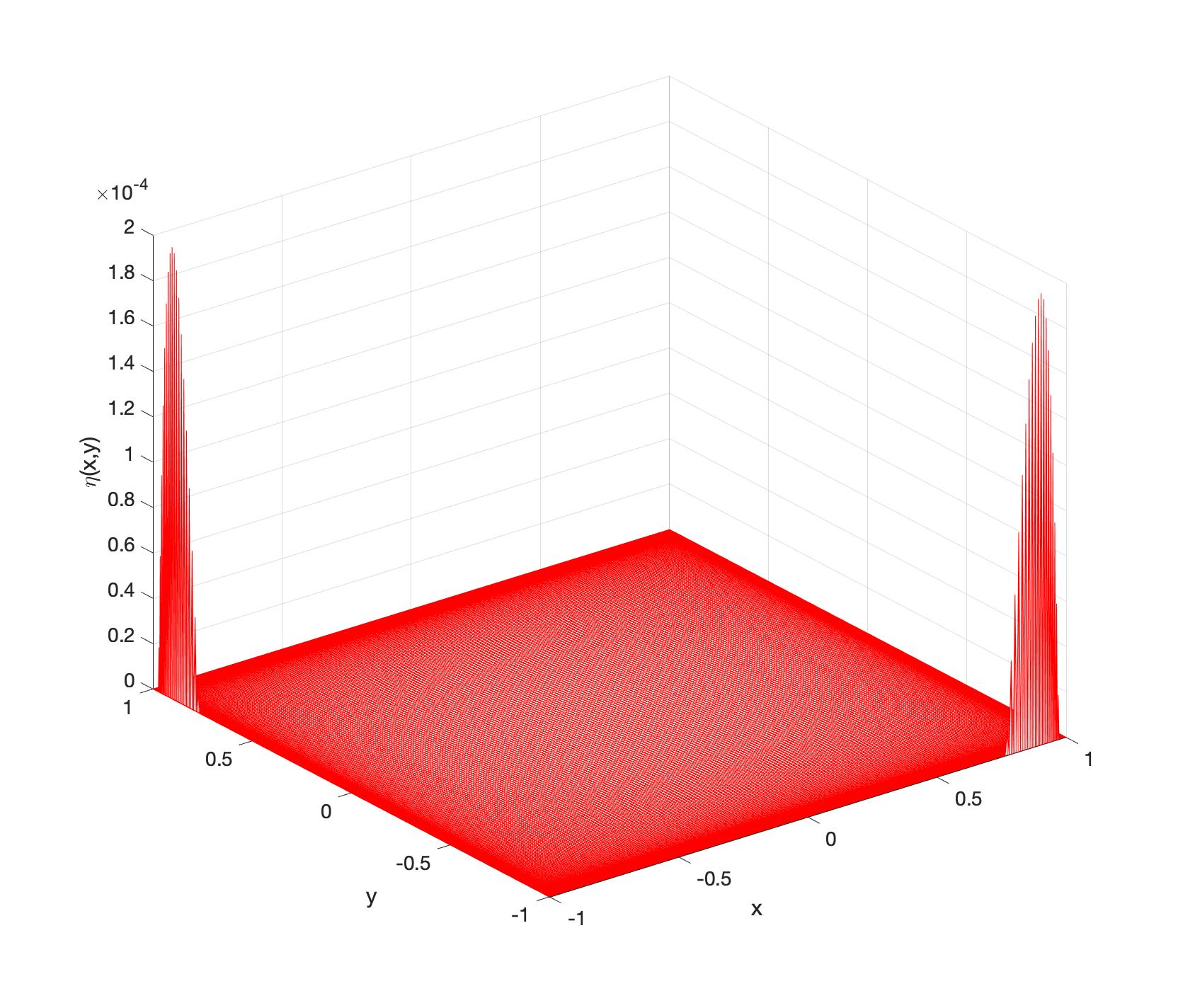}}
\caption{The Lagrange multipliers  of  PNP equation \eqref{PNP:1}-\eqref{PNP:3}  $\lambda$ and  $\eta$ at time $t=3\times 10^{-3}$ for numerical simulations at Fig.\;\ref{pnp:positivity}. }\label{pnp:lambda}
\end{figure}

\subsection{Lubrication-type equation}
As the last example, we consider the following lubrication equation \cite{zhornitskaya1999positivity}
\begin{equation}\label{Lu:eq}
 u_t +\Grad\cdot(f( u)\Grad\Delta  u)=0,
\end{equation}
where $f( u) \approx  u^\rho$ as $ u \rightarrow 0$ with $\rho$ depending on the boundary condition at the liquid solid interface: $\rho=3$ for no-slip boundary condition while $0<\rho< 3$ for various other boundary condition.  The above equation has been used, e.g.,  in the study of thin liquid films and fluid interfaces by surface tension \cite{bertozzi1996lubrication}.

Below, we consider \eqref{Lu:eq} in $\Omega=(-1,1)$ and  $\Omega=(-\pi,\pi)^2$  with periodic boundary conditions, and use a Fourier collocation method in space. Since the equation \eqref{Lu:eq} may develop a singularity in finite time, it is a common practice to regularize it \cite{bertozzi1996lubrication,zhornitskaya1999positivity}. In \cite{bertozzi1996lubrication}, the equation is regularized by replacing $f(u)$ with 
$f_\eta(u)=\frac{u^4f(u)}{\eta f(u)+u^4}$ and it is shown in \cite{bertozzi1996lubrication} that the regularized problem is well posed for all time. On the other hand, one can also   regularize the equation by requiring the solution to be bounded away from zero, namely, $u(\bz) \ge \eps$ for a prescribed $\eps$.

Hence, a second-order scheme 
 based on \eqref{mass:por:lag:1}-\eqref{mass:por:lag:2c} with $k=2$ is as follows:
\begin{equation}\label{scheme:Lu:eq:LM:1d}
\frac{3\tilde{ u}_h^{n+1}(\bz)-4 u_h^n(\bz)+ u_h^{n-1}(\bz)}{2\delta t}+ \mathcal{L}_h^n \tilde{ u}_h^{n+1} (\bz)=\lambda_h^n(\bz)+\xi_h^n,\quad\forall \bz\in \Sigma_h,
\end{equation}
and 
\begin{equation}\label{scheme:Lu:eq:LM:1db}
\begin{split}
&\frac{3 u_h^{n+1}(\bz)-3\tilde{ u}_h^{n+1}(\bz)}{2\delta t}=\lambda_h^{n+1}(\bz)-\lambda_h^n(\bz)+\xi_h^{n+1}-\xi_h^n,\quad\forall  \bz\in \Sigma_h,\\
& u_h^{n+1}(\bz) \ge \eps,\;\lambda_h^{n+1}(\bz) \ge0,\;\lambda_h^{n+1} (u^{n+1}(\bz)-\eps)=0,\quad\forall  \bz\in \Sigma_h,\\
& [u_h^{n+1},1]=[u_h^n,1],
\end{split}
\end{equation}
where $\mathcal{L}_h^n $ is defined as follows depending on the type of regularization:
\begin{itemize}
\item $\eps=0$ and $\mathcal{L}_h^n \tilde u_h^{n+1}=\nabla\cdot (f_\eta( u_h^{n+1,*}) \Grad\Delta \tilde u_h^{n+1})$ with $u_h^{n+1,*}$ defined in \eqref{ustar}.
\item $\eps>0$ and $\mathcal{L}_h^n \tilde u_h^{n+1}=\nabla\cdot (f( u_h^{n+1,*}) \Grad\Delta \tilde u_h^{n+1})$ with $u_h^{n+1,*}$ defined in \eqref{ustar}.
\end{itemize}

The second step \eqref{scheme:Lu:eq:LM:1db} can be implemented as 
\begin{equation}
 ( u_h^{n+1},\lambda_h^{n+1})=\left\{
\begin{array}{rcl}
(\tilde{ u}_h^{n+1}-\frac 23\delta t(\lambda_h^n-\xi_h^{n+1}+\xi_h^n),0)       &      & {\mbox{if} \quad \eps      \leq  \tilde{ u}_h^{n+1}-\frac 23\delta t(\lambda_h^n-\xi_h^{n+1}+\xi_h^n) },\\
(\eps,\lambda_h^n+\xi_h^n-\xi_h^{n+1}+\frac{3}{2\delta t}(\eps-\tilde{ u}_h^{n+1}))    &      & \text{otherwise.}\\
\end{array} \right.
\end{equation}

We consider first  the one-dimensional case with  $f( u)= u^{\frac 12}$ 
in $(-1,1)$ with  periodic boundary conditions and the initial condition
 \begin{equation}\label{fourth:init:1d}
  u_0(x)=0.8-\cos(\pi x)+0.25\cos(2\pi x).
 \end{equation}
 This example has been well studied in \cite{zhornitskaya1999positivity}, and  the original equation $f( u)= u^{\frac 12}$ will develop a singularity at $t\approx 0.00074$. However, with a regularization, the solution 
 can be continued beyond the singularity.

\begin{figure}[htbp]
\centering
\includegraphics[width=0.45\textwidth,clip==]{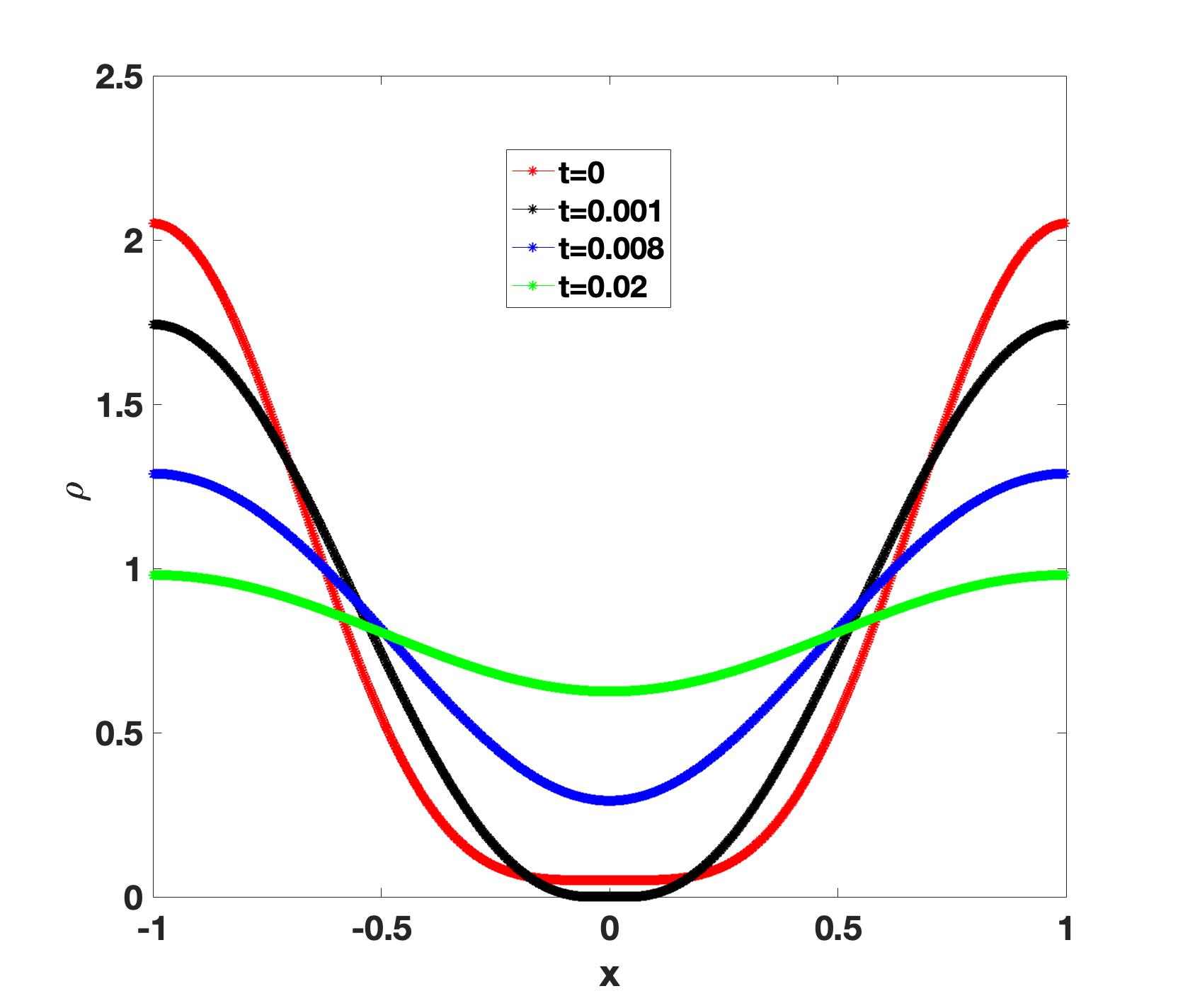}
\includegraphics[width=0.45\textwidth,clip==]{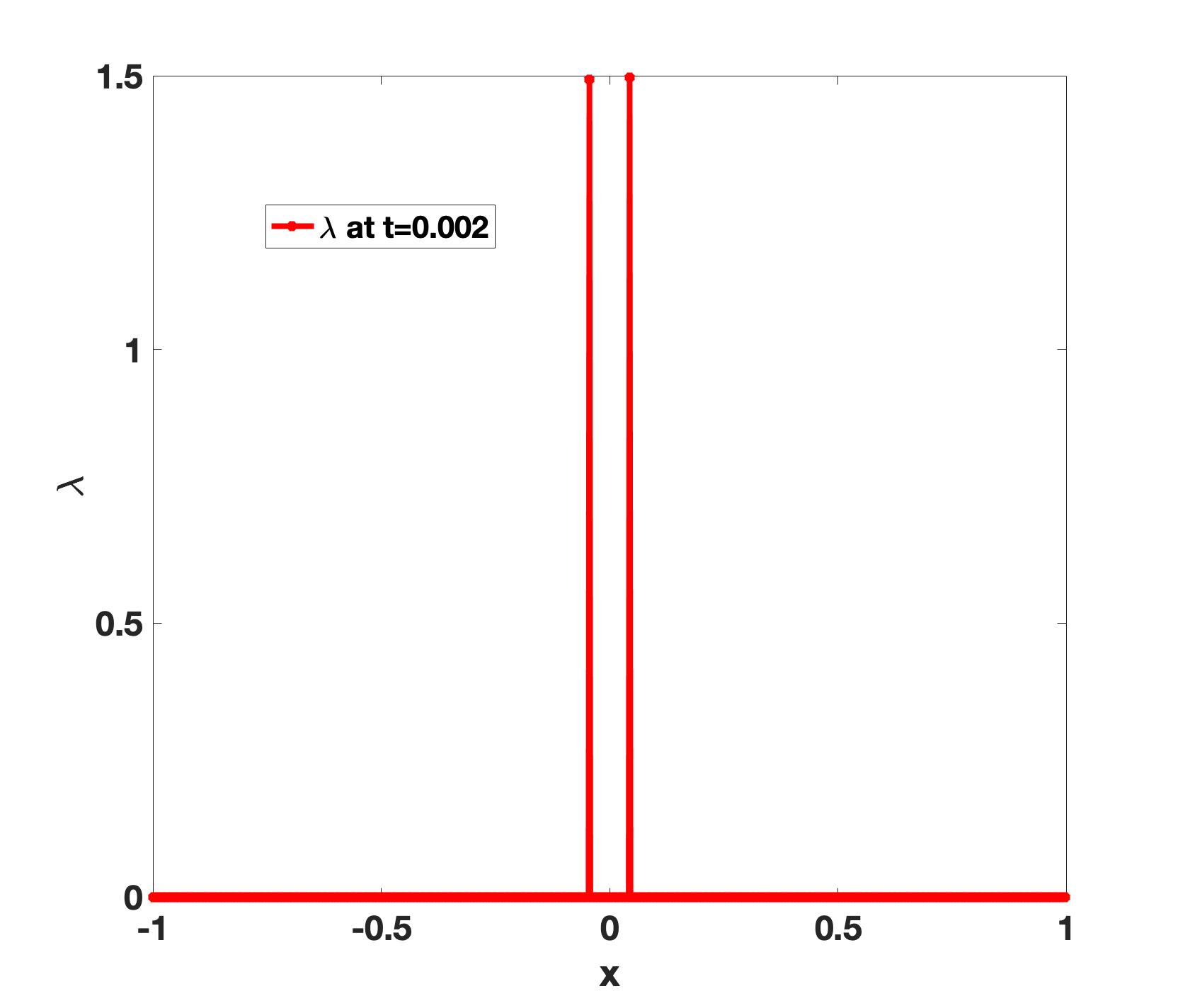}
\caption{Numerical solutions $u$ and Lagrange multiplier $\lambda$ of postivity preserving scheme computed with $\eta=10^{-12}$, $\eps=0$  and $\delta t=2\times 10^{-8}$.}\label{regular:numeric}
\end{figure}

In Fig.\;\ref{regular:numeric}, numerical solutions $u$ and Lagrange multiplier $\lambda$ are shown  at different times computed by regularized postivity preserving scheme  with $N=1000$,  $\eta=10^{-12}$, $\eps=0$  and time step $\delta t=2\times 10^{-8}$. Numerical solutions in Fig.\;\ref{regular:numeric} are indistinguishable from results computed with \eqref{scheme:Lu:eq:LM:1d}.

 In Fig.\;\ref{Lu:1d:numeric}(a-d), we plot the numerical solutions at different times computed with \eqref{scheme:Lu:eq:LM:1d} using $1000$ Fourier modes with various $\eps$ and  $\delta t$. 
 We observe that the numerical solutions are indistinguishable with $\eps$ ranging from $10^{-2}$ to $10^{-4}$. However, as we decrease $\eps$, smaller time steps has to be used.  The Lagrange multiplier $\lambda_h$  at time $t=0.001$ and $t=0.0008$ are plotted in In Fig.\;\ref{Lu:1d:numeric}(e-f). We  observe that $\lambda_h$ becomes large  near the places where the solution approaches zero.

\begin{figure}[htbp]
\centering
\subfigure[$\eps=10^{-3}$ and  $\delta t=10^{-6}$]{
\includegraphics[width=0.45\textwidth,clip==]{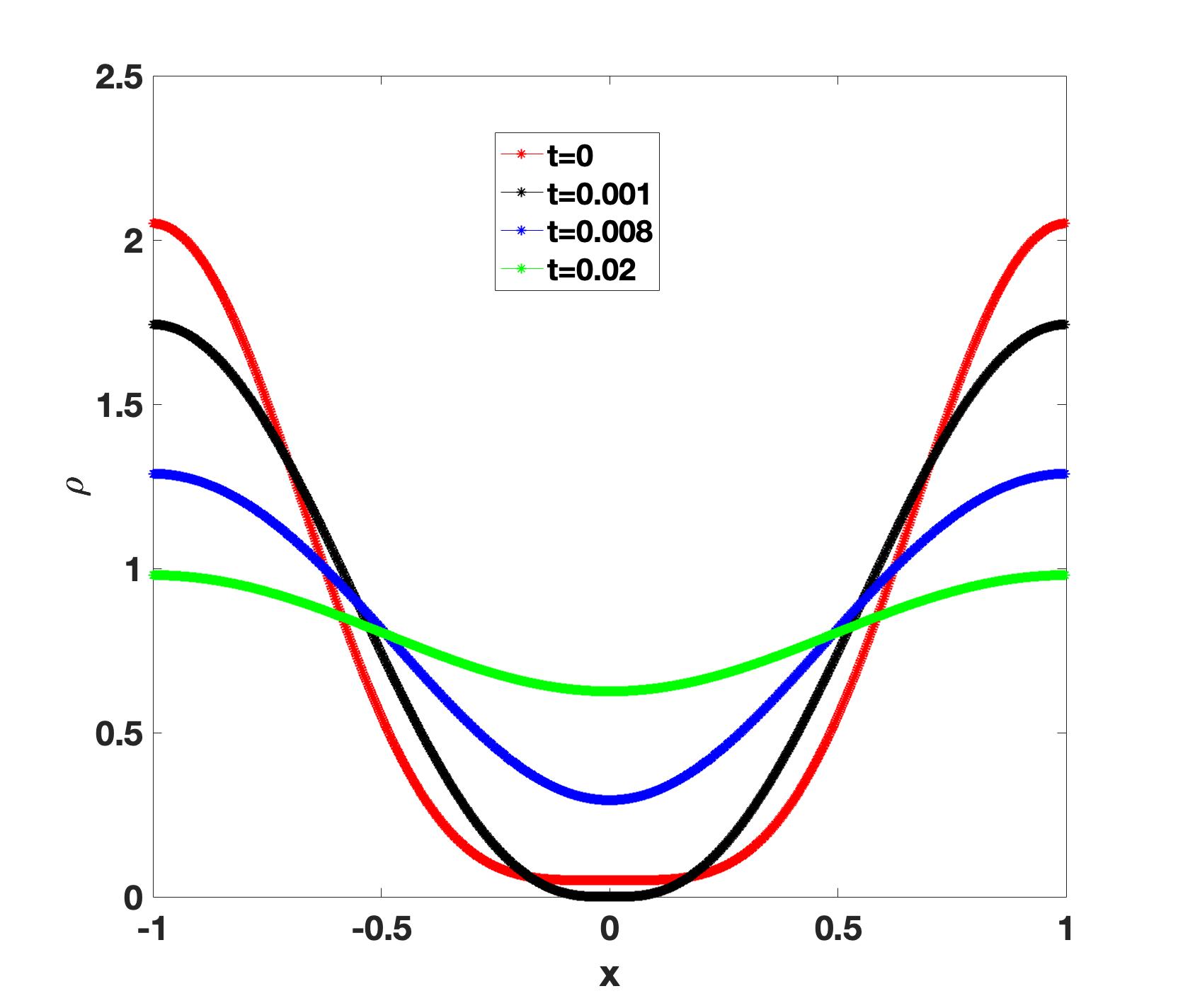}}
\subfigure[$\eps=10^{-4}$ and  $\delta t=2\times 10^{-7}$]{
\includegraphics[width=0.45\textwidth,clip==]{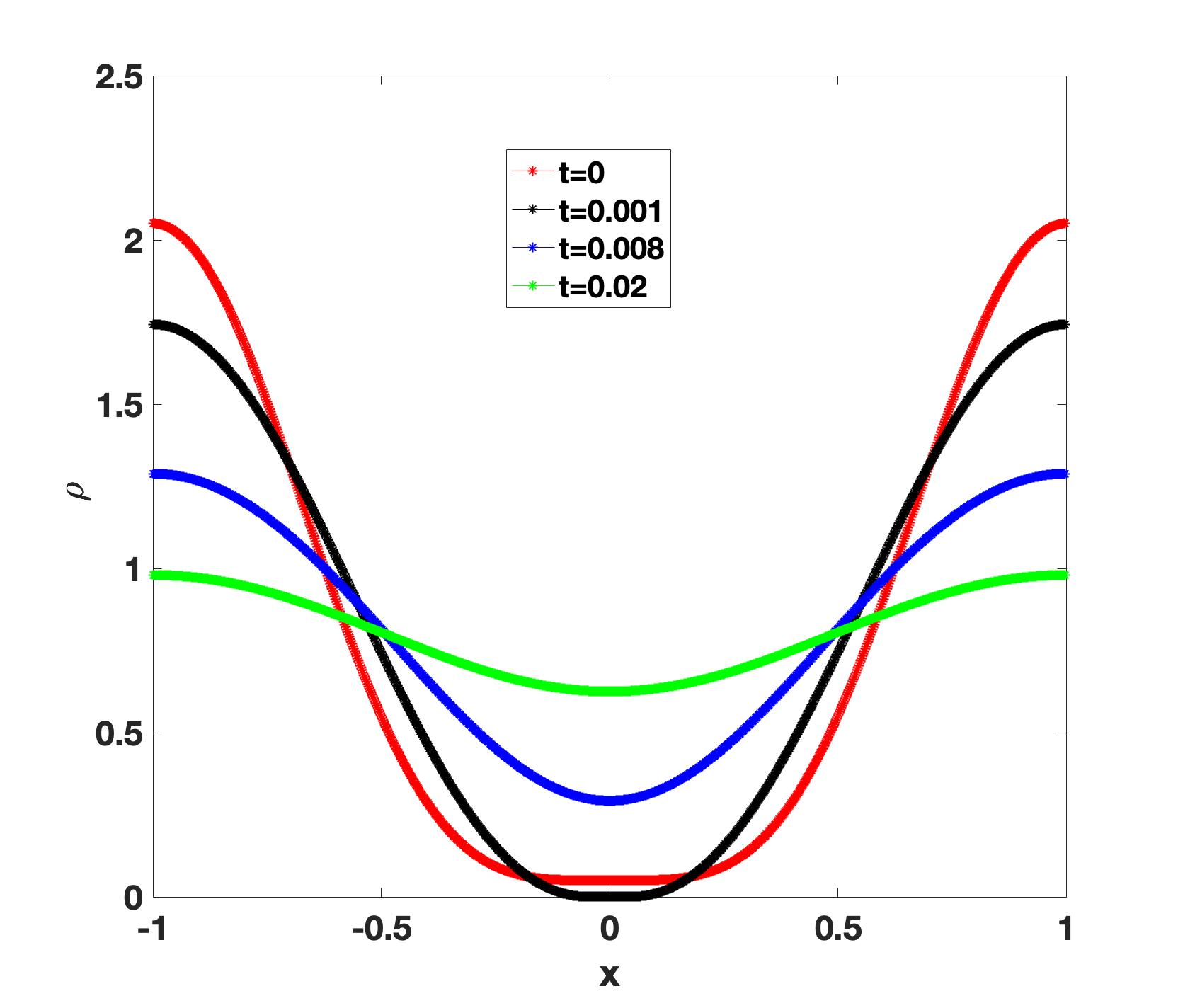}}
\subfigure[$\eps=10^{-2}$ and  $\delta t=10^{-5}$]{
\includegraphics[width=0.45\textwidth,clip==]{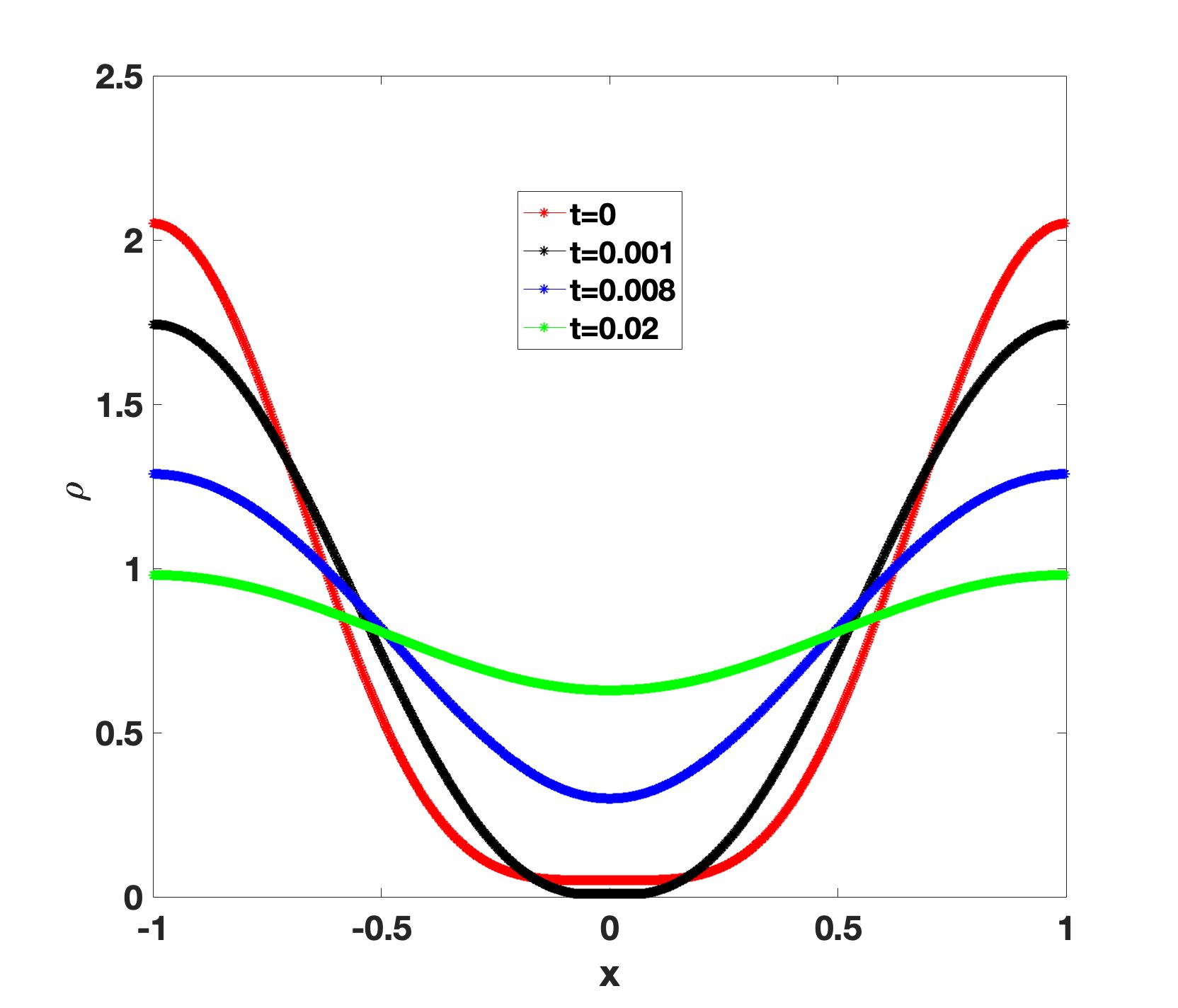}}
\subfigure[$\eps=10^{-2}$ and  $\delta t=10^{-4}$]{
\includegraphics[width=0.45\textwidth,clip==]{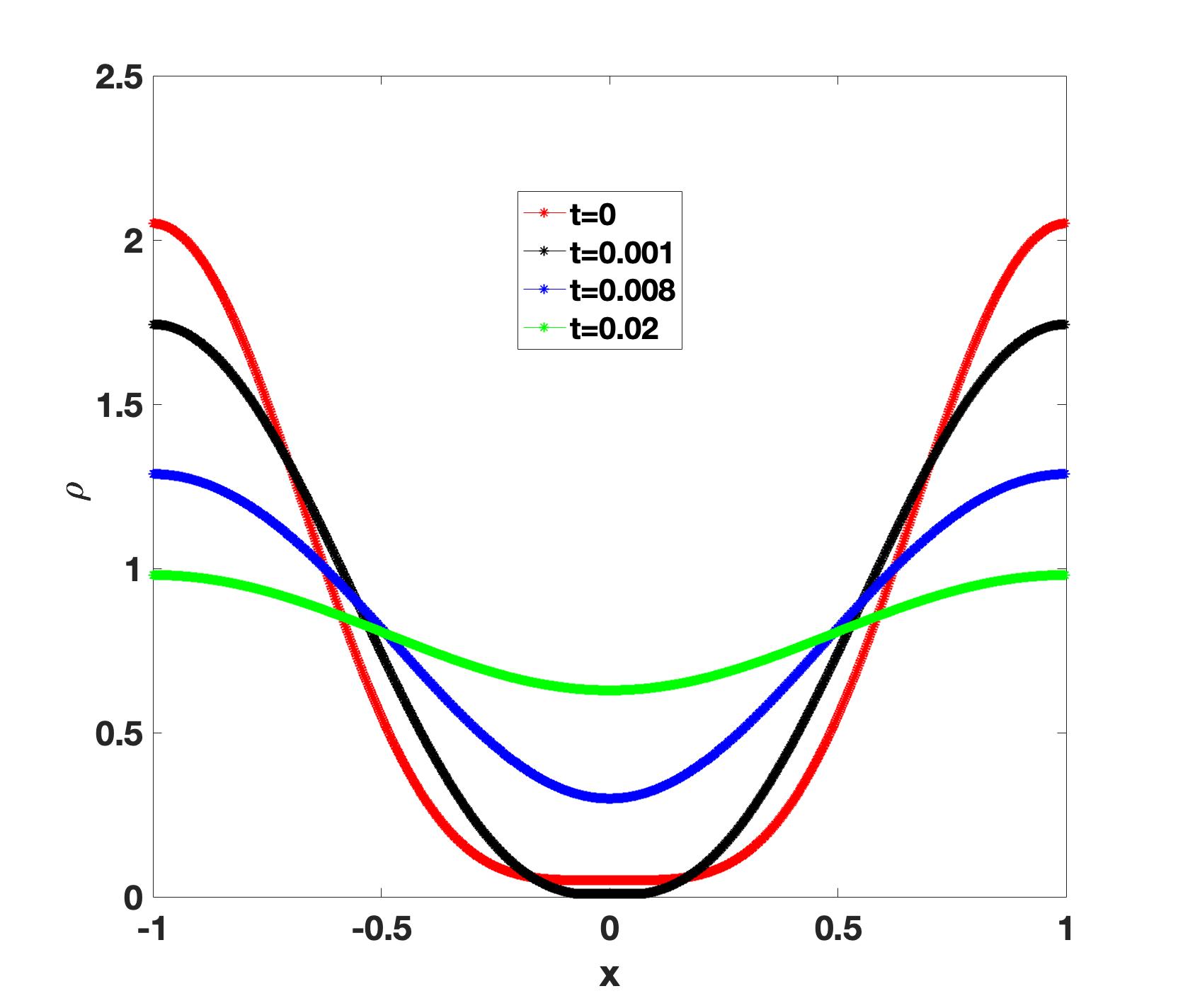}}
\subfigure[$\eps=10^{-2}$ and  $\delta t=10^{-4}$]{
\includegraphics[width=0.45\textwidth,clip==]{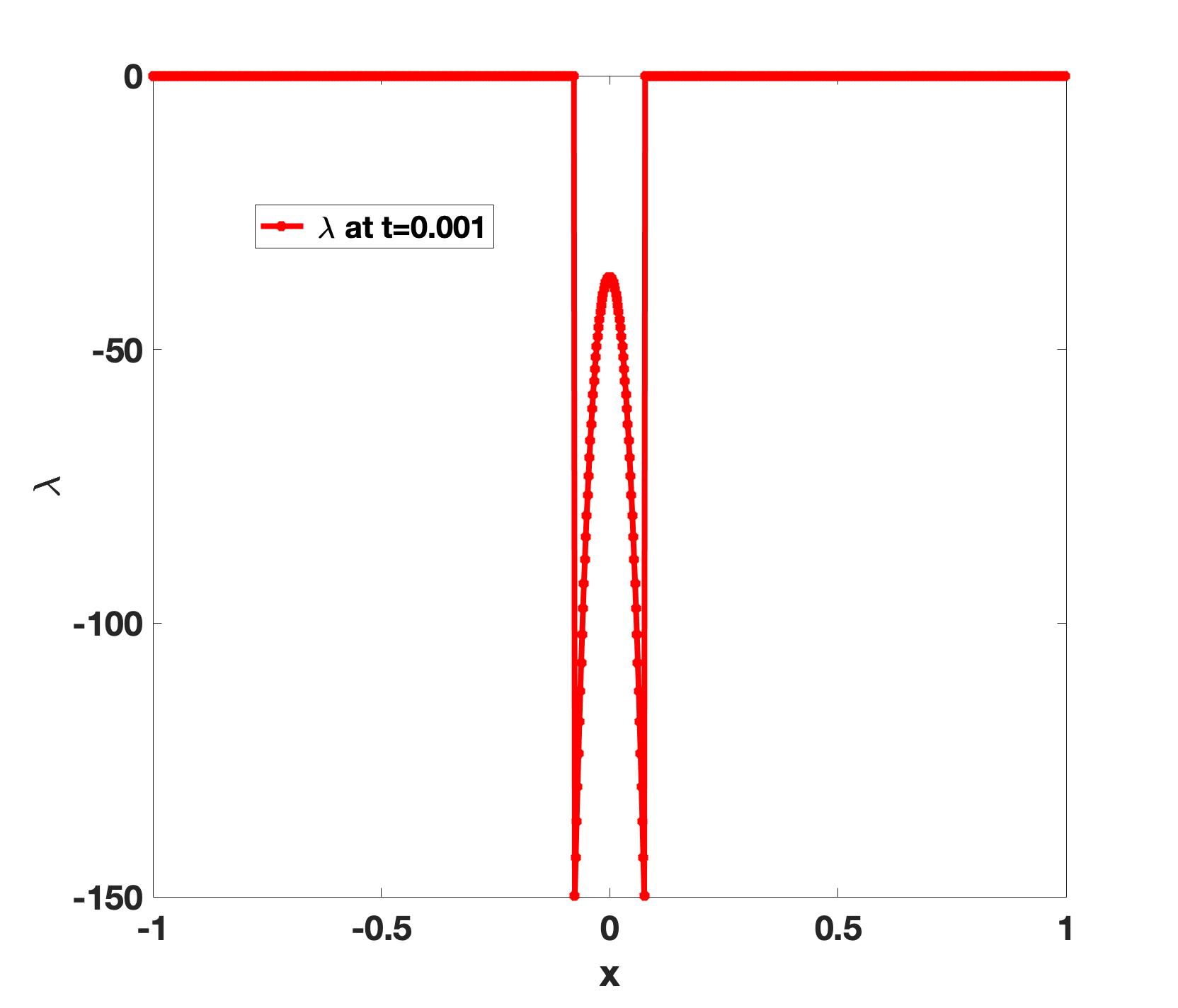}}
\subfigure[$\eps=10^{-4}$ and  $\delta t=2\times 10^{-7}$]{
\includegraphics[width=0.45\textwidth,clip==]{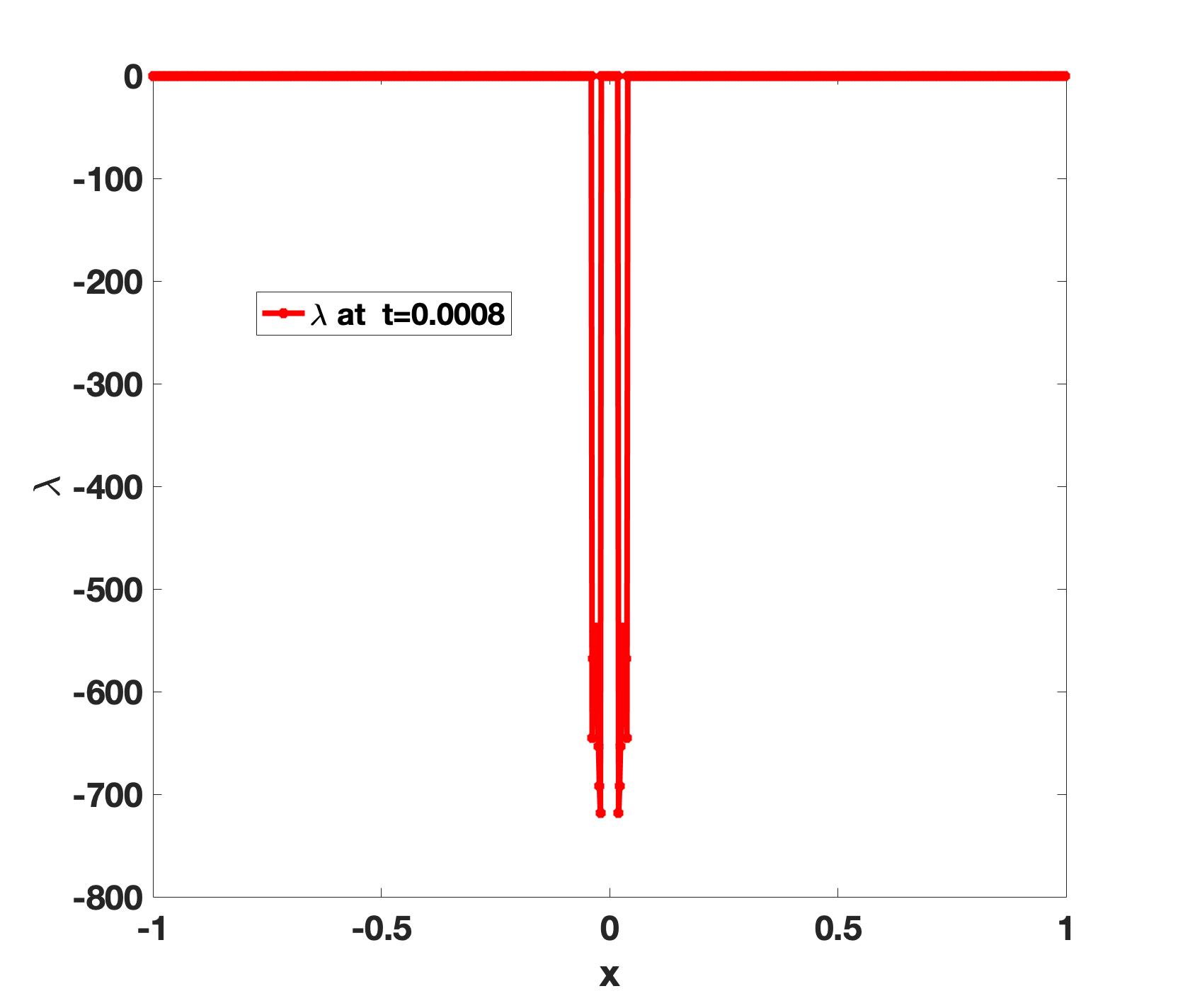}}
\caption{(a)-(d): numerical solution $ u$ of positivity preserving scheme \eqref{scheme:Lu:eq:LM:1d} for lubrication-type equation \eqref{Lu:eq} in 1D at various time  with different  time steps and $\eps$. (e)-(f): Lagrange multiplier $\lambda$ at time $t=0.001$ and $t=0.0008$ using different time steps and $\eps$.}\label{Lu:1d:numeric}
\end{figure}

Next, we consider an  2D example with the initial condition 
\begin{equation}\label{initial:Lu:2d}
 \begin{split}
&   u(x,y)=\begin{cases} (x-0.5)^2(y-0.5)^2,& x^2+y^2 \leq 0.25,\\
  0,& \mbox{otherwise}\end{cases},
  \end{split}
  \end{equation}
in the domain $[-\pi,\pi)^2$. 

We first take  $f( u)= u$ and use the following usual semi-implicit scheme:
\begin{equation}
 \frac{3{ u}_h^{n+1}(\bz)-4 u_h^n(\bz)+ u_h^{n-1}(\bz)}{2\delta t}+ \mathcal{L}_h^n { u}_h^{n+1} (\bz)=0,\quad\forall \bz\in \Sigma_h.
\end{equation}
 The scheme failed to  converge with $\delta t=10^{-5}$.   However, by using the 2D version of the scheme \eqref{scheme:Lu:eq:LM:1d} with $\eps=0$ and  $128\times 128 $ Fourier modes, correct results can be obtained  with $\delta t=10^{-5}$. In  
 Fig.\;\ref{numerical:Lu:2D}(a-d), we  plot the initial condition and numerical solutions  at $t=0.001, 0.01, 0.1$, while we plot  in  
 Fig.\;\ref{numerical:Lu:2D}(e-f) the Lagrange multipliers $\lambda_h$ at $t=0.001,  0.1$. We observe that the Lagrange multiplier takes   nonzero values at a significant part of the domain which explains why the usual  semi-implicit scheme failed to converge.

\begin{figure}[htbp]
\centering
\subfigure[$ u_h:t=0$]{
\includegraphics[width=0.45\textwidth,clip==]{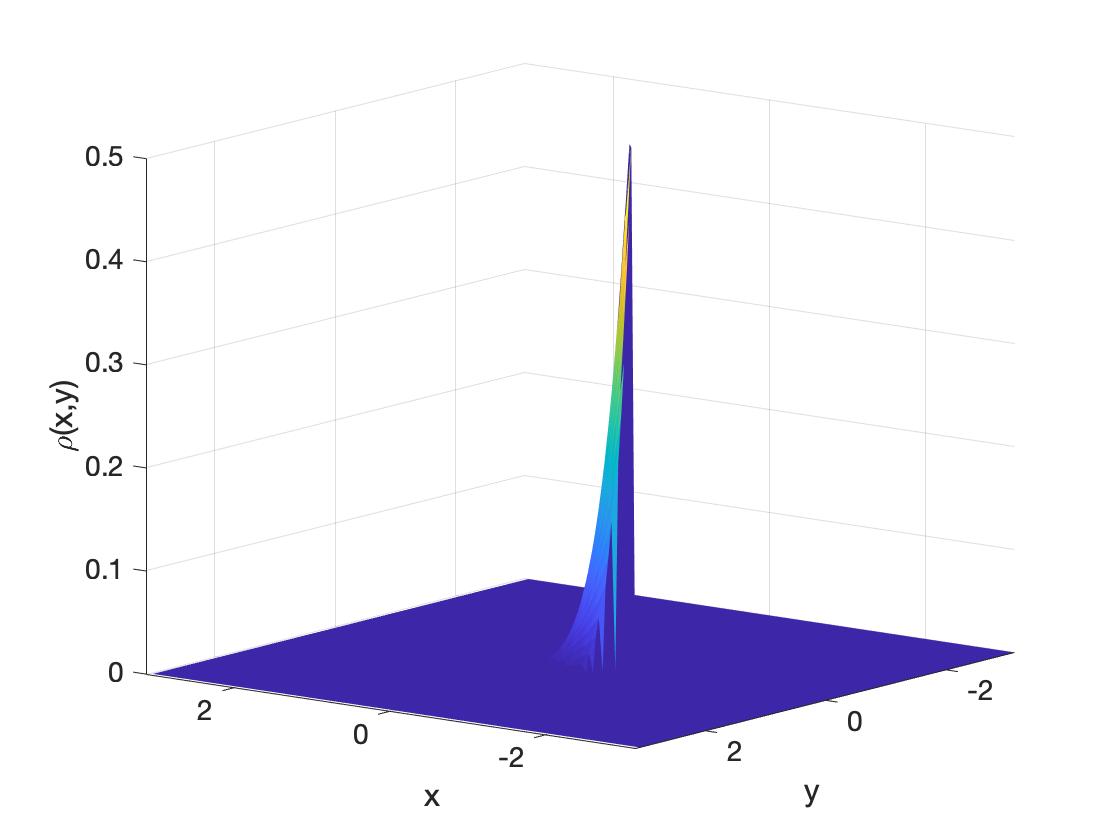}}
\subfigure[$ u_h:t:=0.001$]{
\includegraphics[width=0.45\textwidth,clip==]{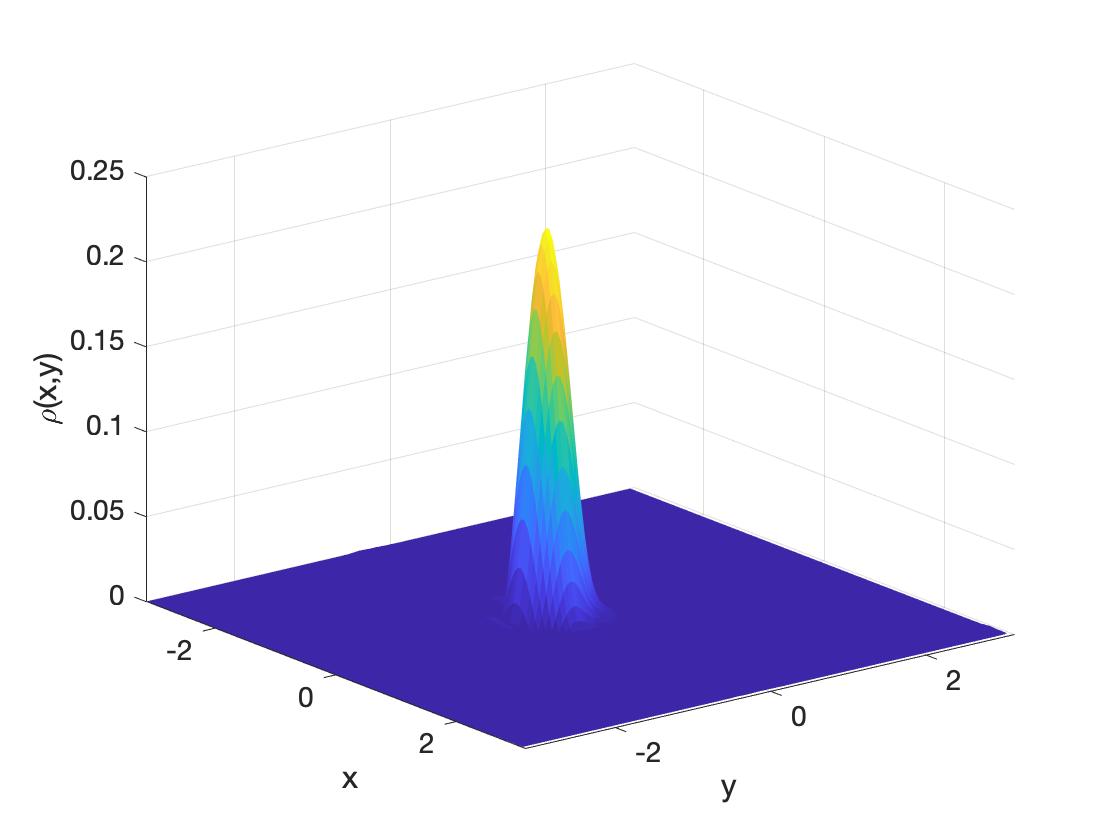}}
\subfigure[$ u_h:t:=0.01$]{
\includegraphics[width=0.45\textwidth,clip==]{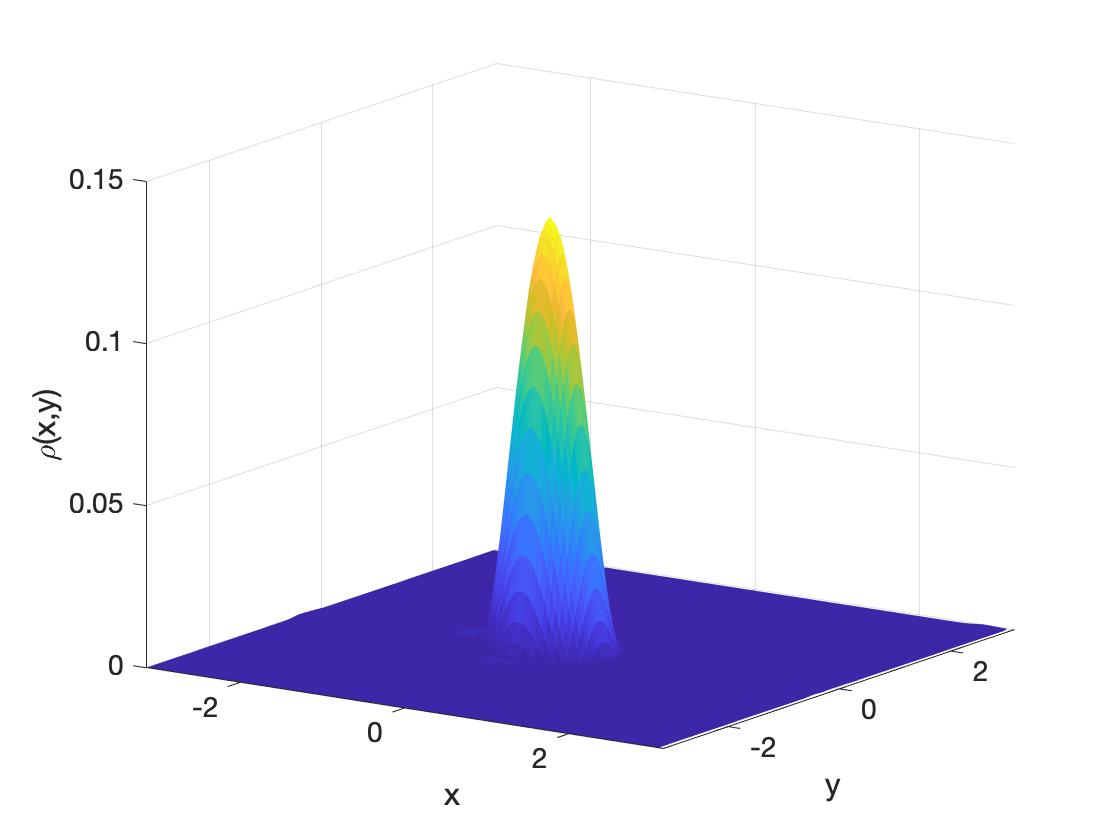}}
\subfigure[$ u_h:t:=0.1$]{
\includegraphics[width=0.45\textwidth,clip==]{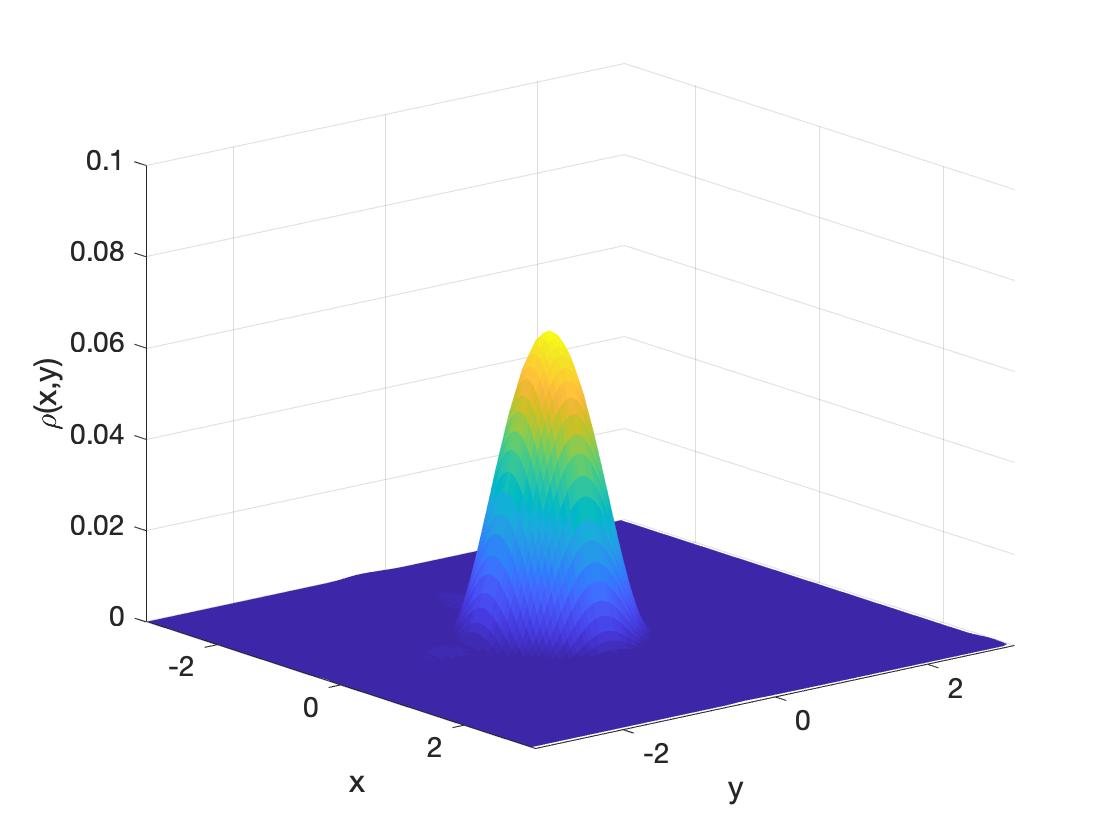}}
\subfigure[$\lambda_h:t:=0.01$]{
\includegraphics[width=0.45\textwidth,clip==]{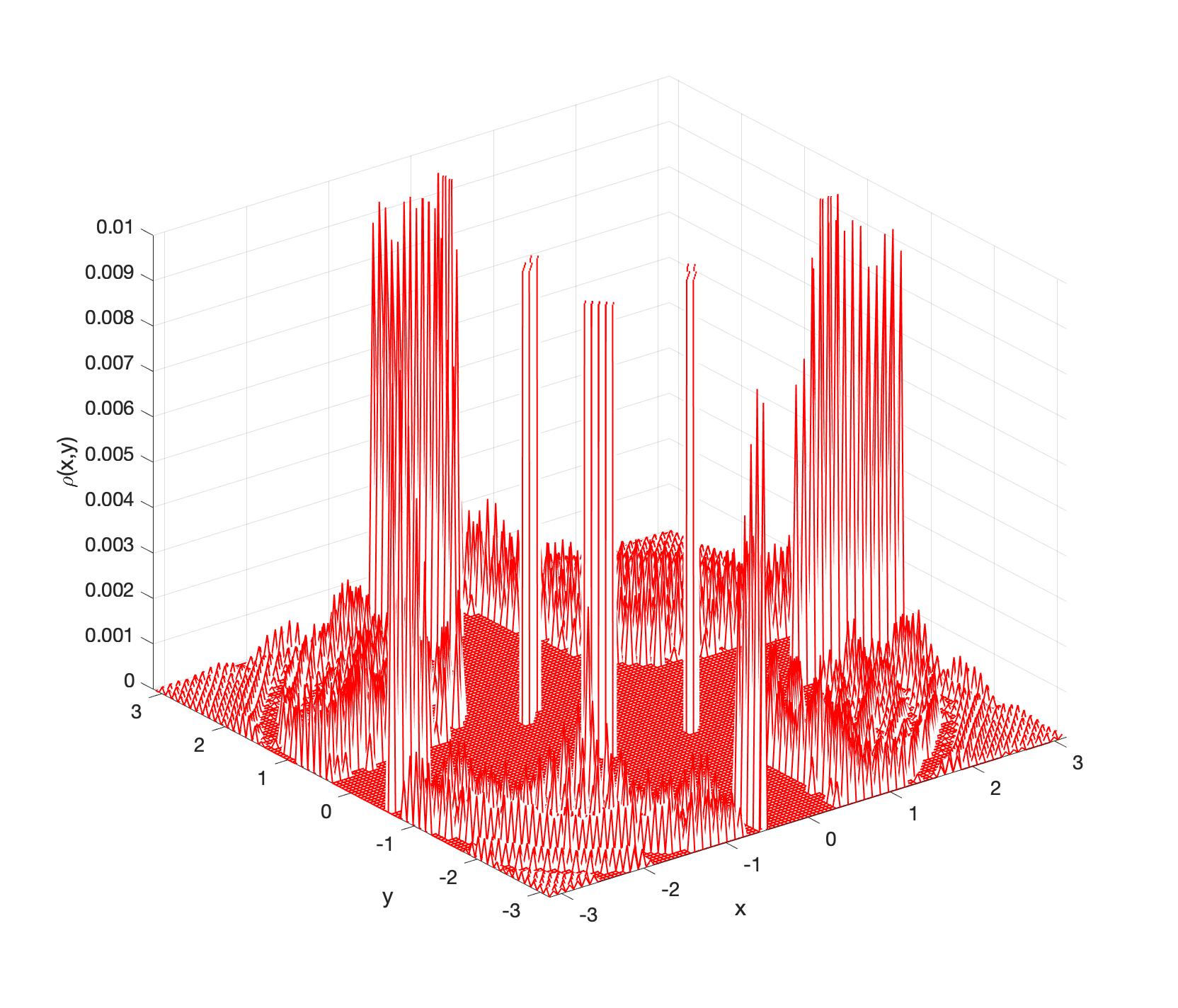}}
\subfigure[$\lambda_h:t:=0.1$]{
\includegraphics[width=0.45\textwidth,clip==]{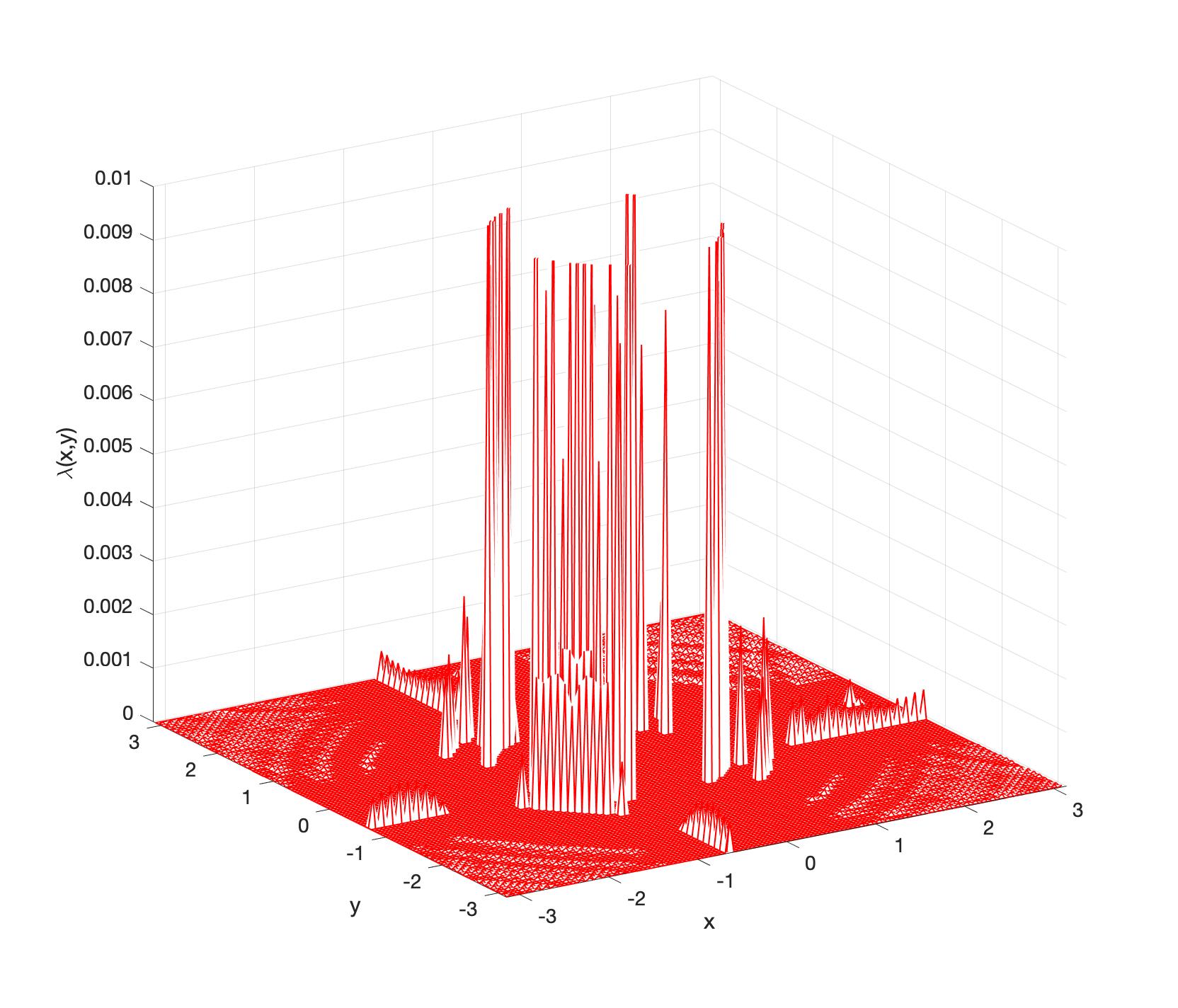}}
\caption{Numerical solutions $u_h$ with positivity preserving scheme  at $t=0, 0.001, 0.01, 0.1$, and  Lagrange multiplier $\lambda_h$ at $t=0.01, 0.1$. }\label{numerical:Lu:2D}
\end{figure}

\section{Concluding remarks}
If a PDE requires its solution to be positive, a generic numerical scheme for the PDE usually can not preserve the positivity. 
We presented in this paper a new  approach to construct positivity preserving schemes for parabolic type equations by a simple modification to  generic numerical schemes.  
More precisely, we introduce a space-time Lagrange multiplier function to enforce the positivity, and expand the underlying PDE using the KKT conditions. The key question  is how to solve the expanded system  efficiently with 
essentially the same cost as the generic numerical scheme.

We  constructed a new class of  positivity preserving schemes by using the predictor-corrector approach to the expanded system: the prediction step can be a generic semi-implicit or implicit scheme, while the correction step 
is used to enforce the positivity and can be implemented as a simple pointwise update with negligible cost. This new approach is not restricted to any particular spatial discretization and can be combined with various time discretization schemes. It can be 
applied to  a large class of parabolic PDEs which require solutions to be positive. It is also non intrusive as you can easily modify your non-positivity preserving schemes for them to become positivity preserving. 
In addition, we also presented a modification to the above approach so that the schemes can also preserve mass  
if the underlying PDE is  mass conserving. 

An interesting and useful observation is that the ad-hoc cut-off approach can be interpreted as a special case of our predictor-corrector approach. Hence, it provides a different justification for the  cut-off approach, moreover  allows us to modify the cut-off approach so that it becomes mass conserving, and opens new avenue for further exploration.

We established  stability results for the first- and second-order schemes based on the new approach under a general setting, and  presented ample numerical experiments to  validate the new approach. Our numerical results indicate that the new approach is very effective for the variety of  problems that we tested.

\bibliographystyle{siamplain}
 \bibliographystyle{plain}
\bibliography{references}
\end{document}